\documentclass[a4paper, 11pt]{amsart} 
\usepackage[headings]{fullpage}               
\usepackage{boxedminipage}
\usepackage{amsfonts}
\usepackage{amsmath} 
\usepackage{amssymb}
\usepackage{graphicx}
\usepackage{amsthm}
\usepackage{subfig}
\usepackage{tikz}
\usepackage{setspace}
\usepackage{hyperref}
\hypersetup{
  colorlinks   = true, 
  urlcolor     = blue, 
  linkcolor    = blue, 
  citecolor   = red 
}
\usepackage{enumerate}
\usepackage{xcolor}
\hypersetup{linktocpage=true,}
\newtheorem{theorem}{Theorem}[section]
\newtheorem{lemma}[theorem]{Lemma}
\newtheorem{conjecture}[theorem]{Conjecture}
\newtheorem{proposition}[theorem]{Proposition}

\theoremstyle{definition}\newtheorem{definition}{Definition}[section]
\theoremstyle{definition}

\theoremstyle{remark}\newtheorem{remark}{Remark}[section]

\theoremstyle{remark}

\usepackage{etoolbox}
\let\bbordermatrix\bordermatrix
\patchcmd{\bbordermatrix}{8.75}{4.75}{}{}
\patchcmd{\bbordermatrix}{\left(}{\left[}{}{}
\patchcmd{\bbordermatrix}{\right)}{\right]}{}{}

\usepackage{kbordermatrix} 
\renewcommand{\kbldelim}{[} 
\renewcommand{\kbrdelim}{]}

\newcommand{\rank}{\operatorname{rank}}

\newcommand{\spann}{\operatorname{span}}

\newcommand{\Vset}{\mathcal{V}_{\tilde{e}}(G,p,L)}
\newcommand{\Vsetcirc}{\mathcal{V}^f_{\tilde{e}}(G,p,L)}

\definecolor{colR}{rgb}{.932,.172,.172}
\definecolor{colB}{rgb}{.255,.41,.884}
\definecolor{colG}{rgb}{0,0.7,0}

\tikzstyle{vertex}=[circle, draw, fill=black, inner sep=0pt, minimum size=4pt]
\tikzstyle{smallvertex}=[circle, line width=1.5pt, draw, fill=black, inner sep=0pt, minimum size=2pt]

\tikzstyle{edge}=[line width=1.5pt]
\tikzstyle{redge}=[edge,colR]
\tikzstyle{bedge}=[edge,colB]
\tikzstyle{gedge}=[edge,colG]
\tikzstyle{lnode}=[circle,white,draw, fill=black,inner sep=1pt, font=\scriptsize]

\begin{document}

\title{Flexible placements of periodic graphs in the plane}

\author{Sean Dewar}\thanks{Supported by the Austrian Science Fund (FWF): P31888}
\address{Johann Radon Institute for Computational and Applied Mathematics (RICAM), Austrian Academy of Sciences, 4040 Linz, Austria}


\begin{abstract}
	Given a periodic graph,
	we wish to determine via combinatorial methods whether it has periodic embeddings in the plane 
	that -- via motions that preserve edge-lengths and periodicity -- can be continuously deformed into another non-congruent embedding of the graph.
	By introducing NBAC-colourings for the corresponding quotient gain graphs,
	we identify which periodic graphs have flexible embeddings in the plane when the lattice of periodicity is fixed.
	We further characterise with NBAC-colourings which 1-periodic graphs have flexible embeddings in the plane
	with a flexible lattice of periodicity,
	and characterise in special cases which 2-periodic graphs have flexible embeddings in the plane
	with a flexible lattice of periodicity.
\end{abstract}

\maketitle
\tableofcontents

\section{Introduction}

A \emph{(bar-joint) framework in the plane} is a pair $(\mathcal{G}, \mathcal{P})$,
where $\mathcal{G}$ is a simple graph and $\mathcal{P}$ (the \emph{placement} of $\mathcal{G}$)
is a map from $V(\mathcal{G})$ to $\mathbb{R}^2$.\footnote{Although $(G,p)$ is the standard notation for a framework, we shall instead reserve this instead for the quotient frameworks that we use throughout the majority of the this paper.}
By considering each edge $vw$ as a rigid bar that restricts the distance between $v$ and $w$,
a natural question to ask is whether or not the structure is \emph{flexible},
i.e.~does there exist a continuous path in the space of placements of $\mathcal{G}$
that preserves the edge distances but is not a rigid body motion?
If the vertex set of $\mathcal{G}$ is finite 
and the coordinates of the vector $( \mathcal{P}(v))_{v \in V(\mathcal{G}) }$ are algebraically independent over $\mathbb{Q}$,
then it has been proven
(first by Pollaczek-Geiringer \cite{Pollaczek} and later by Laman independently \cite{laman})
that $(\mathcal{G}, \mathcal{P})$ is \emph{rigid} (i.e.~not flexible) 
in the plane if and only if $\mathcal{G}$ contains a (somewhat erroneously named) \emph{Laman graph};
a graph $\mathcal{H}$ where $|E(\mathcal{H})| = 2|V(\mathcal{H})| -3$ 
and $|E(\mathcal{H}')| \leq 2|V(\mathcal{H}')| -3$ for all subgraphs $\mathcal{H}'$ of $\mathcal{H}$ with $|V(\mathcal{H}')| \geq 2$.
Given a graph that contains a Laman graph,
there can however still exist non-generic placements that are flexible; see Figure \ref{fig:intro1}.

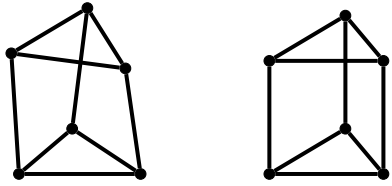
\begin{figure}[ht]
	\begin{center}
		\begin{tikzpicture}
			\node[vertex] (1) at (0,0) {};
			\node[vertex] (2) at (1.6,0) {};
			\node[vertex] (3) at (0.7,0.6) {};
			\node[vertex] (4) at (-0.1,1.6) {};
			\node[vertex] (5) at (1.4,1.4) {};
			\node[vertex] (6) at (0.9,2.2) {};
	
			\draw[edge] (1)edge(2);
			\draw[edge] (2)edge(3);
			\draw[edge] (1)edge(3);
	
			\draw[edge] (4)edge(5);
			\draw[edge] (5)edge(6);
			\draw[edge] (4)edge(6);
			
			\draw[edge] (1)edge(4);
			\draw[edge] (2)edge(5);
			\draw[edge] (3)edge(6);
		\end{tikzpicture}\qquad\qquad
		\begin{tikzpicture}
			\node[vertex] (1) at (0,0) {};
			\node[vertex] (2) at (1.5,0) {};
			\node[vertex] (3) at (1,0.6) {};
			\node[vertex] (4) at (0,1.5) {};
			\node[vertex] (5) at (1.5,1.5) {};
			\node[vertex] (6) at (1,2.1) {};
	
			\draw[edge] (1)edge(2);
			\draw[edge] (2)edge(3);
			\draw[edge] (1)edge(3);
	
			\draw[edge] (4)edge(5);
			\draw[edge] (5)edge(6);
			\draw[edge] (4)edge(6);
			
			\draw[edge] (1)edge(4);
			\draw[edge] (2)edge(5);
			\draw[edge] (3)edge(6);
		\end{tikzpicture}
	\end{center}
	\caption{(Left): A rigid placement of $K_2 \times K_3$ in the plane. 
	As $K_2 \times K_3$ is a Laman graph, almost all placements will give a rigid framework.
	(Right): A flexible placement of the same graph.}
	\label{fig:intro1}
\end{figure}

This raises a new question;
can we use combinatorial methods to determine if a graph $\mathcal{G}$ has \emph{any} placement that defines a flexible 
framework $(\mathcal{G},\mathcal{P})$?
This question was answered in the positive in \cite{jan2},
where it was proved that a finite simple graph will have flexible placements in the plane if and only if it has a \emph{NAC-colouring},
a surjective red-blue edge colouring where no cycle has exactly one red edge or exactly one blue edge.
Detecting whether graphs have flexible placements via NAC-colourings is a very recent area of research
that utilises many different areas of algebraic geometry,
including valuation theory \cite{DewGrasLeg,paradox,jan1,jan2}.

We now wish to extend the method using NAC-colourings to a frameworks in the plane with \emph{$k$-periodic symmetry},
i.e.~frameworks $(\mathcal{G},\mathcal{P})$ where there exists a matrix $L \in M_{2 \times k} (\mathbb{R})$
and a free group action $\theta$ of $\mathbb{Z}^k$ on $\mathcal{G}$ via graph automorphisms,
such that 
$\mathcal{G}$ has a finite set of vertex orbits under $\theta$ and
$\mathcal{P}(\theta(\gamma)v) = \mathcal{P}(v) +L. \gamma$
for all $v \in V(G)$ and $\gamma \in \mathbb{Z}^k$;
we call $L$ the \emph{lattice} of $\mathcal{P}$, 
$\theta$ the \emph{symmetry} of $\mathcal{G}$,
and $\mathcal{P}$ a \emph{$k$-periodic placement} of $(\mathcal{G}, \theta)$.
Specifically,
we wish to be able to determine if a graph $\mathcal{G}$ with symmetry $\theta$ has a 
$k$-periodic placement $\mathcal{P}$
where $(\mathcal{G},\mathcal{P})$ can be deformed by a motion 
that preserves the periodic structure of $(\mathcal{G},\mathcal{P})$,
and if such a placement does exist,
be able to also determine in advance whether the motion will preserve the lattice structure of $(\mathcal{G},\mathcal{P})$.

Research into the rigidity of periodic frameworks has seen much interest in the last decade.
Some of the main areas of research include 
combinatorial characterisations of rigid periodic graphs \cite{periodic,BorStr2011,louisperiodic,nixonross,ross},
periodic graphs with unique realisations \cite{globalperiodic},
rigid unit modes of periodic frameworks \cite{rum2,rum1},
and rigidity under infinitesimal motions where the periodicity is relaxed somewhat \cite{badri,BorStr2015,leftypower,ultrarigid,whiteley14}.

Each $k$-periodic framework $(\mathcal{G},\mathcal{P})$ in the plane with a given symmetry $\theta$
defines a family of \emph{gain-equivalent} triples $(G,p,L)$,
where $G$ is a \emph{$\mathbb{Z}^k$-gain graph} and $p :V(G) \rightarrow \mathbb{R}^2$ is a \emph{placement} of $G$ 
(see Sections \ref{sec:gaingraphs} and \ref{sec:rigidgraphs} for definitions),
and likewise, each such triple $(G,p,L)$ will define a framework $(\mathcal{G},\mathcal{P})$ with $k$-periodic symmetry;
see \cite[Section 2.2]{ross} for more details.
As $\mathbb{Z}^k$-gain graphs have a finite amount of vertices but still encode all the required information needed 
for working with motions that preserve periodicity,
we shall define a \emph{$k$-periodic framework in the plane} to be a triple $(G,p,L)$
for some $\mathbb{Z}^k$-gain graph $G$,
and the pair $(p,L)$ to be a \emph{placement-lattice} of $G$.

\begin{figure}[ht]
	\begin{center}	
		\begin{tikzpicture}
			\node[vertex] (1) at (0,0) {};
			\node[vertex] (2) at (1,1) {};
			\node[vertex] (1r) at (2,0) {};
			\node[vertex] (2r) at (3,1) {};
			\node[vertex] (1u) at (0,2) {};
			\node[vertex] (2u) at (1,3) {};
			\node[vertex] (1ur) at (2,2) {};
			\node[vertex] (2ur) at (3,3) {};
					
			\draw[bedge] (1)edge(1u);
			\draw[bedge] (1r)edge(1ur);
			
			\draw[bedge] (1u) -- (0,3.3);
			\draw[bedge] (1ur) -- (2,3.3);
			
			\draw[bedge] (1) -- (0,-0.3);
			\draw[bedge] (1r) -- (2,-0.3);
			\draw[gedge] (1)edge(2);
			\draw[gedge] (1r)edge(2r);
			\draw[gedge] (1u)edge(2u);
			\draw[gedge] (1ur)edge(2ur);
			
			\draw[redge] (2) -- (-0.5,0.5);
			\draw[redge] (2u) -- (-0.5,2.5);
			
			\draw[redge] (1r) -- (3.5,0.5);
			\draw[redge] (1ur) -- (3.5,2.5);
			
			\draw[redge] (1)edge(2r);
			\draw[redge] (1u)edge(2ur);

		\end{tikzpicture}\qquad\qquad
		\begin{tikzpicture}
			\node[vertex] (1) at (0,0) {};
			\node[vertex] (2) at (1,1) {};
	
			\draw[gedge] (1)edge(2);
			
			\path[redge,-latex] (1) edge [bend right] node {} (2);
			
			\node[font=\scriptsize] at (-1,0) {$(0,1)$};
			\node[font=\scriptsize] at (1.2,0.25) {$(1,0)$};
				
			\path[bedge,-latex, every loop/.style={looseness=30}] (1) edge [loop left] node { } (1);
			
			\node (3) at (-1,-1) {};
			\end{tikzpicture}
		\end{center}
	\caption{(Left): A framework $(\mathcal{G},\mathcal{P})$ with $2$-periodic symmetry. 
	(Right): A corresponding triple $(G,p,L)$ with $L := 2 I_2$,
	where $I_2$ is the $2 \times 2$ identity matrix.}
	\label{fig:intro2}
\end{figure}
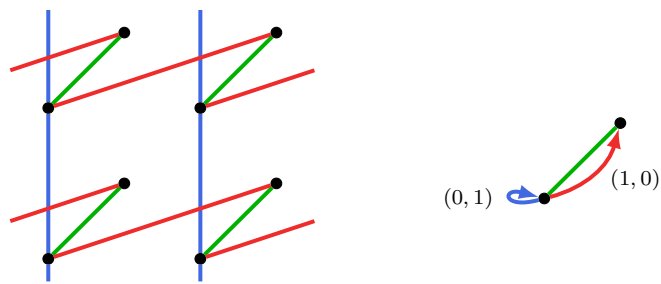

Using the gain graph description of $k$-periodic frameworks,
our question is now the following;
can we use combinatorial methods to determine if a $\mathbb{Z}^k$-gain graph $G$ has \emph{any} 
placement-lattice that defines a flexible $k$-periodic framework $(G,p,L)$?
We shall answer this in the positive for $1$-periodic frameworks where the lattice is allowed to deform 
(see Theorem \ref{thm:1NBAC})
and $k$-periodic frameworks where the lattice is fixed (see Theorem \ref{thm:fixedNBAC}).
We also obtain partial results for the more difficult case of $2$-periodic frameworks where the lattice is allowed to deform
(see Lemma \ref{lem:2NBAC}, Theorem \ref{thm:1loop2per} and Theorem \ref{cor:loopNBAC}).
To do this we shall introduce \emph{NBAC-colourings} (``NBAC'' being an acronym for ``No Balanced Almost Circuit''),
an analogue of NAC-colourings for $\mathbb{Z}^k$-gain graphs.
We shall also characterise the various types of NBAC-colourings that are generated by different motions of a given $k$-periodic framework.

The outline of the paper is as follows.
In Section \ref{sec:Preliminaries},
we shall layout some background on valuation theory,
gain graphs and periodic frameworks in both $\mathbb{R}^d$ and $\mathbb{C}^d$.
In Section \ref{sec:NBAC}, 
we shall define NBAC-colourings and their various sub-types,
including active NBAC-colourings,
and utilise valuations to prove that flexibility will imply the existence of an NBAC-colouring.
In Section \ref{sec:fixedlattice}, Section \ref{sec:1NBAC} and Section \ref{sec:2NBAC},
we shall apply our methods using NBAC-colourings to fixed lattice $k$-periodic frameworks,
flexible lattice $1$-periodic frameworks and flexible lattice $2$-periodic frameworks respectively,
with partial results in the latter case.
In Section \ref{sec:loops},
we shall prove that a full characterisation of $\mathbb{Z}^2$-gain graphs with a flexible placement-lattice is possible
if we assume that our graph has at least a single loop.

\section{Preliminaries}\label{sec:Preliminaries}

\subsection{Function fields and valuations}

We shall refer to all affine algebraic sets over $\mathbb{C}$ as algebraic sets,
and we shall call any irreducible algebraic set a variety.
For an algebraic set $V$ in $\mathbb{C}^n$,
we define $I(V)$ to be the ideal of $\mathbb{C}^n$ that defines $V$.
We recall that the dimension of an algebraic set is the maximal length of chains of distinct nonempty subvarieties of $A$.
An \emph{algebraic curve} is an affine variety of dimension $1$.

\begin{definition}
	Let $V$ be a variety in the polynomial ring $\mathbb{C}[X_1, \ldots,X_n]$.
	We define the \emph{coordinate ring} of $V$ to be the quotient $\mathbb{C}[V] :=  \mathbb{C}[X_1, \ldots,X_n]/ I(V)$
	and the \emph{function field} of $V$ to be the field of fractions of $\mathbb{C}[V]$, denoted by $\mathbb{C}(V)$.
	Each $\hat{f}/\hat{g} \in \mathbb{C}(V)$ can, 
	for any $f \in \hat{f}$ and $g \in \hat{g}$, be considered to be a partially defined function
	\begin{align*}
		f/g : V \rightarrow \mathbb{C}, ~ x \mapsto f(x)/g(x),
	\end{align*}
	and this function is independent of the choice of $f,g$.
\end{definition}

We recall that for a field extension $K/k$,
an element $a \in K$ is \emph{transcendental over $k$} if there is no polynomial $p \in k[X]$ with $p(a) = 0$,
and \emph{algebraic over $k$} otherwise.
The following useful result stems from the observation that any rational function must either be constant on a variety or take an infinite amount of values;
indeed if this was not true,
we would be able to construct a non-invertible element of the function field.

\begin{lemma}\label{lem:constantW2}
	Let $\mathcal{C}$ be an algebraic curve in $\mathbb{C}[X_1, \ldots,X_n]$ and let $f \in \mathbb{C}[x_1, \ldots,x_n]$.
	Then one of the following holds:
	\begin{enumerate}[(i)]
		\item
		$f$ takes an infinite amount of values on $\mathcal{C}$ and is transcendental over $\mathbb{C}$ 
		when considered as an element of $\mathbb{C}(\mathcal{C})$.
		\item
		$f$ is constant on $\mathcal{C}$.
	\end{enumerate}
\end{lemma}
%
	%

\begin{definition}
	For a function field $\mathbb{C}(C)$,
	a function $\nu : \mathbb{C}(C) \rightarrow \mathbb{Z} \cup \{\infty \}$ is a \emph{valuation} if 
	\begin{enumerate}[(i)]
		\item $\nu (x) = \infty$ if and only if $x = 0$.
		
		\item $\nu (xy) = \nu (x) + \nu (y)$.
		
		\item $\nu (x + y) \geq \min \{ \nu (x) , \nu (y) \}$, with equality if $\nu (x) \neq \nu (y)$.
		
		\item $\nu (x) = 0$ if $x \in \mathbb{C} \setminus \{0\}$.
	\end{enumerate}
\end{definition}

The following is a useful rewording of \cite[Corollary 1.1.20]{functionfields}.

\begin{proposition}\label{prop:valuation}
	Let $\mathbb{C}(C)$ be a function field and suppose $f \in \mathbb{C}(\mathcal{C})$ is transcendental over $\mathbb{C}$.
	Then there exists a valuation $\nu$ of $\mathbb{C}(\mathcal{C})$ with $\nu (f) >0$.
\end{proposition}

\subsection{Gain graphs}\label{sec:gaingraphs}

We shall briefly cover the topic of \emph{gain graphs}.
For a more in depth analysis of the topic for general groups,
we refer the reader to \cite{grosstucker}.
We will be mainly be interested in the case when the group is an abelian free group;
for more discussion on techniques often used for this specific topic,
we refer the reader to \cite{ross14}.

\begin{definition}
	A \emph{$\Gamma$-gain graph} is a triple $G := (V(G),E(G), \Gamma)$, where:
	\begin{enumerate}[(i)]
		\item $V(G)$ is a finite set of \emph{vertices}.
	
		\item $\Gamma$ is an additive abelian group with identity $0$.
	
		\item $K(V(G)) := (V(G)^2 \times \Gamma)/ R$,
		where $R$ is the equivalence relation with $(a,b,\gamma) R (c,d,\mu)$ if and only if either $a=c$, $b=d$ and $\gamma = \mu$,
		or $a=d$, $b = c$ and $\gamma = -\mu$.
		
		\item	$E(G) \subset K(V(G))$ is a set of \emph{edges}.
		We shall assume that there is no edge of the form $(v,v , 0)$;
		we shall, however, allow $E(G)$ to be an infinite set.
	\end{enumerate}
\end{definition}

While the edges of a gain graph are not orientated,
we often find it easier to assume that there is some orientation on the edges, i.e.~$G$ is directed.
We may then define the \emph{gain} of an edge $(v,w,\gamma)$ to be $\gamma$.
We refer the reader to Figure \ref{fig:gain} for an example.

\begin{figure}[ht]
	\begin{center}		
		\begin{tikzpicture}
			\node[vertex] (1) at (-1,0) {};
			\node[vertex] (2) at (1,0) {};
			\node[vertex] (3) at (0,1.5) {};
			\node[vertex] (4) at (1,-1) {};
			\node[vertex] (5) at (-1,-1) {};

			\path[edge,-latex] (1) edge [bend left] node {} (2);
			\path[edge] (1) edge [bend right] node {} (2);
			
			\draw[edge] (1)edge(4);
			\draw[edge] (2)edge(4);
			
			\draw[edge,-latex] (1)edge(3);
			\draw[edge] (2)edge(3);
			\draw[edge] (5)edge(1);
			\draw[edge] (5)edge(4);

			\node[font=\scriptsize] at (0,0.5) {$a$};
			\node[font=\scriptsize] at (1.7,0) {$b$};
			\node[font=\scriptsize] at (-0.75,0.7) {$c$};
			\node[font=\scriptsize] at (0.35,1.8) {$d$};
		
			\path[edge,-latex, every loop/.style={looseness=30}] (2) edge [loop right] node { } (2);
			\path[edge,-latex, every loop/.style={looseness=30}] (3) edge [loop above] node { } (3);

		\end{tikzpicture}
	\end{center}
	\caption{A $\Gamma$-gain graph with $a,b,c,d \in \Gamma$.
	We represent any edge $(v,w,\gamma)$ by an arrow from $v$ to $w$ with a label $\gamma$,
	and we represent any edge $(v,w,0)$ by an undirected and unlabelled edge from $v$ to $w$.}
	\label{fig:gain}
\end{figure}
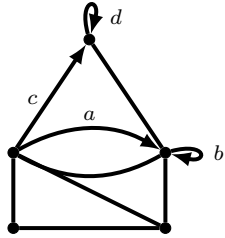
%

A \emph{switching operation at $u$ by $\mu$} is the map $\phi_u^\mu : K(V(G)) \rightarrow K(V(G))$ where
\begin{align*}
	\phi_u^\mu (v,w,\gamma) = 
	\begin{cases}
		(u,w, \gamma + \mu) & \text{if } v=u,~ w \neq u \\
		(v,u, \gamma - \mu) & \text{if } v \neq u,~ w = u \\
		(v,w,\gamma) & \text{if } v, w \neq u \text{ or } v=w=u.
	\end{cases} 
\end{align*}
See Figure \ref{fig:gainswitching} for an example of a gain switching operation at a vertex.

\begin{figure}[ht]
	\begin{center}
		\begin{tikzpicture}
			\node[lnode] (1) at (0,0) {$\boldsymbol{u}$};
			\node[vertex] (2) at (0.7,0.7) {};
			\node[vertex] (3) at (0,-1) {};
			\node[vertex] (4) at (-0.7,0.7) {};
			\draw[edge, -latex] (1)edge(2) (4)edge(1);
			\draw[edge] (1)edge(3);
			
			\node[font=\scriptsize] at (0.6,0.3) {$\beta$};
			\node[font=\scriptsize] at (-0.6,0.3) {$\alpha$};
			
			\draw[-latex] (1.7,-0.5) -- (2.3,-0.5);
			
			\node[lnode] (6) at (4,0) {$\boldsymbol{u}$};
			\node[vertex] (7) at (4.7,0.7) {};
			\node[vertex] (9) at (3.3,0.7) {};
			\node[vertex] (10) at (4,-1) {};
			\draw[edge, -latex] (9)edge(6);
			\draw[edge, -latex] (6)edge(7) (6)edge(10);

			\node[font=\scriptsize] at (5,0.3) {$\beta +\mu$};
			\node[font=\scriptsize] at (3.1,0.3) {$\alpha-\mu$};
			\node[font=\scriptsize] at (4.3,-0.5) {$\mu$};
		\end{tikzpicture}	
	\end{center}
	  \caption{A switching operation at $u$ by $\mu$.}
	  \label{fig:gainswitching}
\end{figure}
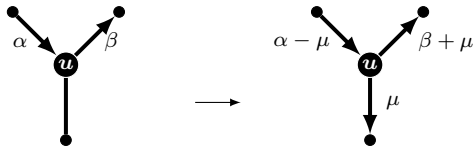

Given the switching operations $\phi_{u_1}^{\mu_1} , \ldots , \phi_{u_n}^{\mu_n}$ (where the vertices $u_1,\ldots,u_n$ and elements $\mu_1,\ldots,\mu_n$ need not be distinct),
we define $\phi := \phi_{u_n}^{\mu_n} \circ \ldots \circ \phi_{u_1}^{\mu_1}$ to be a \emph{gain equivalence}.
We say $\Gamma$-gain graphs $G,G'$ are \emph{gain-equivalent} (or $G \approx G'$) if $G$ and $G'$ are $\Gamma$-gain graphs with the same vertex set and $G' = \phi(G) := (V(G),\phi(E(G)),\Gamma)$ for some gain equivalence $\phi$.
If $H \subset G$ and $H' := \phi(H)$,
then we say $H'$ is the \emph{corresponding subgraph} of $H$ in $G'$.
The relation $\approx$ is an equivalence relation for gain graphs.
%

A \emph{walk} in $G$ is an ordered set $C:= (e_1, \ldots, e_n)$ of edges of $G$ where $e_i = (v_i, v_{i+1}, \gamma_i)$ (with $v_{n+1}=v_1$) for some $\gamma_i$;
we note that we orientate each edge so we have a \emph{directed walk} from $v_1$ to $v_n$.
The \emph{length} of a walk is the amount of edges it contains (including any repetitions).
If $v_1 = v_n$ then $C$ is a \emph{circuit}.
Unless specified otherwise, all walks and circuits of length $n$ will be of the form described above.

For a circuit $C$, we define
\begin{align*}
	\psi(C) := \gamma_1 + \gamma_2 +  \ldots + \gamma_n
\end{align*}
to be the \emph{gain} of $C$.
A circuit is \emph{balanced} if $\psi(C) =0$, and \emph{unbalanced} otherwise.
For a connected subgraph $H \subset G$,
we define the \emph{span of $H$} to be the subgroup
\begin{align*}
	\spann (H) := \{ \psi (C) : C \text{ is a circuit in } H \}.
\end{align*} 
If $\Gamma \cong \mathbb{Z}^k$ for some $k \in \mathbb{N}$,
then we define $\rank (H)$ to be the rank of $\spann (H)$.
A connected subgraph $H$ is \emph{balanced} if $\spann (H)$ is the trivial group, and \emph{unbalanced} otherwise;
likewise, a subgraph is \emph{balanced} if every connected component is balanced and \emph{unbalanced} otherwise.

\begin{proposition}\label{prop:correspondingsub}
	Let $G,G'$ be gain-equivalent $\Gamma$-gain graphs and $H \subset G$ be a connected subgraph.
	If $H'$ is the corresponding subgraph of $H$ in $G$,
	then $\spann (H') = \spann (H) $.
\end{proposition}

\begin{proof}
	This follows from noting that switching operations will not change the span of a circuit.
\end{proof}

\begin{proposition}\label{prop:balancedsub}
	Let $G$ be a $\Gamma$-gain graph and $\{H_1, \ldots, H_n \}$ a set of connected subgraphs with pairwise disjoint vertex sets.
	Then there exists $G' \approx G$ such that for each $i \in \{1, \ldots, n \}$,
	all the edges of the corresponding subgraph $H'_i$ of $H_i$ in $G'$
	have gain in $\spann(H_i)$.
\end{proposition}

\begin{proof}
	Choose a spanning tree $T_i$ for each $i \in \{1, \ldots, n \}$.
	We note that we may choose $G' \approx G$ such that each corresponding subgraph $T'_i$ of $T_i$ in $G'$ has
	only trivial gain for its edges;
	see \cite[Section 2.4]{ross} for a description of the method.
	
	Fix $i \in \{1,\ldots,n\}$ and choose any $e = (v,w,\gamma) \in E(H'_i)$.
	Let $W$ be the unique walk from $w$ to $v$ in $T_i$,
	and define $C$ to be the circuit formed by the travelling along the edge $e$ and then following the walk $W$.
	As $\psi(C) = \gamma$,
	then $\gamma \in \spann(H'_i)$.
	By Proposition \ref{prop:correspondingsub},
	$\spann(H_i) = \spann(H_i')$,
	hence $\gamma \in \spann(H_i)$ as required.
\end{proof}

\subsection{Rigidity and flexibility for k-periodic frameworks}\label{sec:rigidgraphs}

Let $d \in \mathbb{N}$ and $\mathbb{K} := \mathbb{R}$ or $\mathbb{C}$.
We shall define $\| \cdot \|^2 : \mathbb{K}^d \rightarrow \mathbb{K}$ to be the quadratic form with 
\begin{align*}
	\left\|(x_i)_{i=1}^d \right\|^2 := \sum_{i=1}^d x_i^2
\end{align*}
for all $(x_i)_{i=1}^d \in \mathbb{K}^d$.
For $\mathbb{K} = \mathbb{R}$,
the quadratic form $\| \cdot\|^2$ is in fact the square of the Euclidean norm,
however this is not true for $\mathbb{K} = \mathbb{C}$.
The isometries of $(\mathbb{K}^d, \|\cdot\|^2)$ are exactly the affine maps $x \mapsto Mx +y$,
where $y \in \mathbb{K}^d$ and $M \in M_n(\mathbb{K})$ is a $d \times d$-matrix where $M^TM = I_d$.

\begin{remark}
	For any matrix $M \in M_{m \times n}(\mathbb{K})$ and $x:= (x_1, \ldots, x_n) \in \mathbb{K}^n$,
	we shall denote by $M.x$ the matrix multiplication $M [x_1 ~ \ldots ~ x_n]^T$.
\end{remark}

We shall be using the definition of periodic frameworks originally stated by Ross that utilises gain graphs (see \cite{ross14}),
although many of our results can be adapted to fit the terminology used by Borcea and Streinu (see \cite{periodic}).
These two differing definitions can be seen to be identical;
we refer the reader to \cite[Section 3.1]{ross14} for more details.

\begin{definition}
	Let $d \in \mathbb{N}$ and $G$ be a $\mathbb{Z}^k$-gain graph for some $1 \leq k \leq d$.
	A \emph{$k$-periodic framework in $\mathbb{K}^d$} is a triple $(G,p,L)$ such that
	$G$ is a $\mathbb{Z}^k$-gain graph,
	$p :V(G) \rightarrow \mathbb{K}^d$, and 
	$L \in M_{d \times k}(\mathbb{K})$,
	with the assumption that if $(v,w, \gamma) \in E(G)$ then $p(v) \neq p(w)+ L.\gamma$.
	We shall define $p$ to be a \emph{placement},
	$L$ to be a \emph{lattice},
	and the pair $(p,L)$ to be a \emph{placement-lattice}.
	If $L$ is also injective then $(G,p,L)$ is \emph{full},
	and if $\mathbb{K} =\mathbb{R}$ then we simply refer to $(G,p,L)$ as a $k$-periodic framework.
\end{definition}

For a given $\mathbb{Z}^k$-gain graph $G$,
we define $\mathcal{V}^d_{\mathbb{K}}(G)$ to be the space of placement-lattices of $G$,
which we shall consider to be a subspace of $\mathbb{K}^{d|V(G)| + dk}$.
We immediately note that $\mathcal{V}^d_{\mathbb{K}}(G)$ is an open non-empty subset in the Zariski topology,
and if $G$ has an edge,
it is a proper subset.

\begin{definition}
	Let $(G,p,L)$ and $(G,p',L')$ be $k$-periodic frameworks in $\mathbb{K}^d$. 
	Then $(G,p,L) \sim (G,p', L')$ (or $(G,p,L)$ and $(G,p', L')$ are \emph{equivalent}) if for all $(v,w, \gamma) \in E(G)$,
	\begin{align}\label{edgecond}
		\|p(v) - p(w) - L.\gamma \|^2 = \|p'(v) - p'(w) - L'.\gamma \|^2,
	\end{align}
	and $(p,L) \sim (p',L')$ (or $(p,L)$ and $(p',L')$ are \emph{congruent}) if (\ref{edgecond})
	holds for all $v, w \in V(G)$ and $\gamma \in \mathbb{Z}^k$;
	equivalently,
	we may define $(p,L) \sim (p',L')$ if and only if there exists a linear isometry $M \in M_d(\mathbb{K})$ and $y \in \mathbb{K}^d$
	such that $p'(v) = M.p(v) + y$ for all $v \in V(G)$ and $L' = ML$.
	For any $L, L' \in M_{d \times k}(\mathbb{K})$,
	we define $L$ and $L'$ to be \emph{orthogonally equivalent} (or $L \sim L'$) if for any $\gamma, \mu \in \mathbb{Z}^k$,
	\begin{align}\label{latticecond}
		(L.\gamma). (L.\mu) = (L'.\gamma). (L'.\mu).
	\end{align}
\end{definition}

We note that, by linearity, if (\ref{latticecond}) holds for all pairs of some basis of $\mathbb{Z}^k$,
then it holds for all $\gamma, \mu \in \mathbb{Z}^k$.
Furthermore, if $(p,L) \sim (p',L')$ then $(G,p,L) \sim (G,p',L')$ and $L \sim L'$.

\begin{definition}
	For a $k$-periodic framework $(G,p,L)$ we define the algebraic subsets
	\begin{eqnarray*}
	 	\mathcal{V}_{\mathbb{K}}(G,p,L) &:=& \left\{ (p',L') \in \mathcal{V}^d_{\mathbb{K}}(G) : (G,p',L') \sim (G,p,L) \right\}, \\
	 	\mathcal{V}^f_{\mathbb{K}}(G,p,L) &:=& \left\{ (p',L') \in \mathcal{V}^d_{\mathbb{K}}(G) : (G,p',L') \sim (G,p,L), ~ L' \sim L \right\}.
	\end{eqnarray*}
\end{definition}

\begin{definition}
	Let $(G,p,L)$ be a $k$-periodic framework in $\mathbb{K}^d$.
	A \emph{flex} of $(G,p,L)$ is a continuous path $t \mapsto (p_t, L_t)$, $t \in [0,1]$,
	in  $\mathcal{V}_{\mathbb{K}}(G,p,L)$.
	If $(p_t, L_t) \in \mathcal{V}^f_{\mathbb{K}}(G,p,L)$ for all $t \in [0,1]$ then $(p_t, L_t)$ is a \emph{fixed lattice flex}.
	If $(p_t,L_t) \sim (p,L)$ for all $t \in [0,1]$ then $(p_t,L_t)$ is \emph{trivial}.	
\end{definition}

\begin{remark}\label{rem:altdefntriv}
	An equivalent definition for a trivial finite flex is as follows:
	$(p_t,L_t)$ is a trivial flex of $(G,p,L)$ if and only $(p_t,L_t)$ is a trivial flex of $(K,p,L)$,
	where $K$ is $\mathbb{Z}^k$-gain graph with vertex set $V(G)$
	and edge set $K(V(G)) \setminus \{ (v,v,0): v \in V(G)\}$.
\end{remark}

\begin{definition}
	Let $(G,p,L)$ be a $k$-periodic framework.
	Then we define the following:
	\begin{enumerate}[(i)]
		\item $(G,p,L)$ is \emph{rigid} if all flexes of $(G,p,L)$ are trivial, and \emph{flexible} otherwise.
		\item $(G,p,L)$ is \emph{fixed lattice rigid} if all fixed lattice flexes of $(G,p,L)$ are trivial,
		and \emph{fixed lattice flexible} otherwise.
	\end{enumerate}
\end{definition}

Let $\phi_u^\mu$ be a switching operation of $G$.
We define the \emph{framework switching operation at $u$ by $\mu$} to be (by abuse of notation) the linear map
$\phi_u^\mu : \mathbb{K}^{d|V(G)|+dk} \rightarrow \mathbb{K}^{d|V(G)|+dk}$,
where,
given $(p',L') = \phi_u^\mu(p,L)$,
we have $L'=L$ and
\begin{align*}
	p'(v) = 
	\begin{cases}
	p(u) + L.\mu, &\text{ if } v=u \\
	p(v), &\text{ otherwise}
	\end{cases}
\end{align*}
for all $v \in V^d_{\mathbb{K}}(G)$.
We define any composition $\phi:= \phi_{u_n}^{\mu_n} \circ \ldots \circ \phi_{u_1}^{\mu_1}$ to be a \emph{gain equivalence},
and define $\phi(G,p,L) := (\phi_u^\mu (G) , \phi_u^\mu(p,L))$.
If there exists a gain equivalence such that $(G',p',L) = \phi(G,p,L)$,
then we say $(G,p,L)$ and $(G',p',L)$ are \emph{gain-equivalent};
we denote that two $k$-periodic frameworks $(G,p,L)$ and $(G',p',L)$ are gain equivalent by $(G,p,L) \approx (G',p',L)$.

As each gain equivalence $\phi$ is a linear isomorphism and $\phi(\mathcal{V}^d_{\mathbb{K}}(G,p,L)) = \mathcal{V}_{\mathbb{K}}^d(\phi(G,p,L))$,
then the sets $\mathcal{V}_{\mathbb{K}}(G,p,L)$ and $\mathcal{V}_{\mathbb{K}}(\phi(G,p,L))$ are isomorphic as algebraic sets.
It follows that,
given $(G,p,L) \approx (G',p',L)$,
we have that $(G,p,L)$ is (fixed lattice) rigid if and only if $(G',p',L)$ is (fixed lattice) rigid.

\section{NBAC-colourings and flexibility in the plane}\label{sec:NBAC}

\subsection{NBAC-colourings}

\begin{definition}
	Let $G$ be a $\Gamma$-gain graph with edge colouring $\delta: E(G) \rightarrow \{\text{red}, \text{blue}\}$.
	We define the following:
	\begin{enumerate}[(i)]
		\item $G^\delta_{\text{red}} := (V(G),\{ e \in E(G) : \delta(e) = \text{red} \})$.
		\item A \emph{red component} is a connected component of $G^\delta_{\text{red}}$.
		\item A \emph{red walk} (respectively, \emph{red circuit}) is a walk (respectively, circuit) where every edge is red.
		\item An \emph{almost red circuit} is a circuit with exactly one blue edge.
		\item $G^\delta_{\text{blue}}$, \emph{blue components}, \emph{blue walks},
		\emph{blue circuits}, and \emph{almost blue circuits} are defined analogously. 
		\item We define a component/walk/circuit to be \emph{monochromatic} if it is either red or blue,
		and we define an \emph{almost monochromatic circuit} to be any circuit that is either an almost red or almost blue.
\end{enumerate}
	A colouring $\delta$ is a \emph{NBAC-colouring} (\emph{No Balanced Almost Circuits})
	if it is surjective,
	and there are no balanced almost red circuits and no balanced almost blue circuits;
	see Figure \ref{fig:NBAC} for an example of an NBAC-colouring. 
\end{definition}

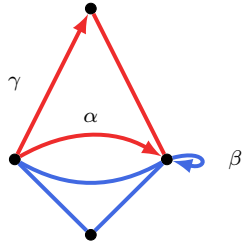
\begin{figure}[ht]
	\begin{center}	
		\begin{tikzpicture}
			\node[vertex] (1) at (-1,0) {};
			\node[vertex] (2) at (1,0) {};
			\node[vertex] (3) at (0,2) {};
			\node[vertex] (4) at (0,-1) {};
			
			\path[redge,-latex] (1) edge [bend left] node {} (2);
			\path[bedge] (1) edge [bend right] node {} (2);
			
			\draw[bedge] (1)edge(4);
			\draw[bedge] (2)edge(4);
			
			\draw[redge,-latex] (1)edge(3);
			\draw[redge] (2)edge(3);
			
			\node[font=\scriptsize] at (0,0.55) {$\alpha$};
			\node[font=\scriptsize] at (1.9,0) {$\beta$};
			\node[font=\scriptsize] at (-1,1) {$\gamma$};
				
			\path[bedge,-latex, every loop/.style={looseness=30}] (2) edge [loop right] node { } (2);
	
		\end{tikzpicture}
	\end{center}
	\caption{A surjective colouring $\delta$ of a $\Gamma$-gain graph.
	If $\alpha \notin \langle \beta \rangle$, $\beta \notin \langle \alpha -\gamma \rangle$ and $\gamma \neq 0$,
	then $\delta$ is a NBAC-colouring.}
	\label{fig:NBAC}
\end{figure}

If $\delta$ is a colouring of $G$ and $G' \approx G$,
then by abuse of notation we shall also define $\delta$ to be a colouring for $G'$.
We note that if $\delta$ is a NBAC-colouring of $G$,
then $\delta$ is a NBAC-colouring of $G'$.

\begin{definition}
	Let $G$ be a $\mathbb{Z}^k$-gain graph for some $k \in \{1,2\}$,
	with an NBAC-colouring $\delta$. 
	If either	$G^\delta_{\text{red}}$ is balanced	and $G$ has no almost blue circuits,
	or $G^\delta_{\text{blue}}$ is balanced	and $G$ has no almost red circuits,
	then $\delta$ is a \emph{fixed lattice NBAC-colouring}.
\end{definition}

\begin{definition}
	Let $G$ be a $\mathbb{Z}$-gain graph with an NBAC-colouring $\delta$. 	
	If both $G^\delta_{\text{red}}$ and $G^\delta_{\text{blue}}$ are balanced,
	then $\delta$ is a \emph{flexible 1-lattice NBAC-colouring}.
\end{definition}

\begin{remark}
	We note that if $G$ is a $\mathbb{Z}$-gain graph with NBAC-colouring $\delta$,
	then $\delta$ can be either both a fixed lattice NBAC-colouring and a flexible $1$-lattice NBAC-colouring,
	one or the other, or neither.
	We can even have that $G$ has no NBAC-colouring that is both, but has both fixed lattice and a flexible $1$-lattice NBAC-colourings;
	see Figure \ref{fig:fixedNBAC} for an example.
\end{remark}

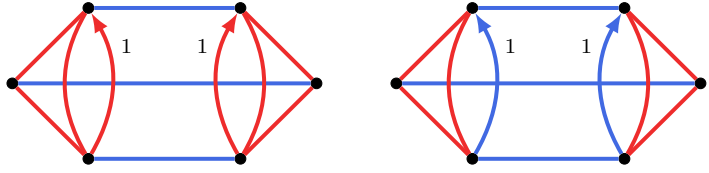
\begin{figure}[ht]
	\begin{center}
		\begin{tikzpicture}
			\node[vertex] (1) at (-1,-1) {};
			\node[vertex] (2) at (-1,1) {};
			\node[vertex] (3) at (1,-1) {};
			\node[vertex] (4) at (1,1) {};
			\node[vertex] (5) at (-2,0) {};
			\node[vertex] (6) at (2,0) {};

			\draw[bedge] (1)edge(3);
			\draw[bedge] (2)edge(4);
			\draw[bedge] (5)edge(6);
			
			\draw[redge] (2)edge(5);
			\draw[redge] (1)edge(5);
			\path[redge] (1) edge [bend left] node {} (2);
			\path[redge,-latex] (1) edge [bend right] node {} (2);
			
			\draw[redge] (4)edge(6);
			\draw[redge] (3)edge(6);
			\path[redge,-latex] (3) edge [bend left] node {} (4);
			\path[redge] (3) edge [bend right] node {} (4);
			
			\node[font=\scriptsize] at (-0.5,0.5) {$1$};
			\node[font=\scriptsize] at (0.5,0.5) {$1$};
	
		\end{tikzpicture} \qquad 
		\begin{tikzpicture}
			\node[vertex] (1) at (-1,-1) {};
			\node[vertex] (2) at (-1,1) {};
			\node[vertex] (3) at (1,-1) {};
			\node[vertex] (4) at (1,1) {};
			\node[vertex] (5) at (-2,0) {};
			\node[vertex] (6) at (2,0) {};

			\draw[bedge] (1)edge(3);
			\draw[bedge] (2)edge(4);
			\draw[bedge] (5)edge(6);
			
			\draw[redge] (2)edge(5);
			\draw[redge] (1)edge(5);
			\path[redge] (1) edge [bend left] node {} (2);
			\path[bedge,-latex] (1) edge [bend right] node {} (2);
			
			\draw[redge] (4)edge(6);
			\draw[redge] (3)edge(6);
			\path[bedge,-latex] (3) edge [bend left] node {} (4);
			\path[redge] (3) edge [bend right] node {} (4);
			
			\node[font=\scriptsize] at (-0.5,0.5) {$1$};
			\node[font=\scriptsize] at (0.5,0.5) {$1$};
	
		\end{tikzpicture} \qquad 
		\begin{tikzpicture}
			\node[vertex] (1) at (-1,-1) {};
			\node[vertex] (2) at (-1,1) {};
			\node[vertex] (3) at (1,-1) {};
			\node[vertex] (4) at (1,1) {};
			\node[vertex] (5) at (-2,0) {};
			\node[vertex] (6) at (2,0) {};

			\draw[redge] (1)edge(3);
			\draw[redge] (2)edge(4);
			\draw[redge] (5)edge(6);
			
			\draw[redge] (2)edge(5);
			\draw[redge] (1)edge(5);
			\path[redge] (1) edge [bend left] node {} (2);
			\path[bedge,-latex] (1) edge [bend right] node {} (2);
			
			\draw[redge] (4)edge(6);
			\draw[redge] (3)edge(6);
			\path[bedge,-latex] (3) edge [bend left] node {} (4);
			\path[redge] (3) edge [bend right] node {} (4);
			
			\node[font=\scriptsize] at (-0.5,0.5) {$1$};
			\node[font=\scriptsize] at (0.5,0.5) {$1$};
	
		\end{tikzpicture}
	\end{center}
	\caption{All three possible NBAC-colourings of a $\mathbb{Z}$-gain graph up to switching the colours red and blue.
	The left is a fixed lattice NBAC-colouring but not a flexible $1$-lattice NBAC-colouring,
	while the middle and right are flexible $1$-lattice NBAC-colourings but not fixed lattice NBAC-colourings.}
	\label{fig:fixedNBAC}
\end{figure}

\begin{definition}
	Let $G$ be a $\mathbb{Z}^2$-gain graph with an NBAC-colouring $\delta$.
	We define the following:
	\begin{enumerate}[(i)]
		\item If both $G^\delta_{\text{red}}$ and $G^\delta_{\text{blue}}$ are balanced,
		then $\delta$ is a \emph{type 1 flexible 2-lattice NBAC-colouring}.
		\item If there exists $\alpha,\beta \in \mathbb{Z}^2$ 
		such that 
		\begin{itemize}
			\item either $\alpha,\beta$ are linearly independent or exactly one of $\alpha,\beta$ is equal to $(0,0)$,
			\item $\spann(G^\delta_\text{red})$ is a non-trivial subgroup of $\mathbb{Z}\alpha$, or $\alpha=(0,0)$ and $G^\delta_\text{red}$ is balanced, 
			\item $\spann(G^\delta_\text{blue})$ is a non-trivial subgroup of $\mathbb{Z}\beta$, or $\beta=(0,0)$ and $G^\delta_\text{blue}$ is balanced,
			\item there are no almost red circuits with gain in $\mathbb{Z}\alpha$, and
			\item there are no almost blue circuits with gain in $\mathbb{Z}\beta$,
		\end{itemize}
		then $\delta$ is a \emph{type 2 flexible 2-lattice NBAC-colouring}.
		\item If there exists $\alpha \in \mathbb{Z}^2 \setminus \{(0,0)\}$ such that
		\begin{itemize}
			\item $\spann(G^\delta_\text{red})$ and $\spann(G^\delta_\text{red})$ are non-trivial subgroups of $\mathbb{Z}\alpha$, and
			\item there are no almost monochromatic circuits with gain in $\mathbb{Z}\alpha$, 
		\end{itemize}
		then $\delta$ is a \emph{type 3 flexible 2-lattice NBAC-colouring}.
	\end{enumerate}		 	
\end{definition}

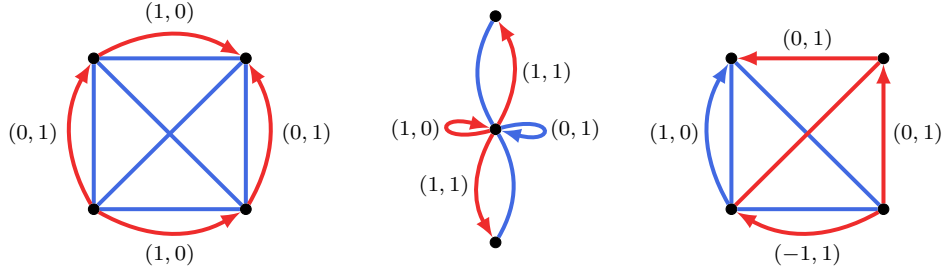
\begin{figure}[ht]
	\begin{center}
		\begin{tikzpicture}
			\node[vertex] (1) at (-1,-1) {};
			\node[vertex] (2) at (1,-1) {};
			\node[vertex] (3) at (1,1) {};
			\node[vertex] (4) at (-1,1) {};

			\draw[bedge] (1)edge(2) (1)edge(3) (1)edge(4) (2)edge(3) (2)edge(4) (3)edge(4);
			\path[redge,-latex] (1) edge [bend right] node {} (2);
			\path[redge,-latex] (2) edge [bend right] node {} (3);
			\path[redge,-latex] (4) edge [bend left] node {} (3);
			\path[redge,-latex] (1) edge [bend left] node {} (4);
			
			\node[font=\scriptsize] at (0,1.6) {$(1,0)$};
			\node[font=\scriptsize] at (0,-1.6) {$(1,0)$};
			\node[font=\scriptsize] at (1.8,0) {$(0,1)$};
			\node[font=\scriptsize] at (-1.8,0) {$(0,1)$};
		\end{tikzpicture} \quad
		\begin{tikzpicture}[scale=1.5]
			\node[vertex] (1) at (0,-1) {};
			\node[vertex] (2) at (0,0) {};
			\node[vertex] (3) at (0,1) {};

			\path[redge,-latex] (2) edge [bend right] node {} (1);
			\path[bedge] (2) edge [bend left] node {} (1);
			\path[redge,-latex] (2) edge [bend right] node {} (3);
			\path[bedge] (2) edge [bend left] node {} (3);

			\path[bedge,-latex, every loop/.style={looseness=30}] (2) edge [loop right] node { } (2);
			\path[redge,-latex, every loop/.style={looseness=30}] (2) edge [loop left] node { } (2);
			
			\node[font=\scriptsize] at (0.45,0.5) {$(1,1)$};
			\node[font=\scriptsize] at (-0.45,-0.5) {$(1,1)$};
			\node[font=\scriptsize] at (0.7,0) {$(0,1)$};
			\node[font=\scriptsize] at (-0.7,0) {$(1,0)$};
			\node[font=\scriptsize] at (0,-1.2) {};
		\end{tikzpicture}\quad
		\begin{tikzpicture}
			\node[vertex] (1) at (-1,-1) {};
			\node[vertex] (2) at (1,-1) {};
			\node[vertex] (3) at (1,1) {};
			\node[vertex] (4) at (-1,1) {};

			\draw[bedge] (2)edge(1);
			\draw[bedge] (1)edge(4);
			\draw[bedge] (2)edge(4);
			
			\draw[redge] (1)edge(3);

			\draw[redge,-latex] (2)edge(3);
			
			\draw[redge,-latex] (3)edge(4);

			\path[redge,-latex] (2) edge [bend left] node {} (1);
			\path[bedge,-latex] (1) edge [bend left] node {} (4);
			
			\node[font=\scriptsize] at (0,-1.6) {$(-1,1)$};
			\node[font=\scriptsize] at (-1.75,0) {$(1,0)$};
			\node[font=\scriptsize] at (0,1.25) {$(0,1)$};
			\node[font=\scriptsize] at (1.45,0) {$(0,1)$};

		\end{tikzpicture}
	\end{center}
	\caption{(Left): A $\mathbb{Z}^2$-gain graph with a type 1 flexible $2$-lattice NBAC-colouring.
	(Middle): A $\mathbb{Z}^2$-gain graph with a type 2 flexible $2$-lattice NBAC-colouring ($\alpha = (1,0)$, $\beta = (0,1)$).
	(Right): A $\mathbb{Z}^2$-gain graph with a type 3 flexible $2$-lattice NBAC-colouring ($\alpha = (1,0)$).}
	\label{fig:2latticeNBAC}
\end{figure}

\begin{remark}
	We note that if $G$ is a $\mathbb{Z}^2$-gain graph with NBAC-colouring $\delta$,
	then the following holds.
	\begin{itemize}
		\item For distinct $i,j \in \{1,2,3\}$,
		$\delta$ cannot be both a type $i$ and type $j$ flexible $2$-lattice NBAC-colouring.
		\item Similarly, $\delta$ cannot be both a fixed lattice NBAC-colouring and type 3 flexible $2$-lattice NBAC-colouring.
		\item The colouring $\delta$ can, however, be both a fixed lattice NBAC-colouring and either a type 1 or 2 flexible $2$-lattice NBAC-colouring; see Figure \ref{fig:type2fixedNBAC} for an example of an NBAC-colouring that is both fixed lattice and type 2.
		\item If $H \subset G$ is not monochromatic and $\delta$ is a type $k$ flexible $2$-lattice NBAC-colouring for some $k \in \{1,2,3\}$,
		then $\delta$ restricted to $H$ is a type $k'$ flexible $2$-lattice NBAC-colouring for some $1 \leq k' \leq k$;
		furthermore, if $k'=1 <k$ then $\delta$ restricted to $H$ will also be a fixed-lattice NBAC-colouring.
	\end{itemize}
\end{remark}

\begin{figure}[ht]
	\begin{center}	
		\begin{tikzpicture}
			\node[vertex] (1) at (-1,-1) {};
			\node[vertex] (2) at (1,-1) {};
			\node[vertex] (3) at (1,1) {};
			\node[vertex] (4) at (-1,1) {};

			\draw[redge] (1)edge(4);
			\draw[redge,-latex] (1)edge(2);
			\draw[redge] (2)edge(4);
			\draw[bedge,-latex] (3)edge(4);
			\draw[bedge] (3)edge(1);
			\draw[bedge,-latex] (2)edge(3);
			
			\node[font=\scriptsize] at (0,-1.3) {$(1,0)$};
			\node[font=\scriptsize] at (0,1.3) {$(0,1)$};
			\node[font=\scriptsize] at (1.4,0) {$(0,1)$};
			\node[font=\scriptsize] at (1.4,1.3) {$(1,0)$};
				
			\path[redge,-latex, every loop/.style={looseness=30}] (3) edge [loop right] node { } (3);
	
		\end{tikzpicture}
	\end{center}
	\caption{A $\mathbb{Z}^2$-gain graph with a colouring that is both a fixed lattice NBAC-colouring and a type 2 flexible $2$-lattice NBAC-colouring ($\alpha= (1,0)$, $\beta= (0,0)$).}
	\label{fig:type2fixedNBAC}
\end{figure}
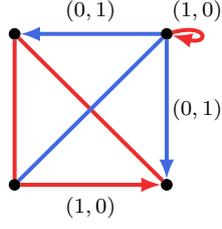

\subsection{k-periodic frameworks in the plane}

Let $G$ be a $\mathbb{Z}^k$-gain graph for $k \in \{1,2\}$,
with placement $p : V(G) \rightarrow \mathbb{R}^2$ and lattice $L \in M_{2 \times k}(\mathbb{R})$;
if $k=1$ we shall define $L_1:= L.1$ and if $k=2$ we shall define $L_1 := L.(1,0)$ and $L_2 := L.(0,1)$.
For each $e = (v,w,\gamma)$ with $\gamma := (\gamma_j)_{j=1}^k$,
we define
\begin{align*}
	\lambda(e) := \left\|p(v)- p(w) - \sum_{j=1}^k \gamma_j L_j \right\|
\end{align*}
(we note that this is well-defined as $(G,p,L)$ is a $k$-periodic framework in $\mathbb{R}^2$).
We further define for each $1 \leq j,l \leq k$,
\begin{align*}
	\lambda(j,l) := L_j . L_l.
\end{align*}

We shall consider each point $(q,M) \in \mathcal{V}_{\mathbb{C}}^2(G)$ to be point
\begin{align*}
	\left((x_v, y_v)_{v \in V(G)} ,(x_j,y_j)_{j=1}^k \right),
\end{align*}
where $x_v,y_v,x_j,y_j \in \mathbb{C}$;
the points $(x_v,y_v)$ will correspond to the coordinates of $q_v$,
and the points $(x_j,y_j)$ will correspond to the coordinates of $L_j$.
To help simplify things later on, we will first wish to quotient out $\mathcal{V}_{\mathbb{C}}^2(G)$ by the orientation-preserving isometries by fixing an edge $\tilde{e}= (\tilde{v},\tilde{w}, \tilde{\gamma})$.
To do so,
we define the algebraic set $\Vset \subset \mathcal{V}_{\mathbb{C}}^2(G)$ of all points where
\begin{align*}
	x_{\tilde{v}}= y_{\tilde{v}}= 0, \qquad y_{\tilde{w}} + \sum_{j=1}^k \tilde{\gamma}_j y_j = 0,
\end{align*}
and for all $e= (v,w,\gamma) \in E(G)$,
\begin{align}\label{edgecond2}
	\left(x_v- x_w - \sum_{j=1}^k \gamma_j x_j \right)^2 + \left(y_v- y_w - \sum_{j=1}^k \gamma_j y_j \right)^2 = \lambda(e)^2.
\end{align}
We further define $\Vsetcirc$ to be the algebraic subset of $\Vset$ where
\begin{align*}
	x_j x_l + y _j y_l = \lambda(j,l)^2,
\end{align*}
for each $1 \leq j,l \leq k$.

We note that the placement-lattice $(p,L)$ may not be contained in $\Vset$.
However,
the unique $k$-periodic framework obtained by translating and rotating $(G,p,L)$ so that $p_{\tilde{v}}$ lies on the origin and $p_{\tilde{w}}+L.\tilde{\gamma}$ lies on the $x$-axis will be contained in $\Vset$.
Hence, the set $\mathcal{V}_{\mathbb{C}} (G,p,L)$ is homeomorphic to
\begin{align*}
	\Vset \times SO(2, \mathbb{C}) \times \mathbb{C}^2,
\end{align*}
as $\Vset$ is the set of equivalent frameworks to $(G,p,L)$ in $\mathbb{C}^2$ where the edge $\tilde{e}$ is fixed to lie on the $x$-axis.
Similarly, $\mathcal{V}^f_{\mathbb{C}}(G,p,L)$ is homeomorphic to
\begin{align*}
	\Vsetcirc \times SO(2, \mathbb{C}) \times \mathbb{C}^2.
\end{align*}
It follows that, if we require it, we may assume $(p,L) \in \Vset$.

Given an algebraic curve $\mathcal{C} \subset \Vset$ and any $v,w \in V(G)$, $\gamma \in \mathbb{Z}^k$, we define the maps 
\begin{align*}
	W^\gamma_{v,w,}|_\mathcal{C},Z^\gamma_{v,w}|_\mathcal{C} : \mathcal{C} \rightarrow \mathbb{C}
\end{align*}
by the polynomials
\begin{eqnarray*}
	W^\gamma_{v,w} |_\mathcal{C}
	&:=& \left( x_v - x_w - \sum_{j=1}^k \gamma_j x_j \right) + i \left(y_v - y_w - \sum_{j=1}^k \gamma_j y_j \right) \\
	Z^\gamma_{v,w} |_\mathcal{C}
	&:=& \left( x_v - x_w - \sum_{j=1}^k \gamma_j x_j \right) - i \left(y_v- y_w - \sum_{j=1}^k \gamma_j y_j \right).
\end{eqnarray*}
We further define the maps $W_j|_{\mathcal{C}},Z_j|_{\mathcal{C}} :\mathcal{C} \rightarrow \mathbb{C}$ for $1 \leq j \leq k$ as the polynomials,
\begin{eqnarray*}
	W_j |_\mathcal{C} &:=& x_j + i y_j \\
	Z_j |_\mathcal{C} &:=& x_j - i y_j.
\end{eqnarray*}
For the case of $k=2$,
we shall define for each $\gamma := (a,b) \in \mathbb{Z}^2$ the maps
\begin{align*}
	\gamma W|_{\mathcal{C}} := a W_1|_{\mathcal{C}} + b W_2|_{\mathcal{C}}, \qquad \gamma Z|_{\mathcal{C}} := a Z_1|_{\mathcal{C}} + b Z_2|_{\mathcal{C}}.
\end{align*}
When there is no ambiguity regarding which algebraic curve we are observing,
we shall for brevity drop the notation ``$~|_{\mathcal{C}}$'';
for example, $W^\gamma_{v,w} |_\mathcal{C}$ shall be shortened to $W^\gamma_{v,w}$.
%
%

We first observe that $W^{-\gamma}_{w,v} = - W^\gamma_{v,w}$ and $Z^{-\gamma}_{w,v} = - Z^\gamma_{v,w}$.
Furthermore,
we note that if $e = (v,w,\gamma) \in E(G)$,
\begin{align*}
	W^\gamma_{v,w} Z^\gamma_{v,w} = \lambda(e)^2,
\end{align*}
 and if $\mathcal{C}  \subset \Vsetcirc$ then 
\begin{align*}
	W_j.Z_j = \lambda(j,j)^2 , \qquad W_j.Z_l + W_l.Z_j = 2 \lambda(j,l)^2,
\end{align*}
for all $1 \leq j,l \leq k$.

\subsection{Active NBAC-colourings}

Active NAC-colourings for finite simple graphs were first introduced in \cite{jan1}.
We shall now give an analogue of them for $\mathbb{Z}^k$-gain graphs.

\begin{definition}
	Let $(G,p,L)$ be $k$-periodic framework in $\mathbb{R}^2$,
	$\mathcal{C} \subset \Vset$ be an algebraic curve
	and $\delta$ a NBAC-colouring of $G$.
	We define $\delta$ to be an \emph{active NBAC-colouring of $\mathcal{C}$} 
	if there exists a valuation $\nu$ of $\mathbb{C}(\mathcal{C})$ 
	and $\alpha \in \mathbb{R}$ such that for each $e \in E(G)$,
	\begin{align*}
		\delta(e) =
		\begin{cases}
			\text{red}, & \text{ if } \nu(W^\gamma_{v,w}) > \alpha, \\
			\text{blue}, & \text{ if } \nu(W^\gamma_{v,w}) \leq \alpha;		
		\end{cases}
	\end{align*}
	if this is the case, we shall say that $\delta$ is the NBAC-colouring \emph{generated by $\nu$ and $\alpha$}.
	For a $k$-periodic framework $(G,p,L)$ in $\mathbb{R}^2$,
	we define $\delta$ to be an \emph{active NBAC-colouring of $(G,p,L)$} if it is an active NBAC-colouring of 
	an algebraic curve $\mathcal{C} \subset \Vset$.
	We define $\delta$ to be an \emph{active NBAC-colouring of $G$} if it is an active NBAC-colouring of a full $k$-periodic framework $(G,p,L)$ in $\mathbb{R}^2$.
\end{definition}

\begin{remark}\label{rem:active2}
	If $\delta$ is an active NBAC-colouring of an algebraic curve $\mathcal{C} \subset \Vset$ 
	and $\delta'$ is the NBAC-colouring with $\delta'(e) \neq \delta(e)$ for all $e \in E(G)$,
	then $\delta'$ is also an active NBAC-colouring of $\mathcal{C}$;
	this can be shown in a similar way to the proof of \cite[Lemma 1.13]{jan1}.
\end{remark}

\begin{lemma}\label{lem:active1}
	Let $(G,p,L)$ be a $k$-periodic framework in $\mathbb{R}^2$,
	and $e_1, e_2 \in E(G)$, with $e_1 = (v_1,w_1,\gamma_1)$
	and $e_2 = (v_2,w_2,\gamma_2)$.
	Then the map 
	\begin{align*}
		f_{e_1,e_2} : \mathcal{V}_{e_1}(G,p,L) \rightarrow \mathcal{V}_{e_2}(G,p,L), ~ 
		(q,M) \mapsto \left( (R_{e_2}.(q(v) - q(v_2)))_{v\in V(G)}, R_{e_2}M \right)
	\end{align*}
	is biregular,
	where
	\begin{align*}
		R_{e_2}:= 
		\frac{1}{\lambda(e_2)}
		\begin{bmatrix}
			x_{w_2} - x_{v_2} & y_{w_2} - y_{v_2} \\
			-(y_{w_2} - y_{v_2}) & x_{w_2} - x_{v_2} 
		\end{bmatrix}.
	\end{align*}
	Furthermore,
	for any algebraic curve $\mathcal{C} \subset \mathcal{V}_{e_1}(G,p,L)$
	and any $v,w \in V(G)$, $\gamma \in \mathbb{Z}^k$,
	we have that $\mathcal{C}' := f_{e_1,e_2}(\mathcal{C})$ is an algebraic curve and
	\begin{align}\label{eqn:WZ}
		W_{v,w}^{\gamma}|_{\mathcal{C}'} \circ f_{e_1,e_2} = \frac{1}{\lambda(e_2)}
		W_{v,w}^{\gamma}|_{\mathcal{C}} Z_{v_2,w_2}^{\gamma_2}|_{\mathcal{C}}, \quad
		Z_{v,w}^{\gamma}|_{\mathcal{C}'} \circ f_{e_1,e_2} = \frac{1}{\lambda(e_2)}
		Z_{v,w}^{\gamma}|_{\mathcal{C}} W_{v_2,w_2}^{\gamma_2}|_{\mathcal{C}}.
	\end{align}
\end{lemma}

\begin{proof}
	We note that the transform $z \mapsto R_{e_2}.(z - q(v_2))$ will preserve distance under $\| \cdot\|^2$ in $\mathbb{C}^2$.
	It follows that $(G,f_{e_1,e_2}(q,M))$ will be an equivalent framework to $(G,q,M)$,
	except now the edge $e_2$ (not $e_1$) has been fixed,
	with $v_2$ at the origin and $w_2$ on the $y$-axis,
	Hence $f_{e_1,e_2}(q,M) \in \mathcal{V}_{e_2}(G,p,L)$ for all $(q,M) \in \mathcal{V}_{e_1}(G,p,L)$,
	i.e.~the map $f_{e_1,e_2}$ is well-defined.
	It is clear that the map $f_{e_1,e_2}$ is regular.
	To see that $f_{e_1,e_2}$ is biregular,
	we note that the map $f_{e_2,e_1}$ is the inverse of $f_{e_1,e_2}$.
	
	Since $f_{e_1,e_2}$ is biregular, $\mathcal{C}'$ will be an algebraic curve.
	Equation (\ref{eqn:WZ}) now holds by direct computation.
\end{proof}

\begin{proposition}\label{prop:active}
	Let $(G,p,L)$ be a $k$-periodic framework in $\mathbb{R}^2$,
	$e_1, e_2 \in E(G)$ with $e_1 = (v_1,w_1,\gamma_1)$
	and $e_2 = (v_2,w_2,\gamma_2)$, and $\mathcal{C} \subset \mathcal{V}_{e_1}(G,p,L)$.
	If $\delta$ is an active NBAC-colouring of $\mathcal{C}$ then there exists an algebraic curve 
	$\mathcal{C}' \subset \mathcal{V}_{e_2}(G,p,L)$ such that 
	$\delta$ is an active NBAC-colouring of $\mathcal{C}'$.
\end{proposition}

\begin{proof}
	Let $\mathcal{C}' := f_{e_1,e_2}(\mathcal{C})$,
	where $f_{e_1,e_2}$ is the map defined in Lemma \ref{lem:active1}.
	Let $\nu$ be the valuation of $\mathbb{C}(\mathcal{C})$ and $\alpha \in \mathbb{R}$ be chosen so that they generate $\delta$.
	Define $\nu'$ to be the valuation of $\mathbb{C}(\mathcal{C}')$ where $\nu'(f) := \nu(f \circ f_{e_1,e_2})$ 
	for each $f \in \mathbb{C}(\mathcal{C}')$.
	By Lemma \ref{lem:active1},
	\begin{align*}
		\nu'\left( W_{v,w}^{\gamma}|_{\mathcal{C}'} \right) = \nu \left(W_{v,w}^{\gamma}|_{\mathcal{C}'} \circ f_{e_1,e_2} \right) =
		\nu \left( \frac{1}{\lambda(e_2)}
		W_{v,w}^{\gamma}|_{\mathcal{C}} Z_{v_2,w_2}^{\gamma_2}|_{\mathcal{C}} \right) =
		\nu \left( W_{v,w}^{\gamma}|_{\mathcal{C}} \right) + \nu \left( Z_{v_2,w_2}^{\gamma_2}|_{\mathcal{C}} \right).
	\end{align*}
	If we define $\alpha' := \alpha + \nu \left( Z_{v_2,w_2}^{\gamma_2}|_{\mathcal{C}} \right)$,
	then $\nu'$ and $\alpha'$ will generate $\delta$.
\end{proof}

\begin{lemma}\label{lem:activegain}
	Let $(G,p,L)$ and $(G',p',L)$ be gain equivalent frameworks with gain equivalence 
	$\phi : \mathcal{V}^d_{\mathbb{K}}(G) \rightarrow \mathcal{V}^d_{\mathbb{K}}(G')$.
	If $\tilde{e} \in E(G)$ and $\tilde{e}' := \phi(\tilde{e})$,
	then $\phi$ is a biregular map with $\phi(\Vset) = \mathcal{V}_{\tilde{e}'}(G',p',L)$.
	Furthermore,
	for any algebraic curve $\mathcal{C} \subset \mathcal{V}_{e_1}(G,p,L)$
	and any $v,w \in V(G)$, $\gamma \in \mathbb{Z}^k$,
	we have that $\mathcal{C}' := \phi(\mathcal{C})$ is an algebraic curve and
	\begin{align}\label{eqn:WZgain}
		W_{v,w}^{\gamma}|_{\mathcal{C}'} \circ \phi = 	W_{v,w}^{\gamma}|_{\mathcal{C}}, \quad
		Z_{v,w}^{\gamma}|_{\mathcal{C}'} \circ \phi = Z_{v,w}^{\gamma}|_{\mathcal{C}}.
	\end{align}
\end{lemma}

\begin{proof}
	As $\phi$ is a bijective map that is the restriction of an invertible linear map,
	it is a biregular map;
	hence, $\phi(\mathcal{C})$ is an algebraic curve.
	Equation (\ref{eqn:WZgain}) now follows by direct computation.	
\end{proof}

\begin{proposition}\label{prop:activegain}
	Let $G$ and $G'$ be gain equivalent $\mathbb{Z}^k$-gain graphs.
	Then $\delta$ is an active NBAC-colouring of $G$ if and only if $\delta$ is an active NBAC-colouring of $G'$.
\end{proposition}

\begin{proof}
	Let $\delta$ be an active NBAC-colouring of $\mathcal{C} \subset \Vset$ generated by the valuation 
	$\nu$ of $\mathbb{C}(\mathcal{C})$ and $\alpha \in \mathbb{R}$.
	Let $\phi$ be the gain equivalence from $G$ to $G'$.
	We define the gain equivalent framework $(G',p',L) := \phi (G,p,L)$,
	the algebraic curve $\mathcal{C}' := \phi(\mathcal{C})$ (Lemma \ref{lem:activegain}),
	and the valuation $\nu'$ of $\mathbb{C}(\mathcal{C}')$ where $\nu'(f) := \nu(f \circ \phi)$ 
	for each $f \in \mathbb{C}(\mathcal{C}')$.	
	By Lemma \ref{lem:activegain},
	\begin{align*}
		\nu' \left( W_{v,w}^{\gamma}|_{\mathcal{C}'} \right) = 
		\nu \left( W_{v,w}^{\gamma}|_{\mathcal{C}'}  \circ \phi\right) = \nu \left( W_{v,w}^{\gamma}|_{\mathcal{C}} \right),
	\end{align*}
	thus $\nu'$ and $\alpha$ generate $\delta$ for $G'$.
\end{proof}

\subsection{Key tools}

We are now ready to outline the key tools that shall help us throughout the rest of the paper.

\begin{lemma}\label{lem:existalgcurv}
	Let $(G,p,L)$ be a $k$-periodic framework in $\mathbb{R}^2$.
	Then the following holds:
	\begin{enumerate}[(i)]
		\item \label{lem:existalgcurvitem1} If $(G,p,L)$ is flexible,
		there exists an algebraic curve $\mathcal{C} \subset \Vset$.
		\item \label{lem:existalgcurvitem2} If $(G,p,L)$ is fixed lattice flexible,
		there exists an algebraic curve $\mathcal{C} \subset \Vsetcirc$.
	\end{enumerate}
\end{lemma}

\begin{proof}
	(\ref{lem:existalgcurvitem1}): 
	If $(G,p)$ is flexible then $\Vset$ cannot be finite.
	As every algebraic set that is not finite contains a variety with positive dimension
	and every variety with positive dimension contains an algebraic curve, the result holds.
	
	(\ref{lem:existalgcurvitem2}): 
	This follows from a similar method.
\end{proof}

\begin{lemma}\label{lem:constantW}
	Let $(G,p,L)$ be $k$-periodic and $\mathcal{C} \subset \Vset$ be an algebraic curve.
	Suppose $G$ contains a spanning tree $T$ that contains $\tilde{e}$ and has trivial gain for all of its edges.
	If $\rank(G) =k$,
	then there exists $(v,w, \gamma) \in E(G)$ such that $W^\gamma_{v,w}$ takes an infinite amount of values on $\mathcal{C}$.
\end{lemma}

\begin{proof}
	Suppose that for each $(v,w, \gamma) \in E(G)$,
	the map $W^\gamma_{v,w}$ takes a finite amount of values.
	By Lemma \ref{lem:constantW2},
	each map $W^\gamma_{v,w}$ is constant.
	As $W^\gamma_{v,w}Z^\gamma_{v,w}$ is constant,
	$Z^\gamma_{v,w}$ is also constant.
	Choose any two vertices $v, w \in V(G)$ with $v \neq w$.
	Then there exists a unique walk $v_1, \ldots, v_n$ from $v$ to $w$ in $T$.
	As
	\begin{align*}
	W^0_{v,w} = \sum_{j=1}^{n-1} W^0_{v_j,v_{j+1}}, \qquad Z^0_{v,w} = \sum_{j=1}^{n-1} Z^0_{v_j,v_{j+1}},
	\end{align*}
	both $W^0_{v,w}$ and $Z^0_{v,w}$ are constant;
	furthermore, as
	\begin{align*}
		(x_v - x_w) = \frac{1}{2}(W^0_{v,w} + Z^0_{v,w}), \qquad (y_v - y_w) = \frac{i}{2}(Z^0_{v,w} - W^0_{v,w}),
	\end{align*}	
	then $(x_v-x_w)$ and $(y_v-y_w)$ are also constant on $\mathcal{C}$.
	Since $x_{\tilde{v}},y_{\tilde{w}},x_{\tilde{w}},y_{\tilde{v}}$ are constant on $\mathcal{C}$ and both $\tilde{v}$ and $\tilde{w}$ are contained in $T$,
	both $x_v$, $y_v$ are constant on $\mathcal{C}$ for every $v \in V$ also.
		
	Suppose $k=1$ and let $e = (v,w, \gamma)$ be any edge with $\gamma \neq 0$.
	By observing the maps $W^\gamma_{v,w}$ and $Z^\gamma_{v,w}$,
	we note that $x_1$ and $y_1$ are constant on $\mathcal{C}$
	(since $x_v,x_w,y_v,y_w$ are all constant on $\mathcal{C}$).
	It now follows that $\mathcal{C}$ is a single point,
	contradicting that $\dim \mathcal{C} >0$. 
	
	Now suppose $k=2$.
	As $\rank (G) =k$,
	there exists edges $(v,w, \gamma)$ and $(v',w', \gamma')$
	such that $\gamma, \gamma'$ are independent.
	By observing the maps $W^\gamma_{v,w}, Z^\gamma_{v,w},W^{\gamma'}_{v',w'}, Z^{\gamma'}_{v',w'}$,
	we note that the polynomials
	\begin{eqnarray*}
		f &:=& \gamma_1 x_1 + \gamma_2 x_2, \qquad g:= \gamma_1 y_1 + \gamma_2 y_2 \\
		f' &:=& \gamma'_1 x_1 + \gamma'_2 x_2, \qquad g':= \gamma'_1 y_1 + \gamma'_2 y_2
	\end{eqnarray*}
	are constant on $\mathcal{C}$. 	
	As both $x_1$ and $x_2$ can be formed by linear combinations of $f,f'$,
	both are constant on $\mathcal{C}$;
	similarly,
	as both $y_1$ and $y_2$ can be formed by linear combinations of $g,g'$ then
	both $y_1$ and $y_2$ are also constant on $\mathcal{C}$.
	It now follows that $\mathcal{C}$ is a single point,
	contradicting that $\dim \mathcal{C} >0$.
\end{proof}

\begin{lemma}\label{lem:NBAC}
	Let $(G,p,L)$ be a full $k$-periodic framework in $\mathbb{R}^2$,
	$\rank (G) =k$ for $k \in \{1,2\}$ and $\mathcal{C} \subset \Vset$ an algebraic curve.
	Suppose there exists $a:=(a_1, a_2, \alpha) \in E(G)$ such that $W^\alpha_{a_1,a_2}$ takes an infinite amount of values on $\mathcal{C}$.
	Then there exists a valuation $\nu$ of $\mathbb{C}(\mathcal{C})$ such that the colouring 
	$\delta : E(G) \rightarrow \{ \text{red}, \text{blue} \}$ given by 
	\begin{align*}
		\delta(e) :=
		\begin{cases}
			\text{red}, &\text{ if } \nu(W^\gamma_{v,w}) >0, \\
			\text{blue}, &\text{ if } \nu(W^\gamma_{v,w}) \leq 0.		
		\end{cases}
	\end{align*}
	for each $e = (v,w,\gamma)$,
	is an NBAC-colouring of $G$;
	furthermore,
	$\delta(\tilde{e}) = \text{blue}$ and $\delta(a) = \text{red}$.
\end{lemma}

\begin{proof}
	By Lemma \ref{lem:constantW2}, 
	$W^\alpha_{a_1,a_2}$ is transcendental over $\mathbb{C}$,
	thus by Proposition \ref{prop:valuation},
	there exists a valuation $\nu$ of $\mathbb{C}(\mathcal{C})$ such that $\nu (W^\alpha_{a_1,a_2}) >0$.
	As $\tilde{e}$ is fixed and $\lambda(\tilde{e}) \neq 0$,
	we have $\nu ( W^{\tilde{\gamma}}_{\tilde{v},\tilde{w}}) =0$.
	We note that $\nu (W^\gamma_{v,w}Z^\gamma_{v,w}) =0$ for each $(v,w,\gamma) \in E(G)$ since $W^\gamma_{v,w} Z^\gamma_{v,w}$ is constant,
	hence $\nu (W^\gamma_{v,w}) = - \nu (Z^\gamma_{v,w})$.
	
	Let $\delta : E(G) \rightarrow \{ \text{red}, \text{blue} \}$ as described in the statement of the lemma for the valuation $\nu$.
	It follows that $a$ is red and $\tilde{e}$ is blue, thus $\delta$ is surjective. 
	Suppose there exists a balanced almost red circuit $C$ of length $n$ in $G$ with $\delta(e_n) = \text{blue}$.
	Then 
	\begin{align*}
		\nu(W^{\gamma_n}_{v_1,v_n})
		= \nu \left( \sum_{j=1}^{n-1} W^{\gamma_j}_{v_j,v_{j+1}} \right) 
		\geq \min \left\{ \nu (W^{\gamma_j}_{v_j,v_{j+1}}) : j =1, \ldots, n-1 \right\} >0,
	\end{align*}
	however this contradicts that $\nu (W^{\gamma_n}_{v_1,v_n}) \leq 0$.
	Now suppose instead that $C$ is a balanced almost blue circuit with $\delta(e_n) = \text{red}$.
	Then 
	\begin{align*}
		\nu(Z^{\gamma_n}_{v_1,v_{n}}) 
		= \nu \left( \sum_{j=1}^{n-1} Z^{\gamma_j}_{v_j,v_{j+1}} \right) 
		\geq \min \left\{ \nu (Z^{\gamma_j}_{v_j,v_{j+1}}) : j =1, \ldots, n-1 \right\}  \geq 0,
	\end{align*}
	however this contradicts that $\nu (Z^{\gamma_n}_{v_1,v_{n}}) < 0$. 
\end{proof}

\begin{definition}
	For any two edges $e_1,e_2$ of a $k$-periodic framework $(G,p,L)$ in $\mathbb{R}^2$
	with $e_i := (v_i,w_i,\gamma_i)$ for each $i \in \{1,2\}$,
	we define the \emph{angle function of $e_1,e_2$} to be the map
	\begin{align*}
		A_{e_1,e_2}: \mathcal{V}_{\mathcal{C}}^2(G) \rightarrow \mathbb{C}, 
		~ (p',L') \mapsto (p'(v_1)-p'(w_1) - L'.\gamma_1).(p'(v_2)-p'(w_2)- L'.\gamma_2).
	\end{align*}
\end{definition}

\begin{remark}
	We note that for any $\tilde{e} \in E(G)$ and any algebraic curve $\mathcal{C} \subset \Vset$,
	\begin{align*}
		A_{e_1,e_2}|_{\mathcal{C}} = 
		\frac{1}{2} \left( W_{v_1,w_1}^{\gamma_1}Z_{v_2,w_2}^{\gamma_2} + Z_{v_1,w_1}^{\gamma_1}W_{v_2,w_2}^{\gamma_2} \right).
	\end{align*}
	Furthermore, if $(p,L) \sim (p',L')$, then $A_{e_1,e_2}(p,L) = A_{e_1,e_2}(p',L')$;
	this is as linear isometries of $(\mathbb{C}^2, \|\cdot\|^2)$ will preserve the bilinear form associated to $\|\cdot\|^2$.
\end{remark}

\begin{lemma}\label{lem:active2}
	Let $(G,p,L)$ be $k$-periodic framework in $\mathbb{R}^2$ for $k \in \{1,2\}$,
	$\mathcal{C} \subset \Vset$ be an algebraic curve
	and $e_1,e_2 \in E(G)$, with $e_j := (v_j,w_j,\gamma_j)$ for each $j \in \{1,2\}$.
	If $\delta(e_1)=\delta(e_2)$ for all active NBAC-colourings of $\mathcal{C}$,
	then $A_{e_1,e_2}|_{\mathcal{C}}$	is constant.
\end{lemma}

\begin{proof}
	As $A_{e_1,e_2}$ is invariant for congruent placement-lattices,	
	by Proposition \ref{prop:active},
	we may assume $\tilde{e} = e_1$.
	We note the map
	\begin{align}\label{eqn:active2.1}
		(p',L') \mapsto p'(v_1)-p'(w_1) - L'.\gamma_1
	\end{align}
	is constant on $\mathcal{C}$, and $W^{\gamma_1}_{v_1,w_1}$ is constant also.
	Suppose $A_{e_1,e_2}|_{\mathcal{C}}$ is not constant,
	then as (\ref{eqn:active2.1}) is constant,
	\begin{align*}
		(p',L') \mapsto p'(v_2)-p'(w_2) - L'.\gamma_2
	\end{align*}
	is not constant on $\mathcal{C}$.
	This in turn implies that $W^{\gamma_2}_{v_2,w_2}$ takes an infinite amount of values over $\mathcal{C}$.
	By Lemma \ref{lem:NBAC},
	there exists an active NBAC-colouring $\delta$ of $\mathcal{C}$ with $\delta(e_1) \neq \delta(e_2)$.	
\end{proof}

\begin{lemma}\label{lem:active3}
	Let $(G,p,L)$ be $k$-periodic framework in $\mathbb{R}^2$ for $k \in \{1,2\}$,
	and $\tilde{e},e_1,e_2 \in E(G)$.
	If $A_{e_1,e_2}$ takes an infinite amount of values on $\Vset$
	then there exists an algebraic curve $\mathcal{C} \subset \Vset$ such that $A_{e_1,e_2}|_{\mathcal{C}}$ is not constant.
\end{lemma}

\begin{proof}
	As $A_{e_1,e_2}$ takes an infinite amount of values on $\Vset$,
	there exists a variety $\mathcal{V} \subset \Vset$ and points $(p',L'), (p'',L'') \in \Vset$ such that 
	$A_{e_1,e_2}(p',L') \neq A_{e_1,e_2}(p'',L'')$.
	By \cite[Lemma pg. 56]{mumford},
	there exists an algebraic curve $\mathcal{C}$ that contains $(p',L')$ and $(p'',L'')$.
\end{proof}

\begin{proposition}\label{prop:angleactive}
	Let $(G,p,L)$ be $k$-periodic framework in $\mathbb{R}^2$ for $k \in \{1,2\}$,
	and $\tilde{e},e_1,e_2 \in E(G)$.
	Then $\delta(e_1) = \delta(e_2)$ for all active NBAC-colourings $\delta$ of $(G,p,L)$ if and only if
	$A_{e_1,e_2}$ takes only finitely many values on $\Vset$.
\end{proposition}

\begin{proof}
	Suppose $\delta(e_1) = \delta(e_2)$ for all active NBAC-colourings $\delta$ of $(G,p,L)$.
	By Lemma \ref{lem:active2},
	$A_{e_1,e_2}|_{\mathcal{C}}$ is constant for any algebraic curve $\mathcal{C} \subset \Vset$.
	By Lemma \ref{lem:active3},
	it follows that $A_{e_1,e_2}$ takes only a finite amount of values on $\Vset$.
	
	Suppose there exists an algebraic curve $\mathcal{C}$ and active NBAC-colouring $\delta$ of $\mathcal{C}$ 
	generated by $\nu, \alpha$, 
	such that $\delta(e_1) \neq \delta(e_2)$.
	Let $e_j = (v_j,w_j,\gamma_j)$ for $j \in \{1,2\}$.
	Without loss of generality we may assume $\nu (W_{v_1,w_1}^{\gamma_1}) \leq \alpha < \nu (W_{v_2,w_2}^{\gamma_2})$.
	We now note
	\begin{align*}
		\nu (A_{e_1,e_2}|_{\mathcal{C}}) 
		= \nu (W_{v_1,w_1}^{\gamma_1}Z_{v_2,w_2}^{\gamma_2} + Z_{v_1,w_1}^{\gamma_1}W_{v_2,w_2}^{\gamma_2})
		= \nu (W_{v_1,w_1}^{\gamma_1}) -  \nu (W_{v_2,w_2}^{\gamma_2}) < 0,
	\end{align*}
	thus $A_{e_1,e_2}|_{\mathcal{C}}$ must be transcendental over $\mathbb{C}$ when considered as an element of $\mathcal{C}(\mathbb{C})$.
	It follows from Proposition \ref{prop:valuation} that $A_{e_1,e_2}$ takes an infinite amount of values on $\Vset$ as required.
\end{proof}

We shall end this section by defining a graph operation we shall use later in both
Lemma \ref{lem:1NBAC} and Lemma \ref{lem:2NBAC}.

\begin{definition}
	Let $(G,p,L)$ be a $k$-periodic framework in $\mathbb{R}^2$ and $\gamma \in \mathbb{Z}^k$ be a non-zero element.
	We define a $k$-periodic framework $(G',p',L)$ in $\mathbb{R}^2$ to be a 
	\emph{vertex addition of $(G,p,L)$ at $v_1$ by $\gamma$} if
	\begin{align*}
		V(G') := V(G) \cup \{v_0\} , \qquad E(G') := E(G) \cup \{(v_0,v_1,0) ,(v_0,v_1,\gamma)\} 
	\end{align*}
	and $p'(v)=p(v)$ for all $v \in V(G)$; see Figure \ref{fig:Henneberg1}.
\end{definition}

\begin{remark}
	The graph operation that takes $G$ to $G'$ in the vertex addition described above is the first of the two \emph{gain-preserving Henneberg moves};
	we refer the reader to \cite{nixonross} for more information.
\end{remark}

\begin{figure}[ht]
	\begin{center}
			%
			%
			%
			%
			%
			%
		%
		\begin{tikzpicture}
			\node[lnode] (1) at (4.5,-0.2) {$\boldsymbol{v_1}$};
			\node[lnode] (0) at (4.5,1.5) {$\boldsymbol{v_0}$};
			
			\draw[line width=1.5pt] (0.5,-0.35) ellipse (1.3cm and 0.7cm);
			
			\node[lnode] (1') at (0.5,-0.2) {$\boldsymbol{v_1}$};
			
			\draw[line width=1.5pt] (4.5,-0.35) ellipse (1.3cm and 0.7cm);
			
			\draw[-latex] (2.25,-0.35) -- (2.75,-0.35);
			
			\draw[edge,-latex] (1) edge [bend left] node {} (0);
			\draw[edge] (1)edge [bend right] node {} (0);
			
			\node[font=\scriptsize] at (3.9,0.65) {$\gamma$};
		\end{tikzpicture}
	\end{center}
	\caption{
	A vertex addition of $(G,p,L)$ at $v_1$ by $\gamma$.}
	\label{fig:Henneberg1}
\end{figure}

\begin{lemma}\label{lem:1NBAC.1}
	Let $(G,p,L)$ be a $k$-periodic framework in $\mathbb{R}^2$ with non-trivial flex $(p_t,L_t)$, $t \in [0,1]$.
	Assume that $\|L_t.\gamma\| \neq 0$ for all $t \in [0,1]$.	
	Then there exists a vertex addition $(G',p',L)$ of $(G,p,L)$ at $v_1$ by $\gamma$
	with non-trivial flex $(p_t',L_t)$ such that $p'_t$ restricted to $V(G)$ is the placement $p_t$ for each $t \in [0,1]$.
\end{lemma}
%

\begin{proof}
	As $[0,1]$ is compact,
	we may choose $r>0$ such that $r > \|L_t.\gamma\|/2$ for all $t \in [0,1]$.
	By our choice of $r$, there exists for each $t \in [0,1]$ exactly two points that satisfy the equation
	\begin{align}\label{eq:zpath}
		\| z - p_t(v_1)\|^2 = \| z - p_t(v_1) + L.\gamma \|^2 = r^2.
	\end{align}
	As $(p_t,L_t)$ is continuous, it follows that there exists a continuous path $z_t :[0,1] \rightarrow \mathbb{R}^2$ that satisfies equation (\ref{eq:zpath}).
	We now set $p'_t(v) := p_t$ for all $v \in V(G)$ and $p'_{v_0} := z_t(v_0)$.
\end{proof}

\section{Characterising fixed lattice flexible frameworks}\label{sec:fixedlattice}

In this section we shall prove the following result.

\begin{theorem}\label{thm:fixedNBAC}
	Let $G$ be a connected $\mathbb{Z}^k$-gain graph for $k\in \{1,2\}$. 
	Then there exists a placement-lattice $(p,L)$ of $G$ in $\mathbb{R}^2$ such that
	$(G,p,L)$ is a fixed lattice flexible full $k$-periodic framework if and only if either
	\begin{enumerate}[(i)]
		\item $G$ has a fixed lattice NBAC-colouring, or
		\item $G$ is balanced.
	\end{enumerate}
\end{theorem}

We shall first need to prove four results: Lemma \ref{lem:fixedNBAC1} for $k=1$, Lemma \ref{lem:fixedNBAC2} for $k=2$,
and Lemmas \ref{lem:fixedNBAC3} and \ref{lem:fixedNBAC4} for any $k \in \{1,2\}$.
The latter two will also explicitly show how to construct a fixed lattice flexible framework when either $G$ has a fixed lattice NBAC-colouring or is balanced.

\subsection{Necessary conditions for fixed lattice flexibility}

\begin{lemma}\label{lem:fixedactive1}
 	Let $(G,p,L)$ be a full $1$-periodic framework in $\mathbb{R}^2$ where $G$ is connected and unbalanced,
 	and let $\mathcal{C} \subset \Vsetcirc$ be an algebraic curve. 
 	Then every active NBAC-colouring of $\mathcal{C}$ is a fixed lattice NBAC-colouring.
\end{lemma}

\begin{proof}
	Let $\delta$ be an active NBAC-colouring of $\mathcal{C}$ generated by the valuation $\nu$ 
	and $\alpha \in \mathbb{R}$.	
	As $\mathcal{C} \subset \Vsetcirc$,
	we have $W_1 Z_1 = \| L.1\|^2$.
	Since $W_1 Z_1$ is constant,
	then $\nu(W_1) = - \nu(Z_1)$.
	We shall assume $\nu(W_1) >\alpha$ as the proof for the case $\nu(W_1) \leq \alpha$ follows from a similar method.

	Suppose there exists an almost red circuit $C$ of length $n$ in $G$ with $\delta(e_n) = \text{blue}$.
	As $\delta$ is an NBAC-colouring, we must have that $\gamma: = \psi (C) \neq 0$.
	It then follows that
	\begin{align*}
		\nu(W^{\gamma_n}_{v_1,v_n}) =\nu  \left( \sum_{j=1}^{n-1} W^{\gamma_j}_{v_j,v_{j+1}} + \gamma W_1 \right)
		\geq \min \left\{ \nu (W^{\gamma_j}_{v_j,v_{j+1}}), \nu(W_1) : j =1, \ldots, n-1 \right\} >\alpha,
	\end{align*}
	however this contradicts that $\nu (W^{\gamma_n}_{v_1,v_n}) \leq \alpha$.
	
	Now suppose there exists an unbalanced blue circuit $C$ of length $n$ in $G$ with $\gamma: = \psi (C)$.
	We note
	\begin{align*}
		\nu(-\gamma Z_1) 
		= \nu \left( \sum_{j=1}^{n} Z^{\gamma_j}_{v_j,v_{j+1}} \right) 
		\geq \min \left\{ \nu (Z^{\gamma_j}_{v_j,v_{j+1}}) : j =1, \ldots, n \right\}  \geq \alpha,
	\end{align*}
	contradicting that $\nu (Z_1) <\alpha$.
\end{proof}

We are now ready to prove our first necessity lemma.

\begin{lemma}\label{lem:fixedNBAC1}
 	Let $(G,p,L)$ be a full $1$-periodic framework in $\mathbb{R}^2$. 
 	If $(G,p,L)$ is fixed lattice flexible then either $G$ has an active fixed lattice NBAC-colouring, $G$ is balanced or $G$ is disconnected.
\end{lemma}

\begin{proof}
	Suppose $G$ is unbalanced and connected. 
	It follows from Proposition \ref{prop:balancedsub}
	that we may assume $G$ contains a spanning tree $T$ where every edge has trivial gain and $\tilde{e} \in T$,
	since by Proposition \ref{prop:activegain},
	if an equivalent graph to $G$ has an active NBAC-colouring then so does $G$.
	By Lemma \ref{lem:existalgcurv} (\ref{lem:existalgcurvitem2}),
	there exists an algebraic curve $\mathcal{C} \subset \Vsetcirc$.
	By Lemma \ref{lem:constantW},
	there exists $a:=(a_1, a_2, \alpha) \in E(G)$ such that $W^\alpha_{a_1,a_2}$ is not constant on $\mathcal{C}$.
	By Lemma \ref{lem:NBAC},
	there exists an active NBAC-colouring $\delta$ of $\mathcal{C}$,
	thus by Lemma \ref{lem:fixedactive1},
	$\delta$ is a fixed lattice NBAC-colouring as required.
\end{proof}

\begin{lemma}\label{lem:helpz}
	Let $(G,p,L)$ be a full $2$-periodic framework in $\mathbb{R}^2$,
	$\mathcal{C} \subset \Vsetcirc$ be an algebraic curve
	and suppose the function field $\mathbb{C}(\mathcal{C})$ has valuation $\nu$.
	Then the following holds:
		\begin{enumerate}[(i)]
			\item \label{lem:helpzitem1} $\nu(W_1) = - \nu(Z_1)$, $\nu(W_2) = - \nu(Z_2)$ and $\nu (W_1.Z_2 + W_2.Z_1) = 0$.
			\item \label{lem:helpzitem2} $\nu(W_1) = \nu(W_2)$ and $\nu(Z_1) = \nu(Z_2)$.
			\item \label{lem:helpzitem4} For all $\gamma \in \mathbb{Z}^2$,
			$\nu(\gamma Z) = - \nu(\gamma W)$.
			\item \label{lem:helpzitem3} For any $\gamma \in \mathbb{Z}^2$ and $\alpha \in \mathbb{R}$, 
			if $\nu(W_1) >\alpha$, 
			then $\nu(\gamma W) >\alpha$, 
			and if $\nu(W_1) \leq \alpha$, 
			then $\nu(\gamma W) \leq \alpha$.
		\end{enumerate}
\end{lemma}

\begin{proof}
	(\ref{lem:helpzitem1}): As $\mathcal{C} \subset \Vsetcirc$ then
	\begin{align*}
		W_1 Z_1 = \lambda(1,1)^2, \quad W_2 Z_2 = \lambda(2,2)^2, \quad W_1.Z_2 + W_2.Z_1 = 2 \lambda(1,2)^2,
	\end{align*}
	thus all are non-zero and constant.
	Since $\nu(f) = 0$ for all non-zero and constant $f\in \mathbb{C}(\mathcal{C})$,
	the result follows.
	
	(\ref{lem:helpzitem2}): We see that,
	\begin{align*}
		\nu(W_1.Z_2 + W_2.Z_1) \geq \min \left\{ \nu (W_1) - \nu (W_2), \nu (W_2) - \nu (W_1) \right\}
	\end{align*}
	with equality if $\nu (W_1) \neq \nu (W_2)$.
	If $\nu (W_1) \neq \nu (W_2)$,
	then $\nu(W_1.Z_2 + W_2.Z_1) < 0$,
	contradicting that $\nu(W_1.Z_2 + W_2.Z_1) = 0$, 
	thus $\nu (W_1) = \nu (W_2)$ (and similarly $\nu (Z_1) = \nu (Z_2)$).
	
	(\ref{lem:helpzitem4}):	Let $\gamma := (\gamma_1, \gamma_2)$ and define
	\begin{eqnarray*}
		g &:=& (\gamma_1 W_1 + \gamma_2 W_2) (\gamma_1 Z_1 + \gamma_2 Z_2) \\
		&=& \gamma_1^2 W_1 Z_1+ \gamma_2^2 W_2 Z_2 + \gamma_1 \gamma_2 (W_1 Z_2 + W_2 Z_1) \\
		&=& (\gamma_1 x_1 + \gamma_2 x_2)^2 + (\gamma_1 y_1 + \gamma_2 y_2)^2 .
	\end{eqnarray*}
	As $W_1 Z_1$, $W_2 Z_2 $ and $W_1 Z_2 + W_2 Z_1$ are all constant (since $\mathcal{C} \subset \Vsetcirc$),
	then $g$ is constant. 
	We further note that if $g = 0$ then the vectors $(x_1,y_1)$ and $(x_2,y_2)$ are linearly dependent for all points in $\mathcal{C}$.
	As this would contradict that $(G,p,L)$ is full,
	we have $\nu (g) = 0$.
	The required equality will now follow.
	
	(\ref{lem:helpzitem3}): Let $\gamma := (\gamma_1, \gamma_2)$.
	By (\ref{lem:helpzitem1}) and (\ref{lem:helpzitem2}),
	$\nu(W_1)=\nu(W_2)$.
	If $\nu(W_1) >\alpha$, 
	then
	\begin{align*}
		\nu(\gamma_1 W_1 + \gamma_2 W_2) 
		\geq \min \left\{ \nu (W_1) , \nu (W_2) \right\}  > \alpha,
	\end{align*}
	while if $\nu(W_1) \leq \alpha$, 
	then by (\ref{lem:helpzitem4}),
	\begin{align*}
		\nu(\gamma_1 W_1 + \gamma_2 W_2) 
		= - \nu(\gamma_1 Z_1 + \gamma_2 Z_2) 
		\leq - \min \left\{ \nu (Z_1) , \nu (Z_2) \right\} 
		= \max \left\{ \nu (W_1) , \nu (W_2) \right\} 
		\leq \alpha.
	\end{align*}
\end{proof}

\begin{lemma}\label{lem:fixedactive2}
 	Let $(G,p,L)$ be a full $2$-periodic framework in $\mathbb{R}^2$ where $G$ is connected graph with $\rank(G)=2$,
 	and let $\mathcal{C} \subset \Vsetcirc$ be an algebraic curve. 
 	Then every active NBAC-colouring of $\mathcal{C}$ is a fixed lattice NBAC-colouring.
\end{lemma}

\begin{proof}
	Let $\delta$ be an active NBAC-colouring of $\mathcal{C}$ with corresponding valuation $\nu$ 
	and non-zero $\alpha \in \mathbb{R}$.	
	By Lemma \ref{lem:helpz} (\ref{lem:helpzitem1}) and Lemma \ref{lem:helpz} (\ref{lem:helpzitem2}), 
	$\nu(W_1) = \nu(W_2)$, $\nu(Z_1) = - \nu(W_1)$ and $\nu(Z_2) = - \nu(W_2)$.
	We shall assume $\nu(W_1) >\alpha$ as the proof for the case $\nu(W_1) \leq \alpha$ follows from a similar method.
	
	Suppose there exists an almost red circuit $C$ of length $n$ in $G$ with $\gamma := \psi(C)$ and $\delta(e_n) = \text{blue}$.
	Then,
	\begin{align*}
		W^{\gamma_n}_{v_1,v_n} = \sum_{j=1}^{n-1} W^{\gamma_j}_{v_j,v_{j+1}} + \gamma W.
	\end{align*}
	By Lemma \ref{lem:helpz} (\ref{lem:helpzitem3}),
	\begin{align*}
		\nu(W^{\gamma_n}_{v_1,v_n})
		\geq \min \left\{ \nu (W^{\gamma_j}_{v_j,v_{j+1}}), \gamma W : j =1, \ldots, n-1 \right\} >\alpha,
	\end{align*}
	however this contradicts that $\nu (W^{\gamma_n}_{v_1,v_n}) \leq \alpha$.
	
	Now suppose there exists an unbalanced blue circuit $C$ of length $n$ in $G$ with $\gamma := \psi(C)$.
	We note
	\begin{align*}
		\nu(-\gamma Z) 
		= \nu \left( \sum_{j=1}^{n} Z^{\gamma_j}_{v_j,v_{j+1}} \right) 
		\geq \min \left\{ \nu (Z^{\gamma_j}_{v_j,v_{j+1}}) : j =1, \ldots, n \right\}  \geq \alpha.
	\end{align*}
	However, by Lemma \ref{lem:helpz} (\ref{lem:helpzitem4}) and Lemma \ref{lem:helpz} (\ref{lem:helpzitem3}) we have $\nu(-\gamma Z) <\alpha$, a contradiction.
\end{proof}

We are now ready to prove our final necessity lemma.
	
\begin{lemma}\label{lem:fixedNBAC2}
 	Let $(G,p,L)$ be a full $2$-periodic framework in $\mathbb{R}^2$. 
 	If $(G,p,L)$ is fixed lattice flexible then either $G$ has an active fixed lattice NBAC-colouring, $G$ is balanced or $G$ is disconnected.
\end{lemma}

\begin{proof}
	Suppose $\rank G =1$ and $G$ is connected. 
	We note that any $2$-periodic framework with rank 1 is fixed lattice flexible if and only if
	it is fixed lattice flexible when considered as a $1$-periodic framework.
	By Lemma \ref{lem:fixedNBAC1}, $G$ has an active fixed lattice NBAC-colouring.

	Suppose $\rank G =2$ and $G$ is connected.
	It follows from Proposition \ref{prop:balancedsub} and Proposition \ref{prop:activegain}
	that we may assume $G$ contains a spanning tree $T$ where every edge has trivial gain and $\tilde{e} \in T$.
	By Lemma \ref{lem:existalgcurv} (\ref{lem:existalgcurvitem2}),
	there exists an algebraic curve $\mathcal{C} \subset \Vsetcirc$.
	By Lemma \ref{lem:constantW},
	there exists $a:=(a_1, a_2, \alpha) \in E(G)$ such that $W^\alpha_{a_1,a_2}$ is not constant on $\mathcal{C}$.
	By Lemma \ref{lem:NBAC},
	there exists an active NBAC-colouring $\delta$ of $\mathcal{C}$,
	and by Lemma \ref{lem:fixedactive2},
	$\delta$ is a fixed lattice NBAC-colouring as required.
\end{proof}

\subsection{Constructing fixed lattice flexible frameworks}

\begin{lemma}\label{lem:fixedNBAC3}
	Let $G$ be a connected $\mathbb{Z}^k$-gain graph  for $k\in \{1,2\}$.
	If $G$ has a fixed lattice NBAC-colouring $\delta$
	then there exists a full placement-lattice $(p,L)$ of $G$ in $\mathbb{R}^2$ such that $(G,p,L)$ is fixed lattice flexible.
\end{lemma}

\begin{proof}
	The proof for $k=1$ is identical to that for $k=2$ except we have $L := [c ~ 0]^T$ for some irrational $c>0$.
	Due to this,
	we shall only prove the case for $k=2$.
	
	We may assume without loss of generality that $G^{\delta}_{\text{red}}$ is balanced;
	furthermore, by Proposition \ref{prop:balancedsub},
	we may assume all edges of $G^{\delta}_{\text{red}}$ have trivial gain.
	Let $R_1, \ldots, R_n$ be the red connected components and $B_1, \ldots, B_m$ be the blue connected components.
	As $\delta$ is a NBAC-colouring,
	there exists a blue edge $\tilde{e} \in E(G)$;
	by reordering the blue components we may assume the end points of $\tilde{e}$ lie in $B_1$.
	
	Choose any two points $c_1,c_2 >0$ so that $Ac_1 + Bc_2 \notin \mathbb{Z}$ for all $A,B \in \mathbb{Z} \setminus \{0\}$;
	it is sufficient that the set $\{c_1,c_2\}$ is algebraically independent over $\mathbb{Q}$.
	We define the placement-lattice $(p,L)$ of $G$ with 
	\begin{align*}
		p(v) := (x,y), 
		\qquad 
		L:= 
		\begin{bmatrix}
			c_1 & 0 \\
			0 & c_2 
		\end{bmatrix}
	\end{align*}
	for $v \in V(R_x) \cap V(B_y)$.

	We shall now prove $(G,p,L)$ is a well-defined $k$-periodic framework.
	Suppose there exists a red edge $e:= (v,w,\gamma) \in E(G)$ such that $p(v) = p(w) +L.\gamma$.
	As $e$ is red then $\gamma = (0,0)$,
	thus $p(v) = p(w)$. 
	It follows that for some $1\leq x \leq n$ and $1 \leq y \leq m$,
	we have $v,w \in V(R_{x}) \cap V(B_{y})$, 
	thus there exists a blue path $( e_1, \ldots, e_n)$ that starts at $w$ and ends at $v$.
	We note, however, that $(e_1,  \ldots, e_n, e)$ is an almost blue circuit,
	contradicting that $\delta$ is a fixed-lattice NBAC-colouring.
	
	 Now suppose there exists a blue edge $e:= (v,w,\gamma) \in E(G)$ with $\gamma = (\gamma_1,\gamma_2)$
	 such that $p(v) = p(w) +L.\gamma$,
	 then $p(v) = p(w) + (\gamma_1 c_1, \gamma_2 c_2)$.
	 By our choice of $c_1, c_2$ we must have $\gamma_1=\gamma_2 =0$,
	 thus $p(v) = p(w)$.
	 This implies that for some $1\leq x \leq n$ and $1 \leq y \leq m$,
	 we have $v,w \in V(R_{x}) \cap V(B_{y})$, 
	 and there exists a red path $( e_1, \ldots, e_n)$ that starts at $w$ and ends at $v$.
	 We note, however, that $(e_1, \ldots, e_n,  e)$ is a balanced almost red circuit
	 (since all red edges have trivial gain),
	 contradicting that $\delta$ is an NBAC-colouring.
	 It now follows that $(G,p,L)$ is a full $k$-periodic framework.
	
	Define the motion $(p_t,L_t)$, $t \in [0,1]$, 
	where for $p(v) = (x,y)$, 
	\begin{align*}
		p_t(v) := ( x + y \sin t, y \cos t),
	\end{align*}
	and $L_t =L$. 
	Choose any $t \in [0,1]$ and $e = (v,w, \gamma) \in E(G)$,
	with $\gamma =(\gamma_1, \gamma_2)$,
	$p(v) = (x,y)$ and $p(w) = (x',y')$.
	Suppose $\delta(e) = \text{red}$.
	Then $x'=x$, and $\gamma = (0,0)$ (as all red edges have trivial gain),
	and	it follows that
	\begin{align*}
		\|p_t(v) - p_t(w) - L_t.\gamma\|^2 = ( (y-y') \sin t)^2  + ( (y-y') \cos t)^2 = (y-y')^2.
	\end{align*}
	Now suppose $\delta(e) = \text{blue}$.
	Then $y'=y$ and we note that
	\begin{align*}
		\|p_t(v) - p_t(w) - L_t.\gamma\|^2 = (x-x' + \gamma_1 c_1)^2 + (\gamma_2 c_2)^2.
	\end{align*}
	It follows that $(G,p_t,L_t) \sim (G,p,L)$ for all $t \in [0,1]$,
	thus $(p_t,L_t)$ is a fixed lattice flex of $(G,p,L)$.
	As the edge $\tilde{e}$ is fixed then $(p_t,L_t)$ is non-trivial, 
	thus $(G,p,L)$ is fixed lattice flexible as required.
	We refer the reader to Figure \ref{fig:fixedNBACconstruction} for an example of the construction described.
\end{proof}

\begin{figure}[ht]
	\begin{center}
		\begin{tikzpicture}
			\node[lnode] (1) at (-1,-1) {$\boldsymbol{1}$};
			\node[lnode] (2) at (-1,1) {$\boldsymbol{2}$};
			\node[lnode] (3) at (1,-1) {$\boldsymbol{3}$};
			\node[lnode] (4) at (1,1) {$\boldsymbol{4}$};
			\node[lnode] (5) at (-2,0) {$\boldsymbol{5}$};
			\node[lnode] (6) at (2,0) {$\boldsymbol{6}$};
			\node[lnode] (7) at (0,1.5) {$\boldsymbol{7}$};
			\node(8) at (0,-2.1) {};
			
			\draw[redge] (1)edge(3);
			\draw[redge] (2)edge(4);
			\draw[redge] (5)edge(6);
			
			\draw[bedge] (2)edge(5);
			\draw[bedge] (1)edge(5);
			\path[bedge] (1) edge [bend left] node {} (2);
			\path[bedge,-latex] (1) edge [bend right] node {} (2);
			
			\draw[redge] (4)edge(7);
			\draw[bedge,-latex] (2)edge(7);
			
			\draw[bedge] (4)edge(6);
			\draw[bedge] (3)edge(6);
			\draw[bedge,-latex] (3)edge(4);
			
			\node[font=\scriptsize] at (-0.3,-0.5) {$(1,0)$};
			\node[font=\scriptsize] at (0.5,0.5) {$(0,1)$};
			\node[font=\scriptsize] at (-0.65,1.55) {$(1,1)$};
		\end{tikzpicture}\qquad\qquad
		\begin{tikzpicture}[scale=2]
			\node[lnode] (1) at (0,0) {$\boldsymbol{1}$};
			\node[lnode] (2) at (2,0) {$\boldsymbol{27}$};
			\node[lnode] (3) at (0,1) {$\boldsymbol{3}$};
			\node[lnode] (4) at (2,1) {$\boldsymbol{4}$};
			\node[lnode] (5) at (1,0) {$\boldsymbol{5}$};
			\node[lnode] (6) at (1,1) {$\boldsymbol{6}$};
			
			\draw[redge] (1)edge(3);
			\draw[redge] (2)edge(4);
			\draw[redge] (5)edge(6);
			
			\draw[bedge] (2)edge(5);
			\draw[bedge] (1)edge(5);
			\path[bedge] (1) edge [bend left] node {} (2);
			\path[bedge,-latex] (1) edge [bend right] node {} (2);
			
			\path[redge] (4) edge [bend left] node {} (2);

			
			\draw[bedge] (4)edge(6);
			\draw[bedge] (3)edge(6);
			\path[bedge,-latex] (3)edge[bend left] node {} (4);
			
			\path[bedge,-latex, every loop/.style={looseness=10}] (2) edge [loop right] node { } (2);

			\node[font=\scriptsize] at (1,-0.5) {$(1,0)$};
			\node[font=\scriptsize] at (2.3,-0.25) {$(1,1)$};
			\node[font=\scriptsize] at (1,1.5) {$(0,1)$};
		\end{tikzpicture}
	\end{center}
	\caption{(Left): A $\mathbb{Z}^2$-gain graph $G$ with a fixed lattice NBAC-colouring.
	(Right): The constructed full $2$-periodic framework $(G,p,L)$ in $\mathbb{R}^2$.
	We note that even though we place $(2)$ and $(7)$ at the same point in $\mathbb{R}^2$,
	$p(2) \neq p(7) +L.(1,1)$.}
	\label{fig:fixedNBACconstruction}
\end{figure}
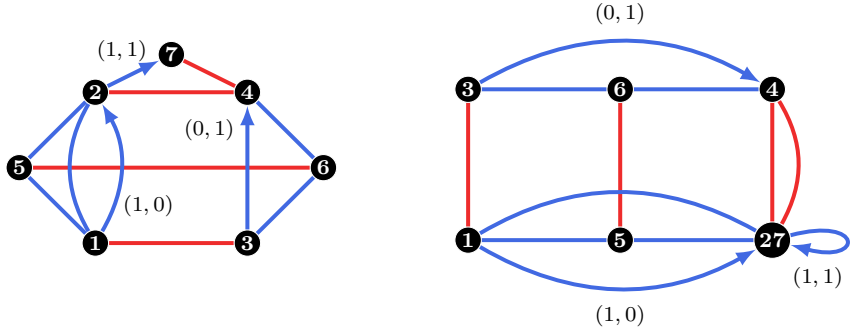

\begin{lemma}\label{lem:fixedNBAC4}
	Let $G$ be a $\mathbb{Z}^k$-gain graph  for $k\in \{1,2\}$.
	If $G$ is balanced,
	then there exists a full placement-lattice $(p,L)$ of $G$ in $\mathbb{R}^2$ such that $(G,p,L)$ is fixed lattice flexible.
\end{lemma}

\begin{proof}
	By Proposition \ref{prop:balancedsub},
	we may assume every edge of $G$ has trivial gain.
	Choose any injective map $p$ and any full lattice $L$.
	We may now define the fixed lattice flex $(p_t,L_t)$ for $t \in [0,1]$,
	where $p_t = p$ and 
	\begin{align*}
		L_t =
		\begin{bmatrix}
			\cos t & -\sin t \\
			\sin t & \cos t
		\end{bmatrix}
		L.
	\end{align*}
\end{proof}

We may now combine the results of this section to prove Theorem \ref{thm:fixedNBAC}

\begin{proof}[Proof of Theorem \ref{thm:fixedNBAC}]
	If $(G,p,L)$ is a fixed lattice flexible full $k$-periodic framework,
	then by Lemma \ref{lem:fixedNBAC1} if $k=1$ or Lemma \ref{lem:fixedNBAC2} if $k=2$,
	either $G$ has a fixed lattice NBAC-colouring or $G$ is balanced.
	
	If $G$ has a fixed lattice NBAC-colouring,
	then by Lemma \ref{lem:fixedNBAC3},
	there exists a fixed lattice flexible full $k$-periodic framework $(G,p,L)$ in $\mathbb{R}^2$.
	
	If $G$ is balanced,
	then by Lemma \ref{lem:fixedNBAC4},
	there exists a fixed lattice flexible full $k$-periodic framework $(G,p,L)$ in $\mathbb{R}^2$.
\end{proof}

\section{Characterising flexible 1-periodic frameworks}\label{sec:1NBAC}

In this section we shall prove the following theorem.

\begin{theorem}\label{thm:1NBAC}
	Let $G$ be a connected $\mathbb{Z}$-gain graph. 
	Then there exists a full placement-lattice $(p,L)$ of $G$ in $\mathbb{R}^2$ such that
	$(G,p,L)$ is a flexible full $1$-periodic framework if and only if either:
	\begin{enumerate}[(i)]
		\item $G$ has a fixed lattice NBAC-colouring,
		\item $G$ has a flexible 1-lattice NBAC-colouring, or
		\item $G$ is balanced.
	\end{enumerate}
\end{theorem}

Fortunately, 
much of the required work has been dealt with in Section \ref{sec:fixedlattice},
since fixed lattice flexible $1$-periodic frameworks are a subclass of flexible $1$-periodic frameworks.
Due to this, we only need to prove two results: 
a necessity lemma that proves a flexible $1$-periodic framework will have one of the required properties 
(see Lemma \ref{lem:1NBAC.2}),
and a construction lemma to prove that we can construct a flexible $1$-periodic framework given a graph with a 
flexible $1$-lattice NBAC-colouring (see Lemma \ref{lem:type1a}).

\subsection{Necessary conditions for 1-periodic flexibility}

\begin{lemma}\label{lem:1NBAC.2}
	Let $(G,p,L)$ be a $1$-periodic framework in $\mathbb{R}^2$
	with edge $(v,w,\gamma) \in E(G)$ for some $\gamma \neq 0$,
	$\mathcal{C} \subset \Vset$ be an algebraic curve and $\nu$ a valuation of $\mathbb{C}(\mathcal{C})$.
	Suppose $x_v - x_w$ and $y_v - y_w$ are constant on $\mathcal{C}$.
	Then $W_{v,w}^\gamma$ is constant if and only if $\mathcal{C} \subset \Vsetcirc$.
\end{lemma}

\begin{proof}
	We note that $W_{v,w}^\gamma$ is constant if and only if $Z_{v,w}^\gamma$ is also constant as
	$W_{v,w}^\gamma Z_{v,w}^\gamma$ is constant.
	As $x_v - x_w$ and $y_v-y_w$ are constant then
	$W_{v,w}^\gamma$ and $Z_{v,w}^\gamma$ are constant if and only if
	both $x_1 + i y_1$ and $x_1- i y_1$ are constant,
	which in turn is equivalent to both $x_1,y_1$ being constant.
	The result now follows.	
\end{proof}

\begin{lemma}\label{lem:1NBACparallel}
	Let $(G,p,L)$ be a full $1$-periodic framework in $\mathbb{R}^2$.
	Suppose that $(G,p,L)$ is flexible, $G$ is connected and unbalanced, and $G$ contains a pair of parallel edges $\tilde{e},\tilde{f}$.
	Then $G$ either has an active fixed lattice NBAC-colouring where $\tilde{e},\tilde{f}$ are the same colour,
	or $G$ has an active flexible 1-lattice NBAC-colouring where $\tilde{e},\tilde{f}$ are opposite colours.
\end{lemma}

\begin{proof}
	We may assume $\tilde{e}$ and $\tilde{f}$ are the pair of parallel edges on $\tilde{v},\tilde{w}$,
	with $\psi(\tilde{f} )= \mu \neq 0$.
	It follows from Proposition \ref{prop:balancedsub} and Proposition \ref{prop:activegain}
	that we may assume $G$ contains a spanning tree $T$ where every edge has trivial gain and $\tilde{e} \in T$.
	By Lemma \ref{lem:existalgcurv} (\ref{lem:existalgcurvitem2}),
	there exists an algebraic curve $\mathcal{C} \subset \Vset$.
	
	Suppose $\mathcal{C} \subset \Vsetcirc$.
	By Lemma \ref{lem:fixedNBAC1},
	$G$ has an active fixed lattice NBAC-colouring $\delta$.
	By Lemma \ref{lem:1NBAC.2},
	we note that we must have $\delta(\tilde{e})= \delta(\tilde{f})$.
	
	Now suppose $\mathcal{C} \not\subset \Vsetcirc$.
	By Lemma \ref{lem:1NBAC.2},
	$W_{\tilde{v},\tilde{w}}^{\mu}$ is not constant on $\mathbb{C}(\mathcal{C})$.
	Let $\nu$ the valuation of $\mathbb{C}(\mathcal{C})$
	and $\delta$ the NBAC-colouring given by Lemma \ref{lem:NBAC} with $a := \tilde{f}$.	
	By our choice of valuation,
	$\nu (W_{\tilde{v},\tilde{w}}^0) = 0$ and $\nu (W_{\tilde{v},\tilde{w}}^\mu) > 0$;
	it follows immediately that $\nu (Z_{\tilde{v},\tilde{w}}^0) = 0$ and $\nu (Z_{\tilde{v},\tilde{w}}^\mu) < 0$
	as both $W_{\tilde{v},\tilde{w}}^0 Z_{\tilde{v},\tilde{w}}^0$ and $W_{\tilde{v},\tilde{w}}^\mu Z_{\tilde{v},\tilde{w}}^\mu$
	are constant.
	As $\mu W_1 = W_{\tilde{v},\tilde{w}}^0 - W_{\tilde{v},\tilde{w}}^\mu$ then $\nu(W_1) = \nu (W_{\tilde{v},\tilde{w}}^0) = 0$.
	Similarly,
	as $\mu Z_1 = Z_{\tilde{v},\tilde{w}}^0 - Z_{\tilde{v},\tilde{w}}^\mu$ then $\nu (Z_1) = \nu (Z_{\tilde{v},\tilde{w}}^\mu) <0$.
	
	Suppose $G$ has an unbalanced monochromatic circuit $C$ of length $n$.
	If $C$ is red,
	then
	\begin{align*}
		\nu (W_1) = \nu ( - \psi (C) W_1) 
		= \nu \left( \sum_{j=1}^n W^{\gamma_j}_{v_j,v_{j+1}} \right) 
		\geq \min \left\{ \nu (W^{\gamma_j}_{v_j,v_{j+1}}) : 1\leq j \leq n \right\} >0,
	\end{align*}
	contradicting that $\nu(W_1) = 0$.
	If $C$ is blue,
	then
	\begin{align*}
		\nu (Z_1) = \nu ( - \psi (C) Z_1) 
		= \nu \left( \sum_{j=1}^n Z^{\gamma_j}_{v_j,v_{j+1}} \right) 
		\geq \min \left\{ \nu (Z^{\gamma_j}_{v_j,v_{j+1}}) : 1\leq j \leq n \right\} \geq 0,
	\end{align*}
	contradicting that $\nu(Z_1) < 0$.
	It now follows that $\delta$ is an active flexible 1-lattice NBAC-colouring.	
\end{proof}

We are now ready to state our necessity lemma.

\begin{lemma}\label{lem:1NBAC}
	Let $(G,p,L)$ be a full $1$-periodic framework in $\mathbb{R}^2$. 
 	If $(G,p,L)$ is flexible then $G$ either has an active fixed lattice NBAC-colouring,
 	an active flexible 1-lattice NBAC-colouring,
 	$G$ is balanced,
 	or $G$ is disconnected.
\end{lemma}

\begin{proof}
	We may suppose $G$ is connected and unbalanced.
	If $G$ contains a pair of parallel edges then the result holds by Lemma \ref{lem:1NBACparallel},
	thus we shall also assume that $G$ does not contain a pair of parallel edges.
	
	By Lemma \ref{lem:1NBAC.1},
	there exists a vertex addition $(G',p',L)$ of $(G,p,L)$ at $v_1$ by $1$
	such that $(G',p',L)$ has a non-trivial not fixed lattice flex;
	we shall define these new edges by $\tilde{e}, \tilde{f}$, with $\psi(\tilde{e})=0$ and $\psi(\tilde{f})=1$.
	As $G'$ contains a pair of parallel edges then by Lemma \ref{lem:1NBACparallel},
	either $G'$ has an active flexible 1-lattice NBAC-colouring $\delta'$ 
	with $\delta'(\tilde{e}) = \text{blue}$ and $\delta'(\tilde{f}) = \text{red}$,
	or $G'$ has an active fixed lattice lattice NBAC-colouring $\delta''$ with $\delta''(\tilde{e}) =\delta''(\tilde{f}) = \text{blue}$.	
	
	Suppose $G'$ has a colouring $\delta'$ as described above.
	Let $\delta$ be the colouring of $G$ with $\delta(e) := \delta'(e)$ for all $e \in E(G)$.
	We note that $\delta$ is a flexible 1-lattice NBAC-colouring 
	if and only if $\delta'$ is not monochromatic on the subgraph $G$ of $G'$.
	As $G$ is unbalanced,
	$\delta'$ cannot be monochromatic on $G$,
	thus $\delta$ is a flexible 1-lattice NBAC-colouring of $G$.
	
	Now suppose $G'$ has a colouring $\delta''$ as described above.
	Let $\delta$ be the colouring of $G$ with $\delta(e) := \delta''(e)$ for all $e \in E(G)$.
	We note that $\delta$ is a fixed lattice NBAC-colouring
	if and only if $\delta'$ is not monochromatic on the subgraph $G$ of $G'$.
	If $\delta'$ is monochromatic on $G$,
	then as $\delta'(\tilde{e}) =\delta'(\tilde{f}) = \text{blue}$ and $G$ is unbalanced,
	we must have $\delta(G) = \text{blue}$,
	however this would contradict that $\delta'(G') =\{ \text{red},\text{blue}\}$.
	It now follows that $\delta$ is a fixed lattice NBAC-colouring of $G$.
\end{proof}

\subsection{Constructing flexible frameworks from flexible 1-lattice NBAC-colourings}
%

\begin{lemma}\label{lem:type1a.1}
	Let $G$ be a $\mathbb{Z}$-gain graph with a flexible 1-lattice NBAC-colouring. 
	Then there exists $G' \approx G$ such that each blue edge has trivial gain and no red edge has trivial gain.
\end{lemma}

\begin{proof}
	As $G^\delta_{\text{blue}}$ is balanced,
	by Proposition \ref{prop:balancedsub},
	we may suppose all blue edges of $G$ have trivial gain.
	Let $B_1, \ldots, B_n$ be the blue components of $G$ and choose $\mu \in \mathbb{N}$ such that
	$\mu >|\gamma|$ for all $(v,w,\gamma) \in E(G)$.
	We now define
	\begin{align*}
		G' := \left( \prod_{i=1}^n \prod_{v \in B_i} \phi_v^{i\mu} \right) (G).
	\end{align*}
	We first note that any blue edge of $G'$ will have trivial gain since both of its ends will lie in the same blue component.
	Choose a red edge $(v,w,\gamma) \in E(G)$ and suppose $v \in B_i$ and $w \in B_j$.
	We note that
	\begin{align*}
		\left( \prod_{i=1}^n \prod_{v \in B_i} \phi_v^{i\mu} \right)(v,w,\gamma) 
		= \phi_v^{i\mu} \circ \phi_w^{j\mu}(v,w,\gamma) 
		= (v,w,\gamma + (i-j)\mu).
	\end{align*}
	As $\mu >|\gamma|$ and $i-j \in \mathbb{Z} $ then
	$\gamma + (i-j)\mu =0$ if and only if $\gamma = 0$ and $i=j$.
	If $v,w \in B_i$ and $\gamma = 0$ then there would exist a balanced almost blue circuit
	as $v,w$ are connected by a blue path and all blue edges of $G$ have trivial gain,
	thus $\gamma + (i-j)\mu \neq 0$ as required.
\end{proof}

\begin{lemma}\label{lem:type1a.2}
	Let $H$ be a balanced $\mathbb{Z}$-gain graph. 
	Then there exists a placement $q$ of $H$ in $\mathbb{Z}$ such that for all $(v,w,\gamma) \in E(H)$,
	$q(w) - q(v) = 2\gamma$.
\end{lemma}

\begin{proof}
	We may suppose without loss of generality that $H$ is connected.
	Choose a spanning tree $T$ of $H$.
	It is immediate that we may choose a placement $q$ of $T$ that satisfies the condition $q(w) - q(v) =2\gamma $
	for all $(v,w,\gamma) \in E(T)$.
	Choose an edge $e=(a,b,\mu) \in E(H) \setminus E(T)$,
	then there exists a path $(e_1,\ldots,e_{n-1})$ in $T$ with $e_i = (v_i,v_{i+1}, \gamma_i)$,
	$v_1=b$ and $v_n =a$.
	As $H$ is balanced,
	$\psi(e_1,\ldots,e_{n-1}) = -\mu$,
	thus by our choice of $q$,
	\begin{align*}
		q(b) - q(a)
		= -\left(\sum_{i=1}^{n-1} q(v_{i+1}) - q(v_i) \right)
		= -2\psi(e_1,\ldots,e_{n-1}) = 2\mu.
	\end{align*}
\end{proof}

We our now ready to prove our construction lemma.

\begin{lemma}\label{lem:type1a}
	Let $G$ be a $\mathbb{Z}$-gain graph with a flexible 1-lattice NBAC-colouring $\delta$. 
	Then there exists a full placement-lattice $(p,L)$ of $G$ in $\mathbb{R}^2$
	such that $(G,p,L)$ is a flexible full $1$-periodic framework.
\end{lemma}

\begin{proof}
	By Lemma \ref{lem:type1a.1}, 
	we may assume all blue edges of $G$ have trivial gain and all red edges have non-trivial gain.
	Let $R_1, \ldots, R_n$ be the red components of $G$ and define $E_j$ to be the set of edges $(v,w,\gamma)$
	in $G^\delta_{\text{red}}$ with $v,w \in R_j$.
	By Lemma \ref{lem:type1a.2}, 
	for each $R_j$ there exists a placement $q_j$ in $\mathbb{R}$ where $q_j(w) - q_j(v) = 2\gamma$
	for all $(v,w,\gamma) \in E_j$. We now define for each $t \in [0,2\pi]$ the full placement-lattice 
	$(p_t,L_t)$ of $G$ in $\mathbb{R}^2$,
	with
	\begin{align*}
		p_t(v) := ( q_j(v), j ), \qquad L_t .1:= (-2 + \cos t, \sin t)
	\end{align*}
	for $v \in R_j$ and $t \in [0,2\pi]$.
	We shall denote $(p,L) := (p_0,L_0)$.
	
	To see that 	$(p,L)$ is a well-defined placement-lattice,
	choose any $e=(v,w,\gamma)$ and suppose that $p(v) = p(w) + L.\gamma$.
	It follows that $v,w \in R_j$ and $q_j(v) - q_j(w)  =  \gamma$.
	If $\delta(e) = \text{red}$ then $\gamma \neq 0$,
	however this contradicts that $q_j(v) - q_j(w) = -2\gamma$.
	Suppose $\delta(e) = \text{blue}$.
	Since every blue edge has trivial gain, $\gamma = 0$.
	As $v,w \in R_j$,
	there exists a red path $(e_1,\ldots,e_{n-1})$ with $e_j = (v_j,v_{j+1}, \gamma_j) \in E_j$,
	$v_1=w$ and $v_n=v$.
	Since $q_j(v)=q_j(w)$,
	we have $\sum_{j=1}^{n-1} \gamma_j =0$.
	However, this implies $(e_1,\ldots,e_{n-1},e)$ is a balanced almost red circuit,
	contradicting that $\delta$ is a NBAC-colouring.
	
	Choose any $e = (v,w,\gamma)$.
	If $\delta(e) = \text{blue}$ then $\gamma = 0$. As $p_t = p$ then for each $t \in [0, 2 \pi]$,
	\begin{align*}
		\| p_t(v) - p_t(w) - L_t.\gamma \|^2 = \|p(v) - p(v) \|^2.
	\end{align*}
	If $\delta(e) = \text{red}$ then $v,w \in R_j$,
	thus for each $t \in [0, 2 \pi]$,
	\begin{align*}
		\| p_t(v) - p_t(w) - L_t.\gamma \|^2 = (-(q_j(w) - q_j(v) )+ 2 \gamma- \gamma  \cos t )^2 + (\gamma \sin t)^2 = \gamma^2.
	\end{align*}
	It follows that $(p_t,L_t)$  is a flex of $(G,p,L)$ as required.
	We refer the reader to Figure \ref{fig:1latticeNBACconstruction} for an example of the construction.
\end{proof}

\begin{figure}[ht]
	\begin{center}
		\begin{tikzpicture}
			\node[lnode] (1) at (-1,0) {\textbf{1}};
			\node[lnode] (2) at (-1,2) {\textbf{2}};
			\node[lnode] (3) at (1,0) {\textbf{3}};
			\node[lnode] (4) at (1,2) {\textbf{4}};
			\node[lnode] (5) at (-2,1) {\textbf{5}};
			\node[lnode] (6) at (2,1) {\textbf{6}};

			\draw[redge,-latex] (1)edge(3);
			\draw[redge,-latex] (2)edge(4);
			\draw[redge,-latex] (5)edge(6);
			
			\draw[bedge] (2)edge(5);
			\draw[bedge] (1)edge(5);
			\path[bedge] (1) edge [bend left] node {} (2);
			\path[redge,-latex] (1) edge [bend right] node {} (2);
			
			\draw[bedge] (4)edge(6);
			\draw[bedge] (3)edge(6);
			\path[redge,-latex] (3) edge [bend left] node {} (4);
			\path[bedge] (3) edge [bend right] node {} (4);
			
			\node[font=\scriptsize] at (-0.5,1.5) {$1$};
			\node[font=\scriptsize] at (0.5,1.5) {$1$};
			\node[font=\scriptsize] at (0,2.2) {$2$};
			\node[font=\scriptsize] at (0,0.2) {$2$};
			\node[font=\scriptsize] at (0,0.8) {$1$};
		\end{tikzpicture}\qquad\qquad
		\begin{tikzpicture}
			\node[lnode] (1) at (0,0) {\textbf{1}};
			\node[lnode] (2) at (2,0) {\textbf{2}};
			\node[lnode] (3) at (4,0) {\textbf{3}};
			\node[lnode] (4) at (6,0) {\textbf{4}};
			\node[lnode] (5) at (0,1.5) {\textbf{5}};
			\node[lnode] (6) at (2,1.5) {\textbf{6}};
			
			\draw[bedge] (2)edge(5);
			\draw[bedge] (1)edge(5);
			\draw[bedge] (1) edge(2);
			\draw[bedge] (4)edge(6);
			\draw[bedge] (3)edge(6);
			
			\draw[bedge] (3) edge (4);
			
			\path[redge,-latex] (1)edge [bend left] node {} (3);
			\path[redge,-latex] (2) edge [bend right] node {} (4);
			\draw[redge,-latex] (5)edge(6);

			\path[redge,-latex] (1) edge [bend right] node {} (2);
			\path[redge,-latex] (3) edge [bend left] node {} (4);

			\node[font=\scriptsize] at (2,0.8) {$2$};
			\node[font=\scriptsize] at (1,-0.6) {$1$};
			\node[font=\scriptsize] at (4,-0.45) {$2$};
			\node[font=\scriptsize] at (5.5,0.45) {$1$};
			\node[font=\scriptsize] at (0.9,1.7) {$1$};
		
		\end{tikzpicture}
	\end{center}
	\caption{(Left): A $\mathbb{Z}$-gain graph with a flexible $1$-lattice NBAC-colouring.
	(Right): The constructed full $1$-periodic framework in $\mathbb{R}^2$.}
	\label{fig:1latticeNBACconstruction}
\end{figure}
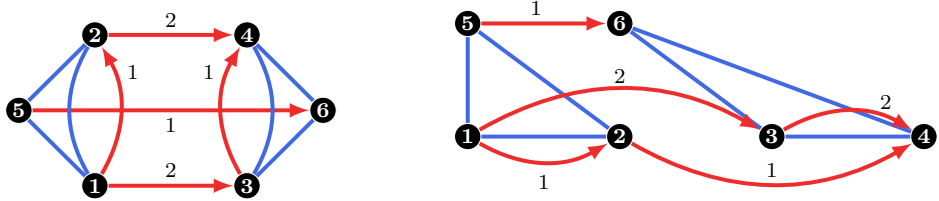

We are now ready to prove the main theorem of this section.

\begin{proof}[Proof of Theorem \ref{thm:1NBAC}]
	Suppose $(G,p,L)$ is flexible. 
	By Lemma \ref{lem:1NBAC},
	either $G$ is balanced,
	$G$ has a fixed lattice NBAC-colouring,
	or $G$ has a flexible 1-periodic NBAC-colouring.
	
	If $G$ is balanced,
	then by Lemma \ref{lem:fixedNBAC4},
	$G$ has a flexible full placement-lattice in $\mathbb{R}^2$.
	
	If $G$ has a fixed lattice NBAC-colouring,
	then by Lemma \ref{lem:fixedNBAC3},
	$G$ has a flexible full placement-lattice in $\mathbb{R}^2$.
	
	If $G$ has a flexible 1-lattice NBAC-colouring,
	then by Lemma \ref{lem:type1a},
	$G$ has a flexible full placement-lattice in $\mathbb{R}^2$.
\end{proof}

\section{Characterising flexible 2-periodic frameworks}\label{sec:2NBAC}

Unlike with $1$-periodic frameworks,
a full characterisation of $\mathbb{Z}^2$-gain graphs with flexible $2$-periodic full placements in the plane 
via NBAC-colourings is unknown.
We would conjecture the following.

\begin{conjecture}\label{con:weaktype3}
	Let $G$ be a connected $\mathbb{Z}^2$-gain graph. 
	Then there exists a full placement-lattice $(p,L)$ of $G$ in $\mathbb{R}^2$ such that
	$(G,p,L)$ is a flexible full $2$-periodic framework if and only if either:
	\begin{enumerate}[(i)]
		\item $G$ has a type 1 flexible 2-lattice NBAC-colouring,
		\item $G$ has a type 2 flexible 2-lattice NBAC-colouring,
		\item $G$ has a type 3 flexible 2-lattice NBAC-colouring,
		\item $G$ has a fixed lattice NBAC-colouring, or
		\item $\rank (G) <2$.
	\end{enumerate}
\end{conjecture}

We are able to obtain the required necessity lemma and most of the required construction lemmas,
however a construction of a flexible full $2$-periodic framework from a type 3 flexible 2-lattice NBAC-colouring is still currently unknown.
In this section we shall,
however, outline some partial results regarding $\mathbb{Z}^2$-gain graphs,
in particular,
Lemma \ref{lem:2NBAC}, Lemma \ref{lem:rank2}, Lemma \ref{lem:type1} and Lemma \ref{lem:type2}.
We shall discuss some other possible conjectures at the end of the section,
and later in Section \ref{sec:loops} we shall obtain analogues of Theorem \ref{thm:1NBAC} for certain types of graphs;
see Theorem \ref{thm:1loop2per} and Theorem \ref{thm:loopNBAC}.

\subsection{Necessary conditions for 2-periodic flexibility}

For any $\gamma = (a,b) \in \mathbb{Z}^2$,
we recall the notation $\gamma W := a W_1 + b W_2$ and $\gamma Z := a Z_1 + b Z_2$.

\begin{lemma}\label{lem:2NBAC.2}
	Let $(G,p,L)$ be a $2$-periodic framework in $\mathbb{R}^2$
	with edge $(v,w,\gamma) \in E(G)$ for some $\gamma = (\gamma_1,\gamma_2) \neq (0,0)$,
	$\mathcal{C} \subset \Vset$ be an algebraic curve and $\nu$ a valuation of $\mathbb{C}(\mathcal{C})$.
	Suppose $x_v - x_w$ and $y_v - y_w$ are constant on $\mathcal{C}$.
	If $W_{v,w}^\gamma$ is constant then
	\begin{align*}
		(\gamma_1 x_1 + \gamma_2 x_2)^2 + (\gamma_1 y_1 + \gamma_2 y_2)^2
	\end{align*}
	is constant.
\end{lemma}

\begin{proof}
	We note that $W_{v,w}^\gamma$ is constant if and only if $Z_{v,w}^\gamma$ is also constant as
	$W_{v,w}^\gamma Z_{v,w}^\gamma$ is constant.
	As $x_v - x_w$ and $y_v-y_w$ are constant then
	both $(\gamma_1 x_1 + \gamma_2x_2) + i(\gamma_1 y_1 + \gamma_2 y_2)$ 
	and $(\gamma_1x_1 + \gamma_2 x_2) - i(\gamma_1y_1 + \gamma_2 y_2)$ are constant.
	The result now follows from the observation that $(a +ib)(a-ib)=a^2+b^2$.
\end{proof}

\begin{lemma}\label{lem:2NBAC.1}
	Let $(G,p,L)$ be a full $2$-periodic framework in $\mathbb{R}^2$ and 
	$\mathcal{C} \subset \Vset$ be an algebraic curve.
	Suppose the function field $\mathbb{C}(\mathcal{C})$ has valuation $\nu$ and for some $\mu \in \mathbb{Z}^2 \setminus \{(0,0)\}$,
	\begin{align*}
		\nu \left( \mu W \right) = 0, \qquad \nu \left( \mu Z \right) < 0.
	\end{align*}
	Then one of the following cases holds:
	\begin{enumerate}[(i)]
		\item \label{lem:2NBAC.1item1} For all $\gamma \in \mathbb{Z}^2\setminus \{(0,0) \}$,
		\begin{align*}
			\nu \left( \gamma W \right) \leq 0, \qquad \nu \left(\gamma Z \right) < 0.
		\end{align*}
		\item \label{lem:2NBAC.1item2} There exists $\alpha, \beta \in \mathbb{Z}^2$ with at least one non-zero
		such that for all 
		$\gamma \in \mathbb{Z}^2 \setminus (\mathbb{Z}\alpha \cup \mathbb{Z}\beta )$,
		\begin{align*}
			\nu \left( \gamma W \right) \leq 0, \qquad \nu \left(\gamma Z \right) < 0,
		\end{align*}
		for all $\gamma \in \mathbb{Z} \alpha \setminus \{(0,0) \}$,
		\begin{align*}
			\nu \left(\gamma W \right) > 0, \qquad \nu \left(\gamma Z \right) < 0,
		\end{align*}
		and for all $\gamma \in \mathbb{Z} \beta \setminus \{(0,0) \}$,
		\begin{align*}
			\nu \left(\gamma W \right) \leq 0, \qquad \nu \left(\gamma Z \right) \geq 0,
		\end{align*}
		\item \label{lem:2NBAC.1item3} There exists $\alpha \in \mathbb{Z}^2 \setminus \{(0,0)\}$
		such that for all 
		$\gamma \in \mathbb{Z}^2 \setminus \mathbb{Z}\alpha$,
		\begin{align*}
			\nu \left( \gamma W \right) \leq 0, \qquad \nu \left(\gamma Z \right) < 0,
		\end{align*}
		and for all $\gamma \in \mathbb{Z}\alpha \setminus \{(0,0) \}$,
		\begin{align*}
			\nu \left(\gamma W \right) > 0, \qquad \nu \left( \gamma Z \right) \geq 0.
		\end{align*}
	\end{enumerate}
\end{lemma}

\begin{proof}
	Choose $\lambda \in \mathbb{Z}^2$ such that $\mu$ and $\lambda$ are linearly independent.
	
	If $\nu( \mu W) \neq \nu (\lambda W)$,
	then we note that for all $\gamma \in \mathbb{Z}^2 \setminus \{(0,0)\}$ with $\gamma = a \mu + b \lambda$,
	\begin{align*}
		\nu ( \gamma W) = \nu ((a \mu + b \lambda)W) = \min \left\{ \nu( \mu W), \nu (\lambda W) \right\} \leq 0;
	\end{align*}
	similarly,
	if $\nu( \mu Z) \neq \nu (\lambda Z)$,
	then $\nu (\gamma Z) <0$
	for all $\gamma \in \mathbb{Z}^2$.
	
	If $\nu( \mu W) = \nu (\lambda W)$,
	then there can exist $\alpha \in \mathbb{Z}^2 \setminus\{(0,0)\}$ that is pairwise independent of $\mu,\lambda$
	such that $\nu(\alpha W) >0$.
	We note that $\alpha$ is unique up to scalar multiplication,
	as if there exists $\gamma \in \mathbb{Z}^2 \setminus \mathbb{Z} \alpha$ such that $\nu (\gamma W) >0$ also,
	then we may choose $A,B \in \mathbb{R}$ such that $A\alpha + B\gamma = \mu$,
	and note that
	\begin{align*}
		\nu ( \mu W) \geq \min \left\{ \nu( \alpha W), \nu (\gamma W) \right\} > 0,
	\end{align*}
	contradicting that $\nu (\mu W) = 0$.
	Likewise,
	if  $\nu( \mu Z) = \nu (\lambda Z)$,
	then there can exist at most one $\beta \in \mathbb{Z}^2 \setminus\{0\}$ such that $\nu(\beta Z) \geq 0$.
	
	We now check the cases:
	\begin{itemize}
		\item Suppose $\nu( \mu W) \neq \nu (\lambda W)$ and $\nu( \mu Z) \neq \nu (\lambda Z)$.
		\begin{itemize}
			\item Case (\ref{lem:2NBAC.1item1}) holds if $\nu(\lambda W) , \nu (\lambda Z)<0$.
			\item Case (\ref{lem:2NBAC.1item2}) holds if $\nu(\lambda W)<0 < \nu (\lambda Z)$
			or $\nu(\lambda Z)<0 < \nu (\lambda W)$.
			\item Case (\ref{lem:2NBAC.1item3}) holds if $\nu(\lambda W) , \nu (\lambda Z) > 0$.
		\end{itemize}
		\item Suppose $\nu( \mu W) = \nu (\lambda W)$ and $\nu( \mu Z) \neq \nu (\lambda Z)$.
		\begin{itemize}
			\item Case (\ref{lem:2NBAC.1item1}) holds if $\alpha$ doesn't exist and $\nu ( \mu Z) < 0$.
			\item Case (\ref{lem:2NBAC.1item2}) holds otherwise.
		\end{itemize}
		\item Suppose $\nu( \mu W) \neq \nu (\lambda W)$ and $\nu( \mu Z) = \nu (\lambda Z)$.
		\begin{itemize}
			\item Case (\ref{lem:2NBAC.1item1}) holds if $\nu ( \mu W) < 0$ and $\beta$ doesn't exist.
			\item Case (\ref{lem:2NBAC.1item2}) holds otherwise.
		\end{itemize}
		\item Suppose $\nu( \mu W) = \nu (\lambda W)$ and $\nu( \mu Z) = \nu (\lambda Z)$.
		\begin{itemize}
			\item Case (\ref{lem:2NBAC.1item1}) holds if $\alpha, \beta$ don't exist.
			\item Case (\ref{lem:2NBAC.1item2}) holds if $\alpha$ exists and $\beta$ doesn't exist,
			$\alpha$ doesn't exist and $\beta$ exists, or if $\alpha, \beta$ exist and $\alpha \neq \beta$.
			\item Case (\ref{lem:2NBAC.1item3}) holds if $\alpha, \beta$ exist and $\alpha = \beta$.
		\end{itemize}
	\end{itemize}
\end{proof}

\begin{lemma}\label{lem:2NBAC.3}
	Let $(G,p,L)$ be a full $2$-periodic framework in $\mathbb{R}^2$
	and $\mathcal{C} \subset \Vset$ be an algebraic curve.
	Further,
	suppose $G$ contains a pair of parallel edges $(v,w,\gamma)$ and $(v,w,\gamma')$ such that 
	$\gamma - \gamma' = (\lambda_1,\lambda_2)$ and
	\begin{align*}
		(\lambda_1 x_1 + \lambda_2 x_2)^2 + (\lambda_1 y_1 + \lambda_2 y_2)^2
	\end{align*}
	is not constant on $\mathcal{C}$.
	Then one of the following holds:
	\begin{enumerate}[(i)]
		\item $G$ has an active type 1 flexible 2-lattice NBAC-colouring,
		\item $G$ has an active type 2 flexible 2-lattice NBAC-colouring, or
		\item $G$ has an active type 3 flexible 2-lattice NBAC-colouring.
	\end{enumerate}
\end{lemma}

\begin{proof}
	By our choice of $\tilde{e}$,
	we may assume $\tilde{e}$ and $\tilde{f}$ are the pair of parallel edges on $\tilde{v},\tilde{w}$,
	with $\psi(\tilde{f} )= \mu$ for some $\mu = (\mu_1,\mu_2) \in \mathbb{Z}^2 \setminus \{(0,0)\}$.
	It follows from Proposition \ref{prop:balancedsub} and Proposition \ref{prop:activegain}
	that we also may assume $G$ contains a spanning tree $T$ where every edge has trivial gain and $\tilde{e} \in T$.
	Since $\mu$ is the difference in gains of $\tilde{e},\tilde{f}$,
	then 
	\begin{align*}
		(\mu_1 x_1 + \mu_2 x_2)^2 + (\mu_1 y_1 + \mu_2 y_2)^2
	\end{align*}
	is not constant.
	By Lemma \ref{lem:2NBAC.2},
	$W_{\tilde{v},\tilde{w}}^{\mu}$ is not constant on $\mathbb{C}(\mathcal{C})$.
	Let $\nu$ be the valuation of $\mathbb{C}(\mathcal{C})$
	and $\delta$ be the active NBAC-colouring given by Lemma \ref{lem:NBAC} with $a := \tilde{f}$.
	
	We note that $\nu (W_{\tilde{v},\tilde{w}}^0) = 0$ and $\nu (W_{\tilde{v},\tilde{w}}^\mu) > 0$.
	As $\mu W = W_{\tilde{v},\tilde{w}}^0 - W_{\tilde{v},\tilde{w}}^\mu$
	then
	\begin{align*}
		\nu \left( \mu W \right) = \nu \left( W_{\tilde{v},\tilde{w}}^0 \right) = 0.
	\end{align*}
	Similarly,
	as $\mu Z = Z_{\tilde{v},\tilde{w}}^0 - Z_{\tilde{v},\tilde{w}}^\mu$
	then
	\begin{align*}
		\nu \left( \mu Z \right) = \nu \left( Z_{\tilde{v},\tilde{w}}^\mu \right) < 0.
	\end{align*}
	Let case (\ref{lem:2NBAC.1item1}), case (\ref{lem:2NBAC.1item2}) and case (\ref{lem:2NBAC.1item3})
	refer to the three possibilities given by Lemma \ref{lem:2NBAC.1}.
	We shall now proceed to prove that case (\ref{lem:2NBAC.1item1}) implies $G$ has a type 1 flexible 2-lattice NBAC-colouring,
	case (\ref{lem:2NBAC.1item2}) implies $G$ has either a type 1 or type 2 flexible 2-lattice NBAC-colouring, and
	case (\ref{lem:2NBAC.1item3}) implies $G$ has either a type 1, type 2 or a type 3 flexible 2-lattice NBAC-colouring.
	
	(Case (\ref{lem:2NBAC.1item1}) holds):
	Suppose $G$ has an unbalanced monochromatic circuit $C=$ of length $n$,
	and define $\gamma := \psi(C)$.
	If $C$ is red,
	then
	\begin{align*}
		\nu (\gamma W) 
		= \nu \left(- \sum_{j=1}^n W^{\gamma_j}_{v_j,v_{j+1}} \right) 
		\geq \min \left\{ \nu (W^{\gamma_j}_{v_j,v_{j+1}}) : 1\leq j \leq n \right\} >0,
	\end{align*}
	contradicting that $\nu(\gamma W) \leq 0$.
	If $C$ is blue,
	then
	\begin{align*}
		\nu (\gamma Z) 
		= \nu \left(- \sum_{j=1}^n Z^{\gamma_j}_{v_j,v_{j+1}} \right) 
		\geq \min \left\{ \nu (Z^{\gamma_j}_{v_j,v_{j+1}}) : 1\leq j \leq n \right\} \geq 0,
	\end{align*}
	contradicting that $\nu (\gamma Z)  < 0$.
	It now follows that $\delta$ is a type 1 flexible 2-lattice NBAC-colouring.	
	
	(Case (\ref{lem:2NBAC.1item2}) holds):
	Let $C$ be an unbalanced monochromatic circuit of length $n$ with 
	$\gamma := \psi(C)$.
	If $C$ is red and $\gamma \notin \mathbb{Z} \alpha$,
	then
	\begin{align*}
		\nu (\gamma W) 
		= \nu \left(- \sum_{j=1}^n W^{\gamma_j}_{v_j,v_{j+1}} \right) 
		\geq \min \left\{ \nu (W^{\gamma_j}_{v_j,v_{j+1}}) : 1\leq j \leq n \right\} >0,
	\end{align*}
	contradicting that $\nu(\gamma W) \leq 0$.
	Likewise, if $C$ is blue and $\gamma \notin \mathbb{Z} \beta$,
	then
	\begin{align*}
		\nu (\gamma Z) 
		= \nu \left(- \sum_{j=1}^n Z^{\gamma_j}_{v_j,v_{j+1}} \right) 
		\geq \min \left\{ \nu (Z^{\gamma_j}_{v_j,v_{j+1}}) : 1\leq j \leq n \right\} \geq 0,
	\end{align*}
	contradicting that $\nu (\gamma Z)  < 0$.
	
	Now let $C$ be an almost monochromatic circuit of length $n$	
	where $\delta(e_n) \neq \delta(e_i)$ for all $i \in \{1,\ldots,n-1\}$.
	If $C$ is almost red and $\psi(C) = c\alpha$ for some $c \in \mathbb{Z}$,
	then,
	\begin{align*}
		\nu (W_{v_{1},v_{n}}^{\gamma_{n}})
		= \nu \left(\sum_{j=1}^{n-1} W^{\gamma_j}_{v_j,v_{j+1}} + c \alpha W \right) 
		\geq \min \left\{\nu(\alpha W), \nu (W^{\gamma_j}_{v_j,v_{j+1}}) : 1\leq j \leq n-1 \right\} >0,
	\end{align*}
	contradicting that $\nu (W_{v_{1},v_{n}}^{\gamma_{n}}) \leq 0$.
	Similarly, if $C$ is almost blue and $\psi(C) = c\beta$ for some $c \in \mathbb{Z}$,
	then,
	\begin{align*}
		\nu (Z_{v_{1},v_{n}}^{\gamma_{n}})
		= \nu \left( \sum_{j=1}^{n-1} Z^{\gamma_j}_{v_j,v_{j+1}} + c \beta Z \right) 
		\geq \min \left\{\nu(\beta Z), \nu (Z^{\gamma_j}_{v_j,v_{j+1}}) : 1\leq j \leq n-1 \right\} >0,
	\end{align*}
	contradicting that $\nu (Z_{v_{1},v_{n}}^{\gamma_{n}}) \leq 0$.
	
	It now follows that if $G$ has no unbalanced monochromatic circuits then $\delta$ is a type 1 flexible 2-lattice NBAC-colouring,
	and if $G$ has an unbalanced monochromatic circuit then $\delta$ is a type 2 flexible 2-lattice NBAC-colouring.	
	
	(Case (\ref{lem:2NBAC.1item3}) holds): 
	Let $C$ be an unbalanced monochromatic circuit of length $n$ with	$\gamma:= \psi(C) \notin \mathbb{Z} \alpha$.
	If $C$ is red,
	then
	\begin{align*}
		\nu (\gamma W) 
		= \nu \left(- \sum_{j=1}^n W^{\gamma_j}_{v_j,v_{j+1}} \right) 
		\geq \min \left\{ \nu (W^{\gamma_j}_{v_j,v_{j+1}}) : 1\leq j \leq n \right\} >0,
	\end{align*}
	contradicting that $\nu(\gamma W) \leq 0$.
	Likewise, if $C$ is blue,
	then
	\begin{align*}
		\nu (\gamma Z) 
		= \nu \left(- \sum_{j=1}^n Z^{\gamma_j}_{v_j,v_{j+1}} \right) 
		\geq \min \left\{ \nu (Z^{\gamma_j}_{v_j,v_{j+1}}) : 1\leq j \leq n \right\} \geq 0,
	\end{align*}
	contradicting that $\nu (\gamma Z)  < 0$.
	
	Now let $C$ be an almost monochromatic circuit of length $n$	
	where $\psi(C) := c\alpha$ for some $c \in \mathbb{Z}$ and $\delta(e_n) \neq \delta(e_i)$ for all $i \in \{1,\ldots,n-1\}$.
	If $C$ is almost red, 
	then,
	\begin{align*}
		\nu (W_{v_{1},v_{n}}^{\gamma_{n}})
		= \nu \left( \sum_{j=1}^{n-1} W^{\gamma_j}_{v_j,v_{j+1}} + c \alpha W \right) 
		\geq \min \left\{\nu(\alpha W), \nu (W^{\gamma_j}_{v_j,v_{j+1}}) : 1\leq j \leq n-1 \right\} >0,
	\end{align*}
	contradicting that $\nu (W_{v_{1},v_{n}}^{\gamma_{n}}) \leq 0$.
	Similarly, if $C$ is almost blue,
	then,
	\begin{align*}
		\nu (Z_{v_{1},v_{n}}^{\gamma_{n}})
		= \nu \left(\sum_{j=1}^{n-1} Z^{\gamma_j}_{v_j,v_{j+1}} + c \alpha Z \right) 
		\geq \min \left\{\nu(\alpha Z), \nu (Z^{\gamma_j}_{v_j,v_{j+1}}) : 1\leq j \leq n-1 \right\} >0,
	\end{align*}
	contradicting that $\nu (Z_{v_{1},v_{n}}^{\gamma_{n}}) \leq 0$.
	
	It now follows that if $G$ has no unbalanced monochromatic circuits then $\delta$ is a type 1 flexible 2-lattice NBAC-colouring,
	if $G$ only has unbalanced monochromatic circuits for a single colour then $\delta$ is a type 2 flexible 2-lattice NBAC-colouring,
	and if $G$ has unbalanced monochromatic circuits for both colours then $\delta$ is a type 3 flexible 2-lattice NBAC-colouring.
\end{proof}

We are now ready for our necessity lemma.

\begin{lemma}\label{lem:2NBAC}
	Let $(G,p,L)$ be a full $2$-periodic framework in $\mathbb{R}^2$.
	If $(G,p,L)$ is flexible then one of the following holds:
	\begin{enumerate}[(i)]
		\item $G$ has an active type 1 flexible 2-lattice NBAC-colouring,
		\item $G$ has an active type 2 flexible 2-lattice NBAC-colouring,
		\item $G$ has an active type 3 flexible 2-lattice NBAC-colouring,
		\item $G$ has an active fixed lattice NBAC-colouring,
		\item $\rank (G) <2$, or
		\item $G$ is disconnected.
	\end{enumerate}
\end{lemma}

\begin{proof}
	Suppose $\rank (G) = 2$ and $G$ is connected.
	Choose any $\tilde{e} \in E(G)$.
	By Lemma \ref{lem:existalgcurv} (\ref{lem:existalgcurvitem2}),
	there exists an algebraic curve $\mathcal{C} \subset \Vset$.
	We now have three possible outcomes:
	\begin{enumerate}
		\item \label{poss0} $\mathcal{C} \subset \Vsetcirc$.
		\item \label{poss1} $G$ contains a pair of parallel edges $(v,w,\gamma)$ and $(v,w,\gamma')$ such that 
		$\gamma - \gamma' = (\lambda_1,\lambda_2)$ and
		\begin{align*}
			(\lambda_1 x_1 + \lambda_2 x_2)^2 + (\lambda_1 y_1 + \lambda_2 y_2)^2
		\end{align*}
		is not constant on $\mathcal{C}$.
		\item \label{poss2} Possibilities \ref{poss0} and \ref{poss1} do not hold.
	\end{enumerate}
	
	(Possibility \ref{poss0} holds):
	If $\mathcal{C} \subset \Vsetcirc$ then by Lemma \ref{lem:fixedNBAC1},
	$G$ has an active fixed lattice NBAC-colouring.
	
	(Possibility \ref{poss1} holds):
	By Lemma \ref{lem:2NBAC.3},
	$G$ has either an active type 1, type 2 or type 3 flexible $2$-lattice NBAC-colouring.	
	
	(Possibility \ref{poss2} holds):
	As  $\mathcal{C} \not\subset \Vsetcirc$,
	we may choose $\mu := (\mu_1,\mu_2) \in \mathbb{Z}^2$ such that
	\begin{align*}
		(\lambda_1x_1 + \lambda_2 x_2)^2 + (\lambda_1 y_1 + \lambda_2 y_2)^2
	\end{align*}
	is not constant.	
	By Lemma \ref{lem:1NBAC.1},
	there exists a vertex addition $(G',p',L)$ of $(G,p,L)$ at $v_1$ by $\lambda$
	such that $(G',p',L)$ has a non-trivial not fixed lattice flex.
	As Possibility \ref{poss1} holds for $(G',p',L)$,
	then by Lemma \ref{lem:2NBAC.3},
	$G'$ has an active type $k$ flexible 1-lattice NBAC-colouring $\delta'$ for some $k \in \{1,2,3\}$.
	
	Suppose $G'$ has a colouring $\delta'$ as described above.
	Let $\delta$ be the colouring of $G$ with $\delta(e) := \delta'(e)$ for all $e \in E(G)$.
	We note that $\delta$ is an active type $k'$ flexible 2-lattice NBAC-colouring for some $k' \in \{1,2,3\}$
	if and only if $\delta'$ is not monochromatic on the subgraph $G$ of $G'$.
	As $\rank (G)=2$ and $\delta'$ is a type $k$ flexible 2-lattice NBAC-colouring,
	$\delta'$ is not monochromatic on $G$,
	thus $G$ has an active type $k'$ flexible 2-lattice NBAC-colouring for some $k' \in \{1,2,3\}$.	
\end{proof}

\subsection{Constructing flexible frameworks: Low rank graphs}

Our first construction lemma is the simplest,
as the framework is not connected.

\begin{lemma}\label{lem:rank2}
	Let $G$ be a $\mathbb{Z}^2$-gain graph.
	If $\rank (G) <2$ then there exists a full placement-lattice $(p,L)$ of $G$ in $\mathbb{R}^2$ such that $(G,p,L)$ is flexible.
\end{lemma}

\begin{proof}
	If $\rank(G)=0$ then this holds by Lemma \ref{lem:fixedNBAC3},
	so we may suppose $\rank (G) =1$, i.e.~$\spann(G) =  \mathbb{Z} \alpha$ for some non-zero $\alpha \in \mathbb{Z}^2$.
	By Proposition \ref{prop:balancedsub},
	we may assume every edge of $G$ has gain in $\mathbb{Z} \alpha$.
	Choose any injective map $p$, any full lattice $L$, 
	and any element $\beta \in \mathbb{Z}^2$ that is linearly independent of $\alpha$.
	We may now define the fixed lattice flex $(p_t,L_t)$ for $t \in [0,2 \pi]$,
	where $p_t = p$ and 
	\begin{align*}
		L_t . \alpha := L.\alpha, \qquad L_t . \beta := (1+t) L.\beta.
	\end{align*}
\end{proof}

\subsection{Constructing flexible frameworks: Type 1 flexible 2-lattice NBAC-colourings}

We recall that a type $1$ flexible $2$-lattice NBAC-colouring is a NBAC-colouring $\delta$ where all monochromatic circuits
are balanced.

\begin{lemma}\label{lem:type1.1}
	Let $G$ be a $\mathbb{Z}^2$-gain graph with a type 1 flexible 2-lattice NBAC-colouring. 
	Then there exists $G' \approx G$ such that each blue edge has trivial gain and no red edge has trivial gain.
\end{lemma}

\begin{proof}
	The proof follows a similar method as Lemma \ref{lem:type1a.1}.
\end{proof}

\begin{lemma}\label{lem:type1.2}
	Let $H$ be a balanced $\mathbb{Z}^2$-gain graph with no multiple edges and no loops. 
	Then there exists a placement $q$ of $H$ in $\mathbb{Z}^2$ such that for all $(v,w,\gamma) \in E(H)$,
	$q(w) - q(v) = 2\gamma$.
\end{lemma}

\begin{proof}
	The proof follows the same method as Lemma \ref{lem:type1a.2}.
\end{proof}

We are now ready for our construction lemma for type 1 flexible $2$-lattice NBAC-colourings.
We note that it is essentially the same as the construction given in Lemma \ref{lem:type1a}.

\begin{lemma}\label{lem:type1}
	Let $G$ be a $\mathbb{Z}^2$-gain graph with a type 1 flexible 2-lattice NBAC-colouring $\delta$. 
	Then there exists a full placement-lattice $(p,L)$ of $G$ in $\mathbb{R}^2$
	such that $(G,p,L)$ is a flexible full $2$-periodic framework.
\end{lemma}

\begin{proof}
	By Lemma \ref{lem:type1.1}, 
	we may assume all blue edges of $G$ have trivial gain and all red edges have non-trivial gain.
	Let $R_1, \ldots, R_n$ be the red components of $G$ and define $E_j$ to be the set of edges $(v,w,\gamma)$
	in $G^\delta_{\text{red}}$ with $v,w \in R_j$.
	By Lemma \ref{lem:type1.2}, 
	for each $R_j$ there exists a placement $q_j$ in $\mathbb{R}^2$ where $q_j(w) - q_j(v) = 2\gamma$
	for all $(v,w,\gamma) \in E_j$. 
	By applying translations to each of the placements $q_j$,
	we may assume that for any blue edge $(v,w,0) \in E(G)$ with $v \in R_j$, $w \in R_k$ and $j\neq k$,
	we have	$q_j(v) \neq q_k(w)$.
	We now define for each $t \in [0,2\pi]$ the full placement-lattice $(p_t,L_t)$ of $G$ in $\mathbb{R}^2$,
	with 
	\begin{align*}
		L_t .(1,0):= (-2 + \cos t, \sin t), \qquad L_t .(0,1):= (\sin t, -2 - \cos t)
	\end{align*}
	and
	\begin{align*}
		p_t(v) :=  q_j(v)
	\end{align*}
	for $v \in R_j$.
	We shall denote $(p,L) := (p_0,L_0)$.
	
	To see that 	$(p,L)$ is a well-defined placement-lattice,
	choose any $e=(v,w,\gamma)$ with $\gamma = (\gamma_1,\gamma_2)$ and suppose that $p(v) = p(w) + L.\gamma$.
	If $\delta(e) = \text{red}$,
	then $\gamma \neq (0,0)$ and $v,w \in R_j$ for some $j$.
	It follows that 
	\begin{align*}
		-L.\gamma = q_j(w) - q_j(v) = 2\gamma.
	\end{align*}		
	However as $-L.\gamma = (\gamma_1, 3 \gamma_2)$,
	then $-L.\gamma = 2\gamma$ if and only if $\gamma = (0,0)$,
	contradicting that all red edges have non-trivial gain.
	If $\delta(e) = \text{blue}$,
	then $\gamma = (0,0)$.
	By our choice of placements $\{q_i : 1 \leq i \leq n\}$,
	we must have $v,w \in R_j$ for some $j$;
	furthermore, as $\gamma = (0,0)$ then	$q_j(v) = q_j(w)$.
	Let $(e_1,\ldots,e_{n-1})$ be a red path from $w$ to $v$ with $e_j = (v_j,v_{j+1}, \gamma_j) \in E_j$,
	$v_1=w$ and $v_n=v$.
	Since $q_j(v)=q_j(w)$,
	we have $\sum_{j=1}^{n-1} \gamma_j =0$.
	However this implies $(e_1,\ldots,e_{n-1},e)$ is a balanced almost red circuit,
	contradicting that $\delta$ is a type 1 flexible 2-lattice NBAC-colouring.
	
	Choose any edge $e = (v,w,\gamma)$ with $\gamma = (\gamma_1,\gamma_2)$.
	If $\delta(e) = \text{blue}$ then $\gamma = 0$. As $p_t = p$ then for each $t \in [0, 2 \pi]$,
	\begin{align*}
		\| p_t(v) - p_t(w) - L_t.\gamma \|^2 = \|p(v) - p(v) - L.\gamma\|^2.
	\end{align*}
	If $\delta(e) = \text{red}$ then $v,w \in R_j$,
	thus for each $t \in [0, 2 \pi]$,
	\begin{eqnarray*}
		\| p_t(v) - p_t(w) - L_t.\gamma \|^2 &=&
		(\gamma_1 \cos t + \gamma_2 \sin t)^2 + (\gamma_1 \sin t - \gamma_2 \cos t)^2 \\
		&=& \gamma_1^2 + \gamma_2^2
	\end{eqnarray*}
	It follows that $(p_t,L_t)$  is a flex of $(G,p,L)$ as required.
	We refer the reader to Figure \ref{fig:type1NBACconstruction} for an example of the construction.
\end{proof}

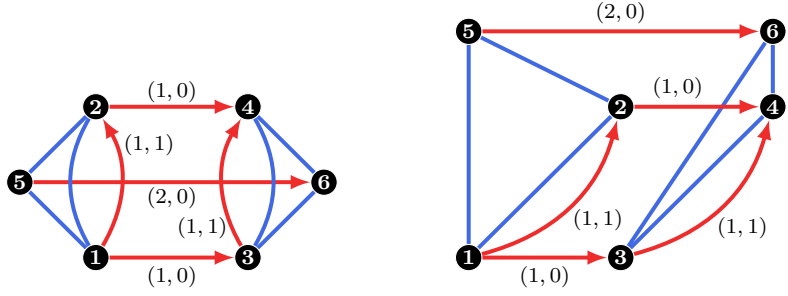
\begin{figure}[ht]
	\begin{center}
		\begin{tikzpicture}
			\node[lnode] (1) at (-1,0) {\textbf{1}};
			\node[lnode] (2) at (-1,2) {\textbf{2}};
			\node[lnode] (3) at (1,0) {\textbf{3}};
			\node[lnode] (4) at (1,2) {\textbf{4}};
			\node[lnode] (5) at (-2,1) {\textbf{5}};
			\node[lnode] (6) at (2,1) {\textbf{6}};

			\draw[redge,-latex] (1)edge(3);
			\draw[redge,-latex] (2)edge(4);
			\draw[redge,-latex] (5)edge(6);
			
			\draw[bedge] (2)edge(5);
			\draw[bedge] (1)edge(5);
			\path[bedge] (1) edge [bend left] node {} (2);
			\path[redge,-latex] (1) edge [bend right] node {} (2);
			
			\draw[bedge] (4)edge(6);
			\draw[bedge] (3)edge(6);
			\path[redge,-latex] (3) edge [bend left] node {} (4);
			\path[bedge] (3) edge [bend right] node {} (4);
			
			\node[font=\scriptsize] at (-0.3,1.5) {$(1,1)$};
			\node[font=\scriptsize] at (0.4,0.4) {$(1,1)$};
			\node[font=\scriptsize] at (0,2.2) {$(1,0)$};
			\node[font=\scriptsize] at (0,-0.25) {$(1,0)$};
			\node[font=\scriptsize] at (0,0.8) {$(2,0)$};
		\end{tikzpicture}\qquad\qquad
		\begin{tikzpicture}
			\node[lnode] (1) at (0,0) {\textbf{1}};
			\node[lnode] (2) at (2,2) {\textbf{2}};
			\node[lnode] (3) at (2,0) {\textbf{3}};
			\node[lnode] (4) at (4,2) {\textbf{4}};
			\node[lnode] (5) at (0,3) {\textbf{5}};
			\node[lnode] (6) at (4,3) {\textbf{6}};
			
			\draw[bedge] (2)edge(5);
			\draw[bedge] (1)edge(5);
			\draw[bedge] (1) edge(2);
			\draw[bedge] (4)edge(6);
			\draw[bedge] (3)edge(6);
			
			\draw[bedge] (3) edge (4);
			
			\path[redge,-latex] (1)edge(3);
			\path[redge,-latex] (2) edge (4);
			\draw[redge,-latex] (5)edge(6);

			\path[redge,-latex] (1) edge [bend right] node {} (2);
			\path[redge,-latex] (3) edge [bend right] node {} (4);
			
			\node[font=\scriptsize] at (1.7,0.5) {$(1,1)$};
			\node[font=\scriptsize] at (3.6,0.4) {$(1,1)$};
			\node[font=\scriptsize] at (2.75,2.25) {$(1,0)$};
			\node[font=\scriptsize] at (1,-0.25) {$(1,0)$};
			\node[font=\scriptsize] at (2,3.25) {$(2,0)$};
		
		\end{tikzpicture}
	\end{center}
	\caption{(Left): A $\mathbb{Z}^2$-gain graph with a type $1$ flexible $2$-lattice NBAC-colouring.
	(Right): The constructed full $2$-periodic framework in $\mathbb{R}^2$.}
	\label{fig:type1NBACconstruction}
\end{figure}

\subsection{Constructing flexible frameworks: Type 2 flexible 2-lattice NBAC-colourings}

We recall that a type $2$ flexible $2$-lattice NBAC-colouring is a NBAC-colouring $\delta$ where 
there exists $\alpha,\beta \in \mathbb{Z}^2$ such that:
\begin{itemize}
	\item either $\alpha,\beta$ are linearly independent or exactly one of $\alpha,\beta$ is equal to $(0,0)$,
	\item $\spann(G^\delta_\text{red})$ is a non-trivial subgroup of $\mathbb{Z}\alpha$, or $\alpha=(0,0)$ and $G^\delta_\text{red}$ is balanced,
	\item $\spann(G^\delta_\text{blue})$ is a non-trivial subgroup of $\mathbb{Z}\beta$, or $\beta=(0,0)$ and $G^\delta_\text{blue}$ is balanced,
	\item there are no almost red circuits with gain in $\mathbb{Z}\alpha$, and
	\item there are no almost blue circuits with gain in $\mathbb{Z}\beta$,
\end{itemize}

\begin{lemma}\label{lem:type2.1}
	Let $G$ be a $\mathbb{Z}^2$-gain graph and $\delta$ a type 2 flexible 2-lattice NBAC-colouring of $G$
	with $\alpha, \beta$ as described previously.
	Suppose $\alpha \neq (0,0)$.
	Then there exists $G' \approx G$ such that 
	each red edge has gain $a \alpha + b \beta$ for some $a,b \in \mathbb{Z}$ with $a\neq 0$,
	and each blue edge has gain $c\beta$ for some $c \in \mathbb{Z}$.
\end{lemma}

\begin{proof}
	As $\spann(G^\delta_{\text{blue}}) = \mathbb{Z} \beta$,
	by Proposition \ref{prop:balancedsub},
	we may suppose all blue edges of $G$ have gain in $\mathbb{Z}\beta$.
	Let $B_1, \ldots, B_n$ be the blue components of $G$ and choose $N \in \mathbb{N}$ such that
	$N >|a|$ for all $(v,w,\gamma) \in E(G)$ with $\gamma = a\alpha + b \beta$.
	We now define the gain equivalent graph
	\begin{align*}
		G' := \left( \prod_{i=1}^n \prod_{v \in B_i} \phi_v^{iN \alpha} \right) (G).
	\end{align*}
	We first note that any blue edge of $G'$ will have gain in $\mathbb{Z} \beta$
	since both of its ends will lie in the same blue component.
	Choose a red edge $(v,w,\gamma) \in E(G)$ with $\gamma = a\alpha + b \beta$ and suppose $v \in B_i$ and $w \in B_j$.
	We note that
	\begin{eqnarray*}
		\left( \prod_{i=1}^n \prod_{v \in B_i} \phi_v^{iN \alpha} \right)(v,w,\gamma)
		&=& \phi_v^{iN \alpha} \circ \phi_w^{jN \alpha}(v,w,\gamma) \\
		&=& (v,w,(N(i-j) + a) \alpha + b \beta).
	\end{eqnarray*}
	As $N >|a|$ and $i-j \in \mathbb{Z}$,
	we have $N(i-j) + a =0$	if and only if $a = 0$ and $i=j$.
	If this holds,
	then as $v,w \in B_i$,
	we can define an almost blue circuit containing $v$ with red edge $(v,w,b\beta)$ and gain in $\mathbb{Z} \beta$
	(as every blue edge has gain in $\mathbb{Z} \beta$),
	contradicting that $\delta$ is a type $2$ flexible $2$-lattice NBAC-colouring.
	It now follows that $a \neq 0$ as required.
\end{proof}

\begin{lemma}\label{lem:type2.2}
	Let $\alpha,\beta \in \mathbb{Z}^2$ be linearly independent
	and let $H$ be a $\mathbb{Z}^2$-gain graph where $\spann(H)$ is a subgroup of $\mathbb{Z} \alpha$.
	Then there exists a placement $q$ of $H$ in $\mathbb{Z}$ such that for all $(v,w,a \alpha +b \beta) \in E(H)$,
	$q(v) - q(w) = b$.
\end{lemma}

\begin{proof}
	Define the $\mathbb{Z}$-gain graph $H'$ with vertex set $V(H') := V(H)$ and edge set 
	\begin{align*}
		E(H') := \{ (v,w,b\beta) : (v,w,a \alpha + b\beta) \in E(H) \};
	\end{align*}
	we delete any loops with trivial gain that may arrive, and note that multiple edges may become a single edge.
	By Lemma \ref{lem:type1a.2},
	we may define a placement $q'$ of $H'$ in $\mathbb{Z}$ such that $q'(v) - q'(w) = -2 b$ for all $(v,w,b\beta) \in E(H')$.
	We now define $q$ to be the placement of $H$ where $q(v) := -\frac{1}{2} q'(v)$.
\end{proof}

We are now ready for our construction lemma for type 2 flexible 2-lattice NBAC-colourings.

\begin{lemma}\label{lem:type2}
	Let $G$ be a $\mathbb{Z}^2$-gain graph with a type 2 flexible $2$-lattice NBAC-colouring $\delta$. 
	Then there exists a full placement-lattice $(p,L)$ of $G$ in $\mathbb{R}^2$
	such that $(G,p,L)$ is a flexible full $2$-periodic framework.
\end{lemma}

\begin{proof}
	Without loss of generality we may assume $\spann (G^\delta_{\text{red}}) = \mathbb{Z} \alpha$
	and $\spann (G^\delta_{\text{blue}}) = \mathbb{Z} \beta$,
	with $\alpha \neq (0,0)$.
	If $\beta = (0,0)$ then $\delta$ is a fixed-lattice NBAC-colouring and the result holds by Lemma \ref{lem:fixedNBAC2},
	thus we may assume $\alpha,\beta$ are linearly independent.	
	
	By Lemma \ref{lem:type2.1}, 
	we may assume all red edges have gain $a \alpha + b\beta$ for some $a,b \in \mathbb{Z}$ with $a\neq 0$,
	and all blue edges have gain $c\beta$ for some $c \in \mathbb{Z}$.
	Let $R_1, \ldots, R_n$ be the red components of $G$ and define $E_j$ to be the set of edges $(v,w,\gamma)$
	in $G^\delta_{\text{red}}$ with $v,w \in R_j$.
	By Lemma \ref{lem:type2.2}, 
	for each $R_j$ there exists a placement $q_j$ in $\mathbb{R}$ where $q_j(v) - q_j(w) = b$
	for all $(v,w,\gamma) \in E_j$ with $\gamma =a \alpha + b\beta$. 
	We now define for each $t \in [0,2\pi]$ the full placement-lattice $(p_t,L_t)$ of $G$ in $\mathbb{R}^2$,
	with 
	\begin{align*}
		L_t .\alpha := (\sin t, \cos t), \qquad L_t .\beta:= (1,0)
	\end{align*}
	and
	\begin{align*}
		p_t(v) :=  (q_j(v) , j)
	\end{align*}
	for $v \in R_j$.
	We shall denote $(p,L) := (p_0,L_0)$.
	
	To see that 	$(p,L)$ is a well-defined placement-lattice,
	choose any $e=(v,w,\gamma)$ and suppose that $p(v) = p(w) + L.\gamma$.
	If $\delta(e) = \text{red}$,
	then $\gamma = a\alpha + b \beta$ for some $a,b \in \mathbb{Z} \setminus \{0\}$ and $v,w \in R_j$ for some $j$.
	We note
	\begin{align*}
		p(v) = (q_j(v),j) = (q_j(w) + b, j + a) = p(w) + L.\gamma,
	\end{align*}		
	which implies $a=0$,
	a contradiction.
	If $\delta(e) = \text{blue}$,
	then $\gamma = b \beta$ for some $b \in \mathbb{Z}$.
	If $v\in R_j$ and $w \in R_k$ then
	\begin{align*}
		p(v) = (q_j(v),j) = (q_k(w) + b, k) = p(w) + L.\gamma,
	\end{align*}		
	thus $j=k$.
	Let $P:=(e_1,\ldots,e_{n-1})$ be a red path from $w$ to $v$ with $e_i = (v_i,v_{i+1}, \gamma_i) \in E_j$,
	$v_1=w$, $v_n=v$, and $\gamma_i = a_i \alpha + b_i \beta$.
	Define $C := (e_1,\ldots,e_{n-1}, e)$.
	As
	\begin{align*}
		b = q_j(v)-q_j(w) = \sum_{i=1}^{n-1} (q_j(v_{i+1}) - q_j(v_{i})) = - \sum_{i=1}^{n-1} b_i,
	\end{align*}
	we have $\psi(C) = a \alpha$ for some $a \in \mathbb{Z}$.
	This contradicts that $\delta$ is a type 2 flexible 2-lattice NBAC-colouring,
	as $C$ is an almost red circuit with $\psi(C) \in \mathbb{Z} \alpha$.
	
	Choose any edge $e = (v,w,\gamma)$ with $\gamma = a \alpha + b \beta$.
	If $\delta(e) = \text{blue}$ then $a = 0$. 
	As $p_t = p$ and $L_t.\beta = (1,0)$ then $\| p_t(v) - p_t(w) - L_t.\gamma \|^2$ is constant.
	If $\delta(e) = \text{red}$ then $v,w \in R_j$,
	thus for each $t \in [0, 2 \pi]$,
	\begin{align*}
		\| p_t(v) - p_t(w) - L_t.\gamma \|^2 =
		(q_j(v)- q_j(w) - b - a \cos t)^2 + (a \sin t)^2 = a^2.
	\end{align*}
	It follows that $(p_t,L_t)$  is a flex of $(G,p,L)$ as required.
	We refer the reader to Figure \ref{fig:type2NBACconstruction} for an example of the construction.
\end{proof}

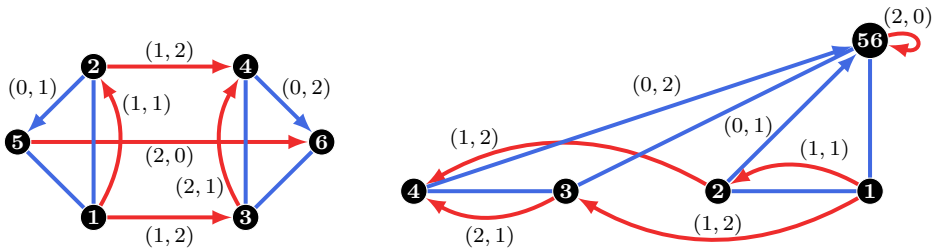
\begin{figure}[ht]
	\begin{center}
		\begin{tikzpicture}
			\node[lnode] (1) at (-1,0) {\textbf{1}};
			\node[lnode] (2) at (-1,2) {\textbf{2}};
			\node[lnode] (3) at (1,0) {\textbf{3}};
			\node[lnode] (4) at (1,2) {\textbf{4}};
			\node[lnode] (5) at (-2,1) {\textbf{5}};
			\node[lnode] (6) at (2,1) {\textbf{6}};

			\draw[redge,-latex] (1)edge(3);
			\draw[redge,-latex] (2)edge(4);
			\draw[redge,-latex] (5)edge(6);
			
			\draw[bedge,-latex] (2)edge(5);
			\draw[bedge] (1)edge(5);
			\draw[bedge] (1) edge (2);
			\path[redge,-latex] (1) edge [bend right] node {} (2);
			
			\draw[bedge,-latex] (4)edge(6);
			\draw[bedge] (3)edge(6);
			\path[redge,-latex] (3) edge [bend left] node {} (4);
			\draw[bedge] (3) edge (4);
			
			\node[font=\scriptsize] at (-0.3,1.5) {$(1,1)$};
			\node[font=\scriptsize] at (0.4,0.4) {$(2,1)$};
			\node[font=\scriptsize] at (0,2.2) {$(1,2)$};
			\node[font=\scriptsize] at (0,-0.25) {$(1,2)$};
			\node[font=\scriptsize] at (0,0.8) {$(2,0)$};
			\node[font=\scriptsize] at (-1.8,1.7) {$(0,1)$};
			\node[font=\scriptsize] at (1.8,1.7) {$(0,2)$};
		\end{tikzpicture}\qquad
		\begin{tikzpicture}
			\node[lnode] (1) at (0,0) {\textbf{1}};
			\node[lnode] (2) at (-2,0) {\textbf{2}};
			\node[lnode] (3) at (-4,0) {\textbf{3}};
			\node[lnode] (4) at (-6,0) {\textbf{4}};
			\node[lnode] (56) at (0,2) {\textbf{56}};
			
			\path[redge,-latex] (1) edge [bend left] node {} (3);
			\path[redge,-latex] (2) edge [bend right] node {} (4);
			\path[redge,-latex, every loop/.style={looseness=10}] (56) edge [loop right] node { } (56);
			
			\draw[bedge,-latex] (2)edge(56);
			\draw[bedge] (1)edge(56);
			\draw[bedge] (1) edge (2);
			\path[redge,-latex] (1) edge [bend right] node {} (2);
			
			\draw[bedge,-latex] (4)edge(56);
			\draw[bedge] (3)edge(56);
			\path[redge,-latex] (3) edge [bend left] node {} (4);
			\draw[bedge] (3) edge (4);
			
			\node[font=\scriptsize] at (-0.6,0.55) {$(1,1)$};
			\node[font=\scriptsize] at (-5,-0.6) {$(2,1)$};
			\node[font=\scriptsize] at (-5.2,0.7) {$(1,2)$};
			\node[font=\scriptsize] at (-2,-0.45) {$(1,2)$};
			\node[font=\scriptsize] at (0.5,2.3) {$(2,0)$};
			\node[font=\scriptsize] at (-1.6,0.85) {$(0,1)$};
			\node[font=\scriptsize] at (-2.8,1.4) {$(0,2)$};
		\end{tikzpicture}
	\end{center}
	\caption{(Left): A $\mathbb{Z}^2$-gain graph with a type $2$ flexible $2$-lattice NBAC-colouring ($\alpha = (1,0)$, $\beta = (0,1)$).
	(Right): The constructed full $2$-periodic framework in $\mathbb{R}^2$.}
	\label{fig:type2NBACconstruction}
\end{figure}

\subsection{Conjectures regarding type 3 flexible 2-lattice NBAC-colourings}

We recall that a NBAC-colouring $\delta$ of a $\mathbb{Z}^2$-gain graph $G$ is a type $3$ flexible $2$-lattice NBAC-colouring 
if there exists $\alpha \in \mathbb{Z}^2 \setminus \{(0,0)\}$ such that
\begin{itemize}
	\item $\spann(G^\delta_\text{red})$ and $\spann(G^\delta_\text{red})$ are non-trivial subgroups of $\mathbb{Z}\alpha$, and
	\item there are no almost monochromatic circuits with gain in $\mathbb{Z}\alpha$.
\end{itemize}

It is an open question whether the existence of a type 3 flexible 2-lattice NBAC-colouring implies the existence of a flexible placement of a $\mathbb{Z}^2$-gain graph in $\mathbb{R}^2$.
As this is the case for all other types of flexible $2$-lattice NBAC-colourings,
we would conjecture the following.

\begin{conjecture}\label{con:type3}
	Let $G$ be a $\mathbb{Z}^2$-gain graph with type 3 flexible 2-lattice NBAC-colouring.
	Then there exists a full placement-lattice $(p,L)$ of $G$ in $\mathbb{R}^2$
	such that $(G,p,L)$ is a flexible full $2$-periodic framework.
\end{conjecture}

All examples of $\mathbb{Z}^2$-gain graphs with a type 3 flexible 2-lattice NBAC-colouring discovered so far will also have either
a type 1 or type 2 flexible 2-lattice NBAC-colouring, a fixed lattice NBAC-colouring,
or have a low rank.
Due to this,
we would also conjecture the following.

\begin{conjecture}\label{con:notype3}
	Let $G$ be a $\mathbb{Z}^2$-gain graph with type 3 flexible 2-lattice NBAC-colouring.
	Then $G$ has either a type 1 or type 2 flexible 2-lattice NBAC-colouring,
	$G$ has a fixed lattice NBAC-colouring,
	or $\rank (G) <2$.
\end{conjecture}

If Conjecture \ref{con:type3} is true,
then by Lemma \ref{lem:2NBAC}, Lemma \ref{lem:type1}, Lemma \ref{lem:type2}, Lemma \ref{lem:fixedNBAC3}, 
and Lemma \ref{lem:rank2},
we can deduce that Conjecture \ref{con:weaktype3} would be also be true. 
If Conjecture \ref{con:notype3} is true,
then we obtain the slightly stronger result.

\begin{conjecture}\label{con:strongtype3}
	Let $G$ be a connected $\mathbb{Z}^2$-gain graph. 
	Then there exists a full placement-lattice $(p,L)$ of $G$ in $\mathbb{R}^2$ such that
	$(G,p,L)$ is a flexible full $2$-periodic framework if and only if either:
	\begin{enumerate}[(i)]
		\item $G$ has a type 1 flexible 2-lattice NBAC-colouring,
		\item $G$ has a type 2 flexible 2-lattice NBAC-colouring,
		\item $G$ has a fixed lattice NBAC-colouring, or
		\item $\rank (G) <2$.
	\end{enumerate}
\end{conjecture}

\section{Special cases of flexible 2-periodic frameworks}\label{sec:loops}

We shall now focus on $2$-periodic frameworks with loops.
With this added assumption,
we can fully characterise whether a $\mathbb{Z}^2$-gain graph has a flexible placement-lattice by observing the graph's 
NBAC-colourings.

\subsection{2-periodic frameworks with loops}

\begin{lemma}\label{lem:loop2per.1}
	Let $(G,p,L)$ be a $k$-periodic framework in $\mathbb{K}^d$
	and suppose $G$ has a loop $(w,w,\alpha)$.
	If $G'$ is the $\mathbb{Z}^k$-gain graph with 
	\begin{align*}
		V(G') := V(G), \quad E(G') :=E(G) \cup  \{ (v,v,c \alpha) : v \in V(G), c \in \mathbb{N}\},
	\end{align*}
	then,
	\begin{align*}
		\mathcal{V}_{\mathbb{K}}(G',p,L) = \mathcal{V}_{\mathbb{K}}(G,p,L)
	\end{align*}
\end{lemma}

\begin{proof}
	First note that $\mathcal{V}_{\mathbb{K}}(G',p,L) \subset \mathcal{V}_{\mathbb{K}}(G,p,L)$.
	Choose any placement-lattice $(p',L') \in \mathcal{V}_{\mathbb{K}}(G,p,L)$.
	As $(w,w,\alpha) \in E(G)$ then
	\begin{align*}
		\|L'.\alpha \|^2 = \|p'(w) - p'(w) - L'.\alpha \|^2 = \|p(w) - p(w) - L.\alpha \|^2 = \|L.\alpha \|^2.
	\end{align*}
	We note that for any $v \in V(G)$ and non-zero $c \in \mathbb{Z}$,
	\begin{align*}
		\|p'(v) - p'(v) - L'.c\alpha \|^2 = c^2 \|L'.\alpha \|^2 = c^2 \|L.\alpha \|^2 = \|p(v) - p(v) - L.c\alpha \|^2,
	\end{align*}
	thus $(p',L') \in \mathcal{V}_{\mathbb{K}}(G',p,L)$ as required.
\end{proof}

\begin{lemma}\label{lem:loop2per.2}
	Let $G$ be a connected $\mathbb{Z}^2$-gain graph 
	and suppose that there exists some $\alpha \in \mathbb{Z}^2\setminus \{(0,0)\}$
	such that for every vertex $v \in V(G)$ and $c \in \mathbb{N}$, we have $(v,v,c \alpha) \in E(G)$.
	If $\delta$ is a NBAC-colouring of $G$, then every loop with gain $c \alpha$ for some $c \in \mathbb{N}$ is the same colour.
\end{lemma}

\begin{proof}
	We first note that every loop at a vertex must have the same colour.
	To see this, suppose there exists a loop $(v,v, c\alpha)$ for some integer $c>1$, where $\delta(v,v,n\alpha) \neq \delta(v,v,\alpha)$.
	This would imply the circuit
	\begin{align*}
		( \overbrace{ (v,v,-\alpha),  \ldots ,  (v,v,-\alpha)}^{c \text{ times}} , (v,v,c\alpha) )
	\end{align*}
	is a balanced almost monochromatic circuit,
	contradicting that $\delta$ is a NBAC-colouring.
	
	Suppose not all loops are the same colour.
	As $G$ is connected,
	there must exist distinct vertices $v,w \in V(G)$ connected by an edge $(v,w,\gamma)$ where $\delta(v,v,\alpha) \neq \delta(w,w,\alpha)$.
	Without loss of generality we may assume $\delta(v,v,\alpha) = \delta(v,w,\gamma)$.
	The circuit
	\begin{align*}
		( (v,w,\gamma) ,(w,w,\alpha), (w,v,-\gamma), (v,v,-\alpha) )
	\end{align*}
	is a balanced almost monochromatic circuit,
	contradicting that $\delta$ is a NBAC-colouring.
\end{proof}
	
\begin{lemma}\label{lem:loop2per.3}
	Let $G$ be a connected $\mathbb{Z}^2$-gain graph.
	Suppose that there exists some $\alpha \in \mathbb{Z}^2\setminus \{(0,0)\}$
	such that for every vertex $v \in V(G)$ and $c \in \mathbb{N}$, we have $(v,v,c \alpha) \in E(G)$.
	Then there are no type 1 or type 3 flexible 2-lattice NBAC-colourings of $G$.
\end{lemma}

\begin{proof}
	Suppose there exists a NBAC-colouring $\delta$ of $G$ that is either a type 1 or type 3 flexible 2-lattice NBAC-colouring.
	As one of $G^\delta_{\text{red}}$ or $G^\delta_{\text{blue}}$ must contain a loop (and thus be unbalanced),
	$\delta$ must be a type 3 flexible 2-lattice NBAC-colouring.
	By Lemma \ref{lem:loop2per.3},
	every loop of the form $(v,v,c \gamma)$ has the same colour, and without loss of generality we shall assume they are all red.
	We note immediately that there cannot be any unbalanced blue circuits;
	indeed, if $C$ was a blue circuit containing $v$ with gain $c \alpha$, 
	then the circuit formed by $C$ followed by $(v,v,-c\alpha)$ will be a balanced almost blue circuit.
	However this implies $\rank (G^\delta_{\text{blue}})=0$,
	contradicting that $\delta$ is a type 3 flexible 2-lattice NBAC-colouring.
\end{proof}

\begin{lemma}\label{lem:1loop2per}
	Let $G$ be a connected $\mathbb{Z}^2$-gain graph with a loop.
	Then there are no active type 1 or type 3 flexible 2-lattice NBAC-colourings of $G$.
\end{lemma}

\begin{proof}
	Let $(w,w,\alpha)$ be a loop of $G$.
	By Lemma \ref{lem:loop2per.1}, 
	we may assume that for every vertex $v \in V(G)$ and $c \in \mathbb{N}$, 
	we have $(v,v,c \alpha) \in E(G)$.
	The result now follows from Lemma \ref{lem:loop2per.3}.
\end{proof}

We may now prove a special case of Conjecture \ref{con:type3}.

\begin{theorem}\label{thm:1loop2per}
	Let $G$ be a connected $\mathbb{Z}^2$-gain graph with a loop.
	Then there exists a full placement-lattice $(p,L)$ of $G$ in $\mathbb{R}^2$ such that
	$(G,p,L)$ is a flexible full $2$-periodic framework if and only if either:
	\begin{enumerate}[(i)]
		\item $G$ has a type 2 flexible 2-lattice NBAC-colouring,
		\item $G$ has a fixed lattice NBAC-colouring,
		\item $\rank (G) =1$.
	\end{enumerate}	
\end{theorem}

\begin{proof}
	Suppose there exists a full placement-lattice $(p,L)$ of $G$ in $\mathbb{R}^2$ such that
	$(G,p,L)$ is a flexible full $2$-periodic framework.
	Since $G$ contains a loop then $\rank (G) \geq 1$.
	By Lemma \ref{lem:2NBAC} and Lemma \ref{lem:1loop2per},
	either $G$ has an active type 2 flexible 2-lattice NBAC-colouring,
	$G$ has an active fixed lattice NBAC-colouring,
	or $\rank (G) =1$.
	
	Now suppose that either $G$ has a type 2 flexible 2-lattice NBAC-colouring,
	$G$ has a fixed lattice NBAC-colouring,
	or $\rank (G) =1$.
	Then by Lemma \ref{lem:type2}, Lemma \ref{lem:fixedNBAC3} or Lemma \ref{lem:rank2},
	there exists a full placement-lattice $(p,L)$ of $G$ in $\mathbb{R}^2$ such that
	$(G,p,L)$ is a flexible full $2$-periodic framework.
\end{proof}

\subsection{Scissor flexes}

We now define a special class of flex.

\begin{definition}
	Let $(p_t,L_t)$ a flex of a $2$-periodic framework $(G,p,L)$ in $\mathbb{R}^2$.
	If there exists linearly independent $\alpha,\beta \in \mathbb{Z}^2$ such that $\|L_t.\alpha\|^2$ and $\|L_t.\beta\|^2$ are constant
	but $(L_t.\alpha).(L_t.\beta)$ is not constant,
	then $(p_t,L_t)$ is a \emph{scissor flex}.
\end{definition}

If $\rank(G) <2$,
then it can be seen that some placement-lattice of $G$ will have a scissor flex.
We shall show in Theorem \ref{cor:loopNBAC} that we can characterise the $\mathbb{Z}^2$-gain graphs with scissor flexes by their NBAC-colouring.
We first prove the following lemmas.

\begin{lemma}\label{lem:2loop2per}
	Let $G$ be a connected $\mathbb{Z}^2$-gain graph 
	and $\alpha,\beta \in \mathbb{Z}^2$ be linearly independent.
	Suppose that $(v,v,c \alpha),(v,v,c \beta) \in E(G)$ for all $v \in V(G)$ and $c \in \mathbb{N}$.
	If $\delta$ is a NBAC-colouring of $G$ then either:
	\begin{enumerate}[(i)]
		\item \label{lem:2loop2peritem1} All loops of $G$ are the same colour.
		\item \label{lem:2loop2peritem2} All loops of $G$ with gain in $\mathbb{Z} \alpha$ are red (respectively, blue),
		all loops of $G$ with gain in $\mathbb{Z} \beta$ are blue (respectively, red),
		and all loops of $G$ have gain in $\mathbb{Z} \alpha \cup \mathbb{Z} \beta$.
	\end{enumerate}
\end{lemma}

\begin{proof}
	By Lemma \ref{lem:loop2per.2} we have (without loss of generality) two possibilities:
	\begin{enumerate}
		\item \label{lem:2loop2perop1} All loops with gain in $\mathbb{Z} \alpha \cup \mathbb{Z} \beta$ are red.
		\item \label{lem:2loop2perop2} All loops with gain in $\mathbb{Z} \alpha$ are red and
		all loops with gain in $\mathbb{Z} \beta$ are blue.
	\end{enumerate}
	
	Suppose (\ref{lem:2loop2perop1}) holds.
	If $G$ only has loops with gain $\gamma \in \mathbb{Z} \alpha \cup \mathbb{Z} \beta$,
	then (\ref{lem:2loop2peritem1}) holds.
	Suppose $G$ has a loop $l :=(v,v,\gamma)$ with $\gamma \notin \mathbb{Z} \alpha \cup \mathbb{Z} \beta$.
	Let $a,b \in \mathbb{Z} \setminus \{0\}$ be any pair where $c\gamma = a \alpha + b \beta$ for some $c >0$.
	If $\delta(l) = \text{blue}$ then we note that the circuit 
	\begin{align*}
		((v,v,-a \alpha),(v,v,-a \beta),\overbrace{ l,  \ldots ,  l}^{c \text{ times}})
	\end{align*}
	is a balanced almost red circuit,
	contradicting that $\delta$ is a NBAC-colouring.
	Hence $\delta(l) = \text{red}$ and (\ref{lem:2loop2peritem1}) holds.
	
	Suppose (\ref{lem:2loop2perop2})
	holds but $G$ has a loop $l :=(v,v,\gamma)$ with $\gamma \notin \mathbb{Z} \alpha \cup \mathbb{Z} \beta$.
	Choose $a,b \in \mathbb{Z}$ such that $c\gamma = a \alpha + b \beta$ for some $c>0$.
	If $\delta(l) = \text{blue}$ then the circuit
	\begin{align*}
		C:= ( (v,v,-a \alpha),(v,v,-a \beta),\overbrace{ l,  \ldots ,  l}^{c \text{ times}})
	\end{align*}
	is a balanced almost blue circuit,
	while if $\delta(l) = \text{red}$ then $C$
	is a balanced almost red circuit.
	As both possibilities contradict that $\delta$ is a NBAC-colouring,
	then no such loop may exist	and (\ref{lem:2loop2peritem2}) holds.	
\end{proof}

\begin{lemma}\label{thm:loopNBAC}
	Let	$G$ be a connected $\mathbb{Z}^2$-gain graph with loops $l_\alpha := (v,v,\alpha)$, $l_\beta := (v,v,\beta)$,
	where $\alpha$ and $\beta$ are linearly independent.
	Then all active NBAC-colourings of $G$ with $\delta(l_\alpha)\neq \delta(l_\beta)$ are type 2 flexible 2-lattice NBAC-colourings.
\end{lemma}

\begin{proof}
	Let $\delta$ be an active NBAC-colouring of $G$ with $\delta(l_\alpha)\neq \delta(l_\beta)$.
	Without loss of generality,
	we may assume $\delta(l_\alpha) = \text{red}$ and $\delta(l_\beta) = \text{blue}$.
	By Lemma \ref{lem:loop2per.1}, 
	we may assume that $(v,v,c \gamma) \in E(G)$ for all $v \in V(G)$ and $c \in \mathbb{N}$,
	where $\gamma \in \{\alpha,\beta\}$.	
	By Lemma \ref{lem:2loop2per},
	all loops with gain in $\mathbb{Z} \alpha$ are red,
	all loops with gain in $\mathbb{Z} \beta$ are blue,
	and there are no loops with gain $\gamma \notin \mathbb{Z} \alpha \cup \mathbb{Z} \beta$.
	
	Suppose there exists a red circuit $C$ containing $v$ with $\psi(C) = a \alpha + b \beta$.
	If $a,b\neq 0$, then the circuit formed from $C$ followed by $(v,v,-a \alpha),(v,v,-a \beta)$ is a balanced almost red circuit,
	while if $a = 0, b\neq 0$, then the circuit formed from $C$ followed by $(v,v,-a \beta)$ is a balanced almost red circuit.
	As both contradict that $\delta$ is a NBAC-colouring of $G$,
	then $\psi(C) \in \mathbb{Z} \alpha$.
	We similarly note that for any blue circuit $C'$, $\psi(C') \in \mathbb{Z} \beta$.
	
	Let $C$ be an almost monochromatic circuit of length $n$ where $\delta(e_n) \neq \delta(e_i)$ for all $i \in \{1,\ldots,n-1\}$.
	If $C$ is almost red and $\psi(C) = c \alpha$ for some $c \in \mathbb{Z}$,
	then
	\begin{align*}
		( e_1, \ldots,e_n, (v_1,v_1,-c \alpha))
	\end{align*}
	is a balanced almost red circuit,
	contradicting that $\delta$ is a NBAC-colouring.
	If $C$ is almost blue and $\psi(C) = c \beta$ for some $c \in \mathbb{Z}$,
	then
	\begin{align*}
		( e_1, \ldots,e_n, (v_1,v_1,-c \beta))
	\end{align*}
	is a balanced almost blue circuit,
	contradicting that $\delta$ is a NBAC-colouring.
	It now follows that $\delta$ is a type 2 flexible 2-lattice as required.
\end{proof}

\begin{lemma}\label{lem:linindloop}
	Let $G$ be a connected $\mathbb{Z}^2$-gain graph with loops $l_\alpha := (v,v,\alpha)$, $l_\beta := (v,v,\beta)$,
	where $\alpha$ and $\beta$ are linearly independent.
	If $\delta(l_\alpha) = \delta(l_\beta)$ for all active NBAC-colourings of $G$,
	then for any $2$-periodic framework $(G,p,L)$,
	every non-trivial flex of $(G,p,L)$ is a fixed-lattice flex.
\end{lemma}

\begin{proof}
	Let $(G,p,L)$ be a flexible $2$-periodic framework with a flex $(p_t,L_t)$, $t \in [0,1]$.
	Since we have loops $l_\alpha, l_\beta \in E(G)$,
	it follows that $\|L_t.\alpha\|^2 =\|L.\alpha\|^2 $ and $\|L_t.\beta\|^2 =\|L.\beta\|^2 $ for all $t \in [0,1]$.
	For each $t \in [0,1]$ we have
	\begin{eqnarray*}
		(L.\alpha).(L.\beta) &=& (p(v) - p(v) - L.\alpha).(p(v) - p(v) - L.\beta) \\
		(L_t.\alpha).(L_t.\beta) &=& (p_t(v) - p_t(v) - L_t.\alpha).(p_t(v) - p_t(v) - L_t.\beta).
	\end{eqnarray*}
	By Proposition \ref{prop:angleactive} and the continuity of the flex $(p_t,L_t)$,
	we must have
	\begin{align*}
		(L.\alpha).(L.\beta) = (L_t.\alpha).(L_t.\beta)
	\end{align*}
	for all $t \in [0,1]$.
	It now follows that $L_t \sim L$ for all $t \in [0,1]$ as required.
\end{proof}

\begin{theorem}\label{cor:loopNBAC}
	Let $G$ be a connected $\mathbb{Z}^2$-gain graph with $\rank (G) =2$.
	Then there exists a full $2$-periodic framework $(G,p,L)$ in $\mathbb{R}^2$ with a scissor flex if and only if either:
	\begin{enumerate}[(i)]
		\item \label{cor:loopNBACitem1} $G$ has a type 2 flexible 2-lattice NBAC-colouring, or
		\item \label{cor:loopNBACitem2} $G$ has a type 1 flexible 2-lattice NBAC-colouring where, for some linearly independent pair $\alpha,\beta \in \mathbb{Z}^2$,
		there are no almost red circuits of $G$ with gain in $\mathbb{Z} \alpha$ and there are no almost blue circuits of $G$ with gain in $\mathbb{Z} \beta$.
	\end{enumerate}
\end{theorem}

\begin{proof}
	If $G$ has a type 2 flexible 2-lattice NBAC-colouring then there exists a full $2$-periodic framework $(G,p,L)$ with a scissor flex by Lemma \ref{lem:type2}.
	Suppose $G$ has a type 1 flexible 2-lattice NBAC-colouring $\delta$ with no almost red circuits with gain in $\mathbb{Z} \alpha$ and no almost blue circuits with gain in $\mathbb{Z} \beta$.
	We note that if we add the loops $(v,v,\alpha)$ and $(v,v,\beta)$ to $G$ to form the graph $G'$,
	then we can extend $\delta$ to a type 2 flexible 2-lattice NBAC-colouring $\delta'$ of $G'$ by setting $\delta'(v,v,\alpha)=\text{red}$ and $\delta'(v,v,\beta)=\text{blue}$.
	A full $2$-periodic framework $(G',p,L)$ with a scissor flex can now be constructed by Lemma \ref{lem:type2}.
	We finish by noting that $(G,p,L)$ will also have a scissor flex as required.
	
	Now suppose there exists a full $2$-periodic framework $(G,p,L)$ with a scissor flex $(p_t,L_t)$;
	we shall assume that $\alpha,\beta \in \mathbb{Z}^2$ are linearly independent gains where $\|L_t.\alpha\|^2$ and $\|L_t.\beta\|^2$ are both constant.
	Choose any $v \in V(G)$ and define the $G'$ to be the $\mathbb{Z}^2$-gain graph formed from $G$ by adding the loops $l_\alpha := (v,v,\alpha)$ and $l_\beta:=(v,v,\beta)$.
	We note that $(p_t,L_t)$ is a flex of $(G',p,L)$ also,
	hence as $\rank(G')=2$,
	we have that $G'$ has an active NBAC-colouring by Lemma \ref{lem:2NBAC}.
	
	If all active NBAC-colourings $\delta'$ of $G'$ have $\delta'(l_\alpha) = \delta'(l_\beta)$,
	then by Lemma \ref{lem:linindloop},
	$(G',p,L)$ is fixed lattice flexible,
	a contradiction.
	It follows that $G'$ has an active NBAC-colouring $\delta'$ with $\delta'(l_\alpha) \neq \delta'(l_\beta)$.
	By Lemma \ref{thm:loopNBAC},
	$\delta'$ is a type 2 flexible 2-lattice NBAC-colouring.
	If we define $\delta$ to be the restriction of $\delta'$ to $G$,
	then $\delta$ is a type $k$ flexible 2-lattice NBAC-colouring for some $k\in \{1,2\}$,
	as $\rank(G)=2$ and $\delta'$ cannot be monochromatic on any subgraph of rank 2.
	We finish by noting that if $\delta$ is a type 1 flexible 2-lattice NBAC-colouring,
	then (with respect to $\delta$) there are no almost red circuits of $G$ with gain in $\mathbb{Z} \alpha$ and there are no almost blue circuits of $G$ with gain in $\mathbb{Z} \beta$.
\end{proof}

\section*{Acknowledgements}
I would like to thank the anonymous referees for their valuable corrections and comments,
all of which helped greatly improve the paper.

\end{document}